\documentclass[11pt]{amsart}
\usepackage{amsmath,amssymb,amscd,amsfonts,latexsym,epsfig,psfrag,graphicx,url,color,xr,hyperref,mathtools}
\usepackage{fullpage}
\usepackage[all]{xy}
\newtheorem{theorem}{Theorem}[section]
\newtheorem{corollary}[theorem]{Corollary}

\newtheorem{proposition}[theorem]{Proposition}
\newtheorem{lemma}[theorem]{Lemma} 
\newtheorem{sublemma}[theorem]{Sublemma}

\newtheorem*{step1}{Step 1}
\newtheorem*{step1a}{Step 1a}
\newtheorem*{step1b}{Step 1b}

\newtheorem*{step2}{Step 2}
\newtheorem*{step2a}{Step 2a}
\newtheorem*{step2b}{Step 2b}
\newtheorem*{step3}{Step 3}

\theoremstyle{definition}
\newtheorem{definition}[theorem]{Definition}
\theoremstyle{remark}

\newcommand\cC{\mathcal{C}}
\newcommand\D{\mathcal{D}}

\newcommand{\J}{\mathcal{J}}

\newcommand{\N}{\mathbb{N}}
\newcommand{\R}{\mathbb{R}}
\renewcommand{\H}{\mathbb{H}}
\newcommand{\C}{\mathbb{C}}
\newcommand{\Z}{\mathbb{Z}}

\newcommand{\on}{\operatorname}

\newcommand\dbar{\ol\partial}

\newcommand\supp{\on{supp}}
\newcommand\loc{{\on{loc}}}
\newcommand\euc{{\on{euc}}}

\newcommand{\im}{\on{im}}
\newcommand{\diam}{\on{diam}}

\newcommand\qu{/\kern-.7ex/} 
\newcommand\lqu{\backslash \kern-.7ex \backslash}

\def\tint{{\textstyle\int}}

\newcommand{\ol}{\overline}
\newcommand{\ul}{\underline}

\newcommand{\sr}{\stackrel}
\newcommand{\wh}{\widehat}
\newcommand{\wt}{\widetilde}

\newcommand{\lan}{\langle}
\newcommand{\ran}{\rangle}

\def\d{\delta}

\newcommand{\eps}{\epsilon}

\def\om{\omega}

\def\rD{{\rm D}}
\def\rT{{\rm T}}
\def\rd{{\rm d}}

\renewcommand{\d}{{\rm d}}

\def\bi{\mathbf{i}}
\def\bj{\mathbf{j}}
\def\bJ{\mathbf{J}}
\def\bQ{\mathbf{Q}}

\def\dt{{\rm d}t}

\newcommand{\less}{{\smallsetminus}}

\newcounter{qcounter}

\newenvironment{enumlist}
   {
      \begin{list}
         {\bf\Alph{qcounter})} 
         {
         \usecounter{qcounter}
                     \setlength{\itemsep}{.5ex}
            \setlength{\leftmargin}{1em}
         }
   }
   {
      \end{list}
   }
   
\newcommand\quotient[2]{
        \mathchoice
            {% \displaystyle
                \text{\raise1ex\hbox{$#1$}\Big/\lower1ex\hbox{$#2$}}%
            }
            {% \textstyle
                #1\,/\,#2
            }
            {% \scriptstyle
                #1\,/\,#2
            }
            {% \scriptscriptstyle
                #1\,/\,#2
            }
    }

\newcommand\quoti[2]{
                \text{\raise1ex\hbox{$#1$}/\lower1ex\hbox{$\scriptstyle#2$}}
  }

\newcommand\quot[2]{
                \text{\raise1ex\hbox{$#1\!\!$}/\lower1ex\hbox{$\!\scriptstyle#2$}}
  }

\newcommand\quo[2]{
                \text{\raise.8ex\hbox{$\scriptstyle#1\!$}/\lower.8ex\hbox{$\!\scriptstyle#2$}}
  }

\newcommand\qq[2]{
                \text{\raise.8ex\hbox{$#1\!$}/\lower.8ex\hbox{$#2$}}
}

\begin{document}

\title{Pseudoholomorphic quilts with figure eight singularity}
\author{Nathaniel Bottman}
\address{School of Mathematics, Institute for Advanced Study,
1 Einstein Dr, Princeton, NJ 08540}
\email{\href{mailto:nbottman@math.ias.edu}{nbottman@math.ias.edu}}

\begin{abstract}  
We show that the novel figure eight singularity in a pseudoholomorphic quilt can be continuously removed when composition of Lagrangian correspondences is cleanly immersed.  The proof of this result requires a collection of width-independent elliptic estimates that allow for nonstandard complex structures on the domain.
\end{abstract} 

\maketitle

\section{Introduction}

%\comment{should consider calling degenerate strip quilts ``\(u\)''

%be clearer about regularity of AC structures --- and I don't think \(\cC^k_\loc\) makes sense for the AC structures}

We consider compact
%N
%\comment{closed??}
Lagrangian correspondences \({L_{01} \subset M_0^- \times M_1}\) and \(L_{12} \subset M_1^- \times M_2\), where \(M_0,M_1,M_2\) are closed symplectic manifolds,
and where \(M_i^- := (M_i, -\om_{M_i})\).
The {\bf geometric composition} of the Lagrangian correspondences is 
\(L_{01} \circ L_{12} := \pi_{02}( L_{01} \times_{M_1} L_{12}) \), the image under the projection \( \pi_{02}\colon M_0^- \times
M_1 \times M_1^- \times M_2 \to M_0^- \times M_2 \) of the fiber product \begin{align*} 
L_{01} \times_{M_1} L_{12} := (L_{01} \times L_{12})
 \cap (M_0^- \times \Delta_{M_1} \times M_2) .
\end{align*}
Here \(\Delta_{M_1} \subset M_1 \times M_1^-\) is the diagonal.
If \(L_{01}\times L_{12}\) intersects \(M_0^- \times \Delta_1 \times M_2\) transversely
then \(\pi_{02}\colon L_{01} \times_{M_1} L_{12} \to M_0^-\times M_2\) is a Lagrangian immersion (see \cite{gu:rev,quiltfloer}), in which case we call \(L_{01}\circ L_{12}\) an {\bf immersed composition}.
In the case of {\bf embedded composition}, where the projection is injective and hence a Lagrangian embedding, monotonicity and Maslov index assumptions allowed Wehrheim--Woodward \cite{isom} to establish an isomorphism of quilted Floer cohomologies  (as defined in \cite{quiltfloer})
\begin{equation} \label{eq:HFiso} 
HF(\ldots, L_{01},L_{12}, \ldots ) \cong HF(\ldots , L_{01} \circ L_{12}, \ldots) .
\end{equation}
The analytic core of the proof was a {\bf{strip-shrinking degeneration}}, in which a triple of pseudoholomorphic strips coupled by Lagrangian seam conditions degenerates to a pair of strips, via the width of the middle strip shrinking to zero.  The monotonicity and embeddedness assumptions allowed for an implicit exclusion of all bubbling, which was conjectured to include a novel {\bf{figure eight bubbling}} that (unlike disk or sphere bubbling) could be an algebraic obstruction to \eqref{eq:HFiso}.

\subsection{Removal of singularity}

\begin{figure}
\centering
\def\svgwidth{\columnwidth}
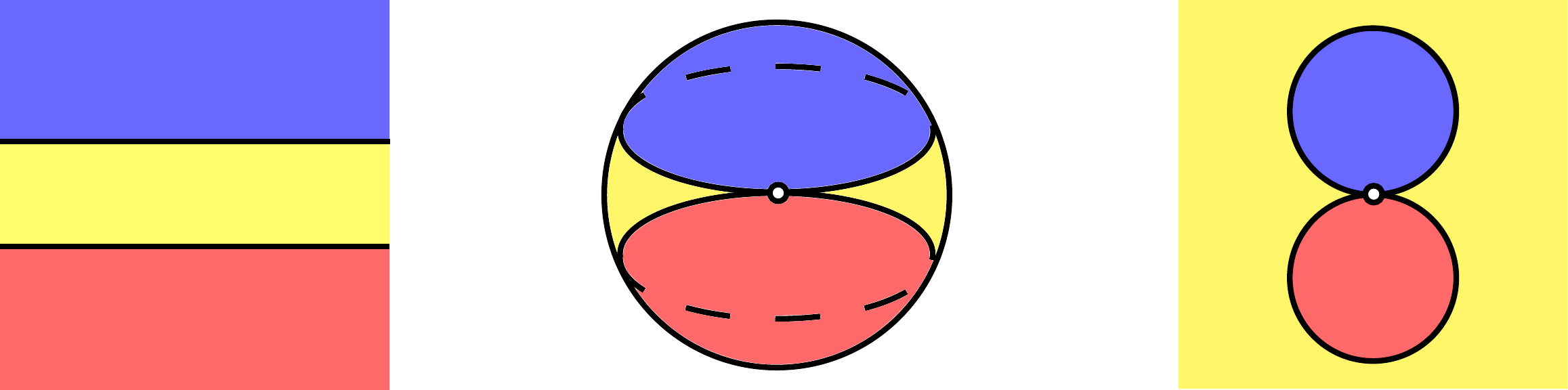
\caption{The left figure illustrates a figure eight bubble, the middle figure illustrates its reparametrization as a pseudoholomorphic quilt whose domain is the punctured sphere, and the right figure illustrates an inverted figure eight (defined in \S\ref{sec:remsing}, and equivalent to the left figure via \(z \mapsto -1/z\)).  The domain of the left and right figures is \(\C\), and the point at infinity in the left figure corresponds to the punctures in the middle and right figures.}
\label{fig:viewsof8}
\end{figure}

The current author and Katrin Wehrheim prove in \cite{bw:compactness} that a blowup of the gradient in a sequence of pseudoholomorphic quilts with an annulus or strip of shrinking width gives rise to one of the standard bubbling phenomena (pseudoholomorphic spheres and disks) or a nontrivial {\bf figure eight bubble}, as depicted in Figure~\ref{fig:viewsof8}. The latter is a tuple of finite energy pseudoholomorphic maps
\begin{equation} \label{eq:8}
w_0\colon \R\times(-\infty,-\tfrac 1 2]\to M_0, \qquad
w_1\colon \R\times[-\tfrac 1 2, \tfrac 1 2]\to M_1, \qquad
w_2\colon \R\times[\tfrac 1 2, \infty)\to M_2
\end{equation}
satisfying the {\bf seam conditions}
\begin{equation*}
(w_0(s,-\tfrac 1 2),w_1(s,-\tfrac 1 2))\in L_{01} , \quad
(w_1(s,\tfrac 1 2),w_2(s,\tfrac 1 2))\in L_{12} 
\qquad\forall \: s\in \R .
\end{equation*}
In the current paper we apply this Gromov Compactness Theorem to show that the figure eight singularity can be removed, as \cite{isom} conjectured:

\medskip

\noindent{\bf Removal of Singularity Theorem~\ref{thm:remsing}:} 
If the composition \(L_{01} \circ L_{12}\) is {\bf cleanly immersed} (immersed, and in addition the local branches of \(L_{01} \circ L_{12}\) intersect one another cleanly), then \(w_0\) resp.\ \(w_2\) extend to continuous maps on \(D^2 \cong (\R \times (-\infty, 0]) \cup \{\infty\}\) resp.\ \(D^2 \cong (\R \times [0, \infty)) \cup \{\infty\}\), and \(w_1(s, -)\) converges to constant paths as $s\to\pm\infty$. If \(L_{01} \circ L_{12}\) is embedded, then the latter limits are equal.

\medskip

\noindent This theorem is the first step in the program outlined in \cite{b:thesis}, which proposes a collection of composition operations amongst Fukaya categories of different symplectic manifolds.

\medskip

\noindent In support of \cite{b:thesis}, Appendix~\ref{app:remsing} also proves the analogous removal of singularities for pseudoholomorphic disks with boundary values in  an immersed Lagrangian with clean self-intersections.
These results are not necessarily new, see Appendix~\ref{app:remsing} for citations, but provided for the sake of completeness. 
It is also conceptually useful to recast the (possibly singular) disk bubbles with boundary on \(L_{01}\circ L_{12}\) as {\bf squashed eight bubbles}, that is as triples of finite energy pseudoholomorphic maps \begin{align*}
w_0\colon \R\times(-\infty,0]\to M_0, \qquad w_1\colon \R \to M_1, \qquad w_2\colon \R\times[0, \infty)\to M_2
\end{align*}
satisfying the generalized seam condition
\begin{equation*}
\bigl(w_0(s,0),w_1(s), w_1(s),w_2(s,0)\bigr)\in L_{01} \times_{M_1} L_{12} \qquad \forall \: s\in \R.
\end{equation*}

\subsection{Uniform elliptic estimates for varying widths and complex structures}
There is a further logical dependence between \cite{bw:compactness} and the current paper: In Lemma~\ref{lem:masterestimate} we substantially strengthen the strip-shrinking estimates in \cite{isom} --- in particular, from embedded to immersed geometric composition.
These strengthened estimates form the analytic core of Theorem~\ref{thm:nonfoldedstripshrink}, which is used to prove a Gromov Compactness Theorem in \cite{bw:compactness}.
One of the ingredients in Lemma~\ref{lem:masterestimate} is a special connection that allows us to obtain estimates without boundary terms for quilted Cauchy--Riemann operators, with uniform constants for all small widths of a strip.
This allows us to strengthen the uniform \(H^2\cap W^{1,4}\) estimates established in \cite{isom} to \(H^3\) and thus \(\cC^1\), which is e.g.\ needed to deduce nontriviality of bubbles with generalized boundary condition in \(L_{01}\circ L_{12}\).

Our estimates allow for nonstandard complex structures on the shrinking strip.
This is necessitated by the following analytic formulation for the figure eight singularity:
In cylindrical coordinates for a neighborhood of infinity in \eqref{eq:8}, 
the two seams become two pairs of curves approaching each other asymptotically (see the right figure in Figure~\ref{fig:viewsof8}).
%N
%\comment{clearer way to say this?}
On finite cylinders, the standard complex structure on this quilted surface can be pulled back to a quilted surface in which the width of the strips is constant and the complex structures are nonstandard, but converge in \(\cC^0\) and stay within a controlled \(\cC^k\)-distance from the standard structure for any \(k \geq 1\).

\medskip

\noindent The hypothesis that \(M_0,M_1,M_2\) are closed is not essential: As explained in \cite{bw:compactness}, it is enough for the symplectic manifolds to be geometrically bounded and to have a priori \(\cC^0\)-bounds on the various pseudoholomorphic curves.
In a future paper we will treat the noncompact setting in a systematic fashion.

\subsection{Acknowledgements}

The author is grateful to his former PhD advisor, Katrin Wehrheim, for suggesting in early 2012 that he study figure eight bubbles, and for generously sharing her knowledge throughout this project.
Casim Abbas and Helmut Hofer shared their approach to a removal-of-singularity result in their unpublished book \cite{ah}, which led to a crucial part of the argument in \S\ref{sec:remsing}.
The author acknowledges support from an NSF Graduate Research Fellowship and a Davidson Fellowship, and would like to thank the Institute for Advanced Study, Princeton University, and the University of California, Berkeley for their hospitality.

\section{Removal of singularity for the figure eight bubble} \label{sec:remsing}

%\comment{box the hypotheses? well, done below\ldots}

In this section and the next we will be working with symplectic manifolds \(M_0, M_1, M_2\), almost complex structures \(J_0, J_1, J_2\), and pseudoholomorphic curves with seam conditions defined by compact Lagrangian correspondences \begin{align} \label{eq:lagcorr}
L_{01} \subset M_0^- \times M_1, \qquad L_{12} \subset M_1^- \times M_2,
\end{align}
with \(L_{01} \circ L_{12}\) either immersed or cleanly immersed:
\begin{itemize}
\item \(L_{01}\) and \(L_{12}\) have {\bf immersed composition} if the intersection \begin{align*}
L_{01} \times_{M_1} L_{12} = (L_{01} \times L_{12}) \cap (M_0 \times \Delta_{M_1} \times M_2)
\end{align*}
is transverse.
This implies that \(\pi_{02}\colon L_{01} \times_{M_1} L_{12} \to M_0^- \times M_2\) is a Lagrangian immersion, e.g.\ by \cite[Lemma~2.0.5]{quiltfloer}, and in this situation we will denote the image by \(L_{01} \circ L_{12} := \pi_{02}( L_{01} \times_{M_1} L_{12} )\).

\item If \(L_{01}, L_{12}\) have immersed composition and furthermore any two local branches of \(L_{01} \circ L_{12}\) intersect cleanly --- i.e.\ at any intersection of two local branches there is a chart for \(M_0^- \times M_2\) (as a smooth manifold) in which each of those two branches is identified with an open subset of a vector subspace of \(\R^n\) --- then the composition \(L_{01} \circ L_{12}\) is {\bf cleanly immersed}.
\end{itemize}

\medskip

\noindent The purpose of \S\ref{sec:remsing} is to prove a removal of singularity theorem for inverted figure eight bubbles.
\begin{definition} \label{def:inverted8}
An {\bf inverted figure eight bubble between \(\mathbf{L_{01}}\) and \(\mathbf{L_{12}}\)} is  a triple of smooth maps \begin{align*}
 \ul w = \left( \begin{matrix}
 w_0\colon \ol B_1(-i) \less \{0\} \to M_0 \\
 w_1\colon \C^* \less (B_1(i) \cup B_1(-i)) \to M_1 \\
 w_2\colon \ol B_1(i) \less \{0\} \to M_2
 \end{matrix} \right)
 \end{align*}
satisfying the Cauchy--Riemann equations 
\(\partial_s w_\ell + J_\ell(w_\ell)\partial_t w_\ell = 0\) for \(\ell\in\{0,1,2\}\)
and the seam conditions 
\begin{gather*}
(w_0(-i + e^{i\theta}), w_1(-i + e^{i\theta})) \in L_{01} \quad \forall \: \theta \neq \tfrac \pi 2, \qquad
(w_1(i + e^{i\theta}), w_2(i + e^{i\theta})) \in L_{12} \quad \forall \: \theta \neq \tfrac {3\pi} 2 ,
 \end{gather*}
and which have finite energy
\begin{align*}
\tint w_0^* \om_0 + \tint w_1^* \om_1 + \tint w_2^* \om_2 
\;=\; \tfrac 12 \Bigl( \tint |\d w_0|^2 +\tint |\d w_1|^2 +\tint |\d w_2|^2 \Bigr)
< \infty ,
\end{align*}
where we have endowed \(M_\ell\) with the metric \begin{align} \label{eq:metrics}
g_\ell := \om_\ell(-,J_\ell-).
\end{align}
\end{definition}

\noindent Throughout \S\ref{sec:remsing}, the norm of a tangent vector on \(M_\ell\) will always be defined using \(g_\ell\).

\medskip

\begin{center}
\parbox{0.9\columnwidth}{\it Fix for \S\ref{sec:remsing} closed symplectic manifolds \(M_0, M_1, M_2\), compatible almost complex structures \(J_\ell \in \J(M_\ell,\om_\ell), \: \ell \in \{0,1,2\}\), compact Lagrangians \(L_{01}, L_{12}\) as in \eqref{eq:lagcorr} with cleanly-immersed composition, and an inverted figure eight bubble \(\ul w\) between \(L_{01}\) and \(L_{12}\).}
\end{center}

\medskip

\noindent In fact, only the arguments in \S2.2 require the composition \(L_{01} \circ L_{12}\) to be cleanly immersed, rather than just immersed, but we assume the stronger hypothesis throughout \S\ref{sec:remsing} for cohesiveness.

\medskip

The following theorem says that the singularity at 0 of a figure eight bubble can be continuously removed, under the hypothesis of cleanly-immersed composition.

\begin{theorem} \label{thm:remsing}
The maps \(w_0, w_2\) continuously extend to 0,
and the limits \(\lim_{z \to 0, \: {\rm Re}(z) > 0} w_1(z)\) and \(\lim_{z \to 0, \: {\rm Re}(z) < 0} w_1(z)\) both exist.
If moreover the immersion \(\pi_{02}\colon L_{01} \times_{M_1} L_{12} \to M_0^- \times M_2\) is an embedding, then the latter limits are equal
so that \(w_1\) also extends continuously to \(0\).
\end{theorem}

\noindent The proof of this theorem draws on the removal of singularity strategies in \cite[\S7.3]{ah} and in \cite[\S4.5]{ms:jh}.
First, we follow \cite{ah} and establish a uniform gradient bound in cylindrical coordinates near the puncture (Lemma~\ref{lem:cylgradbound}), which we use to show that the lengths of the paths \(\theta \mapsto w_\ell(\epsilon e^{i\theta})\) converge to zero as \(\epsilon \to 0\) (Lemma~\ref{lem:lengthsgotozero}).
The substantial modification to the argument of \cite{ah} is that we must use the Gromov Compactness Theorem \cite{bw:compactness} in order to prove uniform gradient bounds in Lemma~\ref{lem:cylgradbound}.
Once we have proven that lengths go to zero, we follow \cite{ms:jh} and prove an isoperimetric inequality for the energy (Lemma~\ref{lem:isoineq}), which we use to show that the energy on 
disks around the puncture decays exponentially with respect to the logarithm of the radius.
Here the nontrivial modification is in the quilted nature of our isoperimetric inequality.
Finally, an argument from \cite{ah} allows us to conclude that \(w_0\) and \(w_2\) extend continuously to the puncture.
The continuous extension of \(w_1\) follows from the gradient bound in cylindrical coordinates and the immersed composition of \(L_{01}\) and \(L_{12}\).
The formal proof of Theorem~\ref{thm:remsing} is given in \S\ref{ss:isoineq}.

\subsection{Lengths tend to zero} \label{ss:lengths}

The first step toward the Removal of Singularity Theorem~\ref{thm:remsing} is to show that the lengths of the paths \(\theta \mapsto w_\ell(\epsilon e^{i\theta})\) converge to zero as \(\epsilon \to 0\).
This is nontrivial since the conformal structure of the quilted surface near the singularity does not allow us to apply mean value inequalities effectively, as in previous removal of singularity results for pseudoholomorphic curves.
Hence the finiteness of energy only provides a sequence \(\eps^\nu\to 0\) along which the lengths tend to zero.
This allowed Bottman--Wehrheim to deduce a weak removal of singularity in \cite{bw:compactness}, but the stronger Theorem~\ref{thm:remsing} will require the full strength of the generalized strip-shrinking analysis developed in \S\ref{sec:convergence} and the resulting Gromov Compactness Theorem in \cite{bw:compactness}.
We record a consequence of the latter as Corollary~\ref{cor:rescale} below.

\begin{figure}
\centering
\def\svgwidth{\columnwidth}
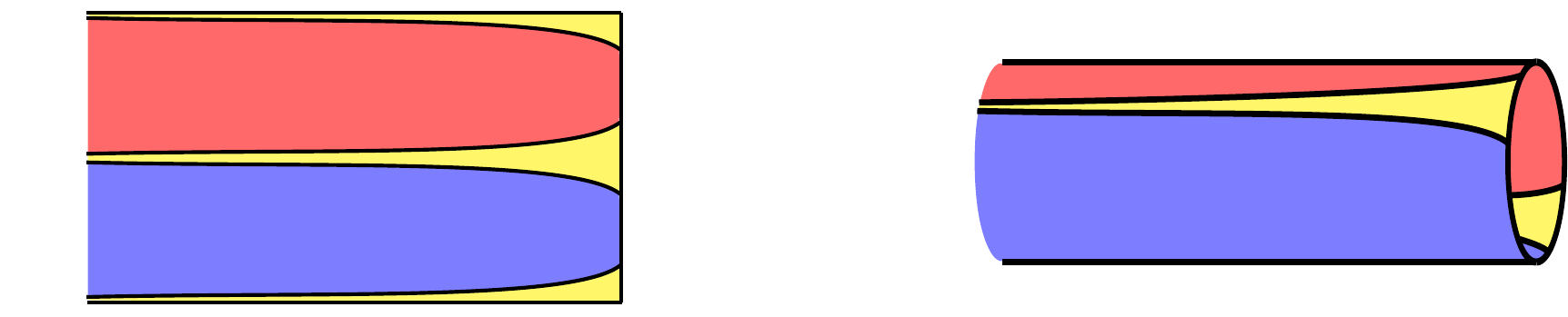
\caption{Two views of the domains $V_0,V_1,V_2 \subset (-\infty,0]\times \R/\Z$ used in \S\ref{ss:lengths}, as a half-infinite strip and cylinder, respectively.
}
\label{fig:Vi}
\end{figure}

In this subsection we will work in cylindrical coordinates centered at the singularity, hence we define the reparametrized maps 
\begin{equation}\label{eq:defexpcoord}
v_\ell(s, t) := w_\ell\bigl( e^{2\pi (s + it) } \bigr) \qquad\text{for} \quad \ell\in\{0,1,2\},
\end{equation}
whose domains \(V_0, V_1, V_2 \subset (-\infty, 0] \times \R/\Z\) are given by
\begin{gather*}
V_0 := \bigl \{ (s,t) \: \left| \: s \leq 0, \: |t - \tfrac 3 4| \leq \tfrac 1 4 - \theta(s) \right. \bigr \}, 
\qquad 
V_2 := \bigl \{ (s, t) \: \left| \: s \leq 0, \: |t - \tfrac 1 4| \leq \tfrac 1 4 - \theta(s) \right. \bigr \}, \\
V_1 := \bigl \{ (s,t) \: \left| \: s \leq 0, \: |t - \tfrac 1 2| \leq \theta(s) \;\lor\; |t - 1| \leq \theta(s) \right. \bigr \}, \nonumber
\end{gather*}
with
\begin{equation}\label{eq:theta}
\theta(s) := \tfrac 1 {2\pi} \arcsin\bigl( \tfrac 1 2 e^{2\pi s} \bigr).
\end{equation}
(See Fig.~\ref{fig:Vi} for an illustration of these domains.)
Now the paths \(w_\ell(\epsilon e^{i\theta})\) for fixed \(\eps \in (0,1]\) correspond to the following paths for fixed \(s = \frac{\log \epsilon}{2\pi} \leq 0\):
\begin{gather}
\gamma_s^0:= v_0(s, -) \, :\;  [\tfrac 1 2 + \theta(s), 1 - \theta(s)] \longrightarrow M_0 , \qquad
\gamma_s^2 := v_2(s, -) \, :\;  [ \theta(s), \tfrac 1 2 - \theta(s)] \longrightarrow M_2 , \nonumber \\
\gamma_s^1 := v_1(s, -) \, :\;  [\tfrac 1 2 - \theta(s), \tfrac 1 2 + \theta(s)] \cup [1 - \theta(s), 1 + \theta(s)] \longrightarrow M_1. \label{eq:gamma}
\end{gather}
The length of \(\gamma_s^\ell\) is given by the integral \(\ell(\gamma_s^\ell):=\tint |\tfrac \d {\d t} \gamma_s^\ell| \,\dt\) over the respective domain, and will be controlled by the following main result of this subsection.

\begin{lemma}\label{lem:lengthsgotozero}
The \(L^2\)-lengths of the paths \(\gamma_s^0, \gamma_s^1, \gamma_s^2\) defined in \eqref{eq:gamma} converge to zero as \(s \to -\infty\):
\begin{align*}
 \int_{1/2 + \theta(s)}^{1 - \theta(s)} |\tfrac\d{\d t} \gamma_s^0|^2 \,\d t
+ \left(\int_{1/2 - \theta(s)}^{1/2 + \theta(s)} + \int_{1 - \theta(s)}^{1 + \theta(s)}\right) |\tfrac\d{\d t} \gamma_s^1|^2 \,\d t + \int_{\theta(s)}^{1/2 - \theta(s)} |\tfrac\d{\d t} \gamma_s^2|^2 \,\d t
\;\underset{s\to-\infty}{\longrightarrow}\; 0 .
\end{align*}
In particular, the length \(\ell(\gamma_s) := \ell(\gamma_s^0) + \ell(\gamma_s^1) + \ell(\gamma_s^2)\) tends to zero as \(s \to -\infty\).
\end{lemma}

\noindent The proof of Lemma~\ref{lem:lengthsgotozero} will use ideas from \cite{ah}.
The novel difficulty --- due to the conformal structure --- is to establish the following uniform gradient bound on \(|\d \ul v|\), the upper semicontinuous function defined by \begin{align} \label{eq:defenergydensity}
| \d \ul v| \colon (-\infty,0] \times \R/\Z \to [0,\infty), \qquad |\d \ul v(s,t)|^2 := |\d v_0(s,t)|^2 + |\d v_1(s,t)|^2 + |\d v_2(s,t)|^2,
\end{align}
where the functions \(|\d v_\ell(s,t)|\) are set to zero where they are not defined.

\begin{lemma}\label{lem:cylgradbound}
The gradient \(|\d \ul v|\) defined in \eqref{eq:defenergydensity} is uniformly bounded.
\end{lemma}

\begin{figure}
\centering
\def\svgwidth{\columnwidth}
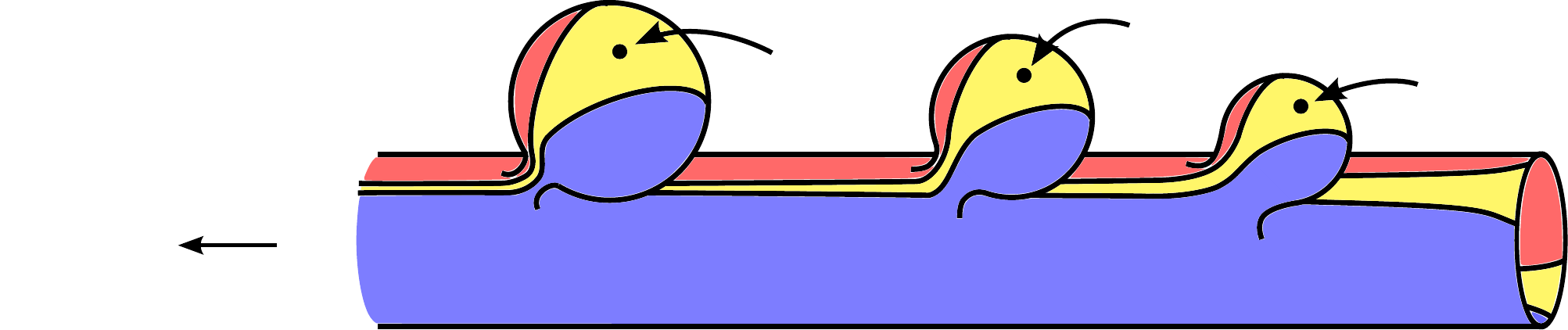
\caption{To prove Lemma~\ref{lem:cylgradbound}, we assume that the cylindrical reparametrizations \(v_\ell\) do not have uniformly bounded gradient, then bubble off a nonconstant quilted map.
In this illustration, the bubbled-off map is a figure eight bubble.
}
\label{fig:gradbound}
\end{figure}

\noindent
We will prove Lemma~\ref{lem:cylgradbound} below.
For now, we sketch the proof.
It proceeds by contradiction: if \(|\d v_\ell|\) is not bounded for some \(\ell\), then there is a sequence of points \((s^\nu,t^\nu)\) (necessarily with \(s^\nu \to -\infty\)) at which \(|\d v_\ell|\) diverges.
Rescaling at these points produces a nonconstant quilted map, as illustrated in Figure~\ref{fig:gradbound}, but this contradicts the finite-energy hypothesis on \(\ul v\).
The technical input is the Gromov Compactness Theorem in \cite{bw:compactness}, a consequence of which we record as Theorem~\ref{cor:rescale}.
This theorem is needed to deduce that the rescaled maps actually converge.
In order to state it, we need to define the domains of the maps and a controlled fashion in which the strip-width can tend to zero.

The following definition is the only instance in \S\ref{sec:remsing} where we allow the almost complex structures to be domain-dependent, so that the notion of a squiggly strip quilt is flexible enough to be used in \S\ref{sec:convergence}.
\begin{definition}
Fix \(\rho>0\), a real-analytic function \(f\colon [-\rho,\rho] \to (0, \rho/2]\), domain-dependent compatible almost complex structures \(J_\ell\colon [-\rho,\rho]^2 \to \J(M_\ell,\om_\ell)\), \(\ell \in \{0,1,2\}\), and a complex structure \(j\) on \([-\rho,\rho]^2\).
\begin{itemize}
\item A {\bf \(\mathbf{(J_0,J_1,J_2,j)}\)-holomorphic size-\(\mathbf{(f,\rho)}\) squiggly strip quilt for \(\mathbf{(L_{01},L_{12})}\)} is a triple of smooth maps
\begin{align} \label{eq:squigglymaps}
 \ul{v}=\left(
\begin{aligned}
v_0\colon \{(s,t) \in (-\rho,\rho)^2 \: | & \; t \leq -f(s) \} \to M_0 \\ 
v_1\colon \{(s,t) \in (-\rho,\rho)^2\: | &  \;  |t | \leq f(s) \} \to M_1 \\ 
v_2\colon \{(s,t) \in (-\rho,\rho)^2\: | & \; t \geq f(s) \} \to M_2 
\end{aligned}
\right) 
\end{align}
that fulfill the seam conditions
\begin{align} \label{eq:squigglyseams}
\bigl(v_0(s,-f(s)),v_1(s,-f(s))\bigr)\in L_{01}, \qquad \bigl(v_1(s,f(s)), v_2(s,f(s))\bigr)\in L_{12} \qquad \forall \: s \in (-\rho,\rho),
\end{align}
satisfy the Cauchy--Riemann equations
\begin{align} \label{eq:squigglyCR}
\d v_\ell(s,t) \circ j(s,t) - J_\ell(s,t,v_\ell(s,t)) \circ \d v_\ell(s,t) = 0 \qquad \forall\: \ell \in \{0,1,2\}
\end{align}
for \((s,t)\) in the relevant domains, 
and have finite energy
\begin{align*}
E(\ul{v}) \,:= \; \tint v_0^*\om_0 + \tint v_1^*\om_1 + \tint v_2^*\om_2  \; < \;\infty .
\end{align*}

\item A {\bf \(\mathbf{(J_0,J_2,j)}\)-holomorphic size-\(\rho\) degenerate strip quilt for \(\mathbf{L_{01}\times_{M_1}  L_{12}}\) with singularities} is a triple of smooth maps
\begin{align} \label{eq:degenmaps}
\ul{v}=\left(
\begin{aligned}
v_0\colon & \, (-\rho,\rho)\times(-\rho,0] \;\less\; S\times\{0\} \to M_0 \\ 
v_1\colon & \, (-\rho,\rho) \;\less\; S \to M_1 \\
v_2\colon & \, (-\rho,\rho)\times[0,\rho) \;\less\; S\times\{0\} \to M_2 
\end{aligned}
\right)
\end{align}
defined on the complement of a finite set \(S\subset\R\) that fulfill the lifted seam condition
\begin{align} \label{eq:degenseams}
\bigl( v_0(s,0), v_1(s), v_1(s), v_2(s,0) \bigr) \in  
L_{01}\times_{M_1}  L_{12}
\qquad\forall \: s\in(-\rho,\rho)\less S,
\end{align}
satisfy the Cauchy--Riemann equation \eqref{eq:squigglyCR} for \(\ell \in \{0,2\}\) and
\((s,t)\) in the relevant domains, and have finite energy 
\begin{align*}
E(\ul{v}) \, := \; \tint v_0^*\om_0 + \tint v_2^*\om_2 \;<\; \infty.
\end{align*}
\end{itemize}
\end{definition}
\noindent When \(j\) is the standard complex structure \(i\colon \partial_s \mapsto \partial_t, \: \partial_t \mapsto -\partial_s\), the Cauchy--Riemann equation \eqref{eq:squigglyCR} can be expressed in coordinates as: \begin{align*}
\partial_tv_\ell(s,t) - J_\ell(s,t,v_\ell(s,t)) \partial_s v_\ell(s,t) = 0.
\end{align*}

\noindent The novel hypothesis necessary for a sequence of squiggly strip quilts of widths \((f^\nu)_{\nu\in\N}\) to converge \(\cC^\infty_\loc\) away from the gradient blow-up points is that the widths ``obediently shrink to zero'':

\begin{definition} \label{def:obedience}
Fix \(\rho>0\).
A sequence \(\bigl( f^\nu \bigr)_{\nu\in\N}\) of real-analytic functions \(f^\nu\colon [-\rho,\rho]\to (0,\rho/2]\)
\textbf{obediently shrinks to zero},
\(\mathbf{f^\nu \Rightarrow 0}\),
if
\(\max_{s\in[-\rho,\rho]} f^\nu(s) \underset{\nu\to\infty}{\longrightarrow} 0\)
and
\begin{align*}
\sup_{\nu\in\N}\; \frac{\max_{s\in[-\rho,\rho]} \bigl| \tfrac{\rd^k}{\rd s^k}f^\nu(s) \bigr|}{\min_{s\in[-\rho,\rho]} f^\nu(s)} =: C_k <\infty  \qquad\forall \: k\in\N,
\end{align*}
and in addition there are holomorphic extensions \(F^\nu\colon [-\rho,\rho]^2 \to \C\) of \(f^\nu(s) = F^\nu(s,0)\) such that \((F^\nu)\) converges \(\cC^\infty\) to zero.
\end{definition}

The key to the following special case of the Gromov Compactness Theorem from \cite{bw:compactness} is a collection of width-independent elliptic estimates proven in \S\ref{sec:convergence} for the linearized Cauchy--Riemann operator.
Those elliptic estimates allow for a nonstandard domain complex structure, which is necessary in order to allow widths \(f^\nu\) that are not constant in \(s\).

\begin{corollary}[consequence of Gromov Compactness Theorem, \cite{bw:compactness}] \label{cor:rescale}
Fix \(\rho>0\), a sequence \((f^\nu\colon [-\rho,\rho] \to (0,\tfrac \rho 2])\) of real-analytic functions shrinking obediently to zero, and a sequence \((\ul{v}^\nu)_{\nu\in\N}\) of \((J_0,J_1,J_2,i)\)-holomorphic size-\((f^\nu,\rho)\) squiggly strip quilts for \((L_{01},L_{12})\) of bounded energy \(E := \sup_{\nu \in \N} E(\ul v^\nu) < \infty\).

If \((s^\nu,t^\nu) \to (s^\infty,t^\infty) \in (-\rho,\rho)^2\) is a sequence of points where the gradient blows up, i.e.\ \begin{align*}
\limsup_{\nu\to\infty} |\ul v^\nu|(s^\nu,t^\nu) = \infty,
\end{align*}
then there must be a concentration of energy \(\hbar>0\) at \((s^\infty,t^\infty)\), i.e.\ radii \(r^\nu\to 0\) such that:
\begin{align*}
\liminf_{\nu\to\infty} \tint_{B_{r^\nu}(s^\infty,t^\infty)} \tfrac 1 2|\d\ul{v}^\nu|^2 \;>\;  0.
\end{align*}
\end{corollary}

We are finally in a position to bound the gradients of the reparametrized maps \(v_\ell\) from \eqref{eq:defexpcoord}.

\begin{proof}[Proof of Lemma~\ref{lem:cylgradbound}]
We will prove the equivalent statement that the ``folded maps''
\begin{align*}
u_\ell\colon U_\ell \to M_\ell \times M_\ell^-, \qquad u_\ell(s, t) := (v_\ell(s, t), v_\ell(s, \tfrac 1 2 - t)) \quad \text{for} \quad \ell = 0,1,2
\end{align*}
have uniformly-bounded gradients, where the domains \(U_\ell\) are given by 
\begin{gather*}
U_0 := \{ (s,t) \: | \: s \leq 0, \: -\tfrac 1 4 \leq t \leq -\theta(s)\}, \qquad 
U_2 := \{ (s,t) \: | \: s \leq 0, \: \theta(s) \leq t \leq \tfrac 1 4\}, \\
U_1 := \{ (s,t) \: | \: s \leq 0, \: -\theta(s) \leq t \leq \theta(s)\}.
\end{gather*}
These maps are pseudoholomorphic with respect to the almost complex structures \(\wh J_\ell := J_\ell \oplus (-J_\ell)\) and satisfy the following boundary and seam conditions for \(s\leq 0\):
\begin{gather*}
u_0(s, -\tfrac1 4) \in \Delta_{M_0}, \qquad (u_0(s, -\theta(s) ), u_1(s,-\theta(s)) ) \in (L_{01} \times L_{01})^T, \\
\;\;\, u_2(s, \tfrac 1 4 ) \in \Delta_{M_2}, \qquad (u_1(s,\theta(s)), u_2(s,\theta(s))) \in (L_{12} \times L_{12})^T .
\end{gather*}
(Here \(\theta(s) = \tfrac 1 {2\pi}\arcsin(\tfrac 1 2e^{2\pi s})\) as in \eqref{eq:theta}, and \((L_{ij} \times L_{ij})^T\) is the image of \(L_{ij} \times L_{ij}\) under the permutation \((x_i,x_j,y_i,y_j) \mapsto (x_i,y_i,x_j,y_j)\).)
Finiteness of the energy of the inverted figure eight \(\ul w\) translates into convergence of the integral \(\lim_{S\to-\infty}\tint_{(S, 0] \times [-1/4, 1/4]} \tfrac 1 2|\d \ul u|^2 <\infty\) of the energy density
\begin{align*}
|\d \ul u|\colon (-\infty, 0] \times \left[-\tfrac 1 4, \tfrac 1 4\right] \to [0, \infty) , \qquad
|\d \ul u(s, t)|^2 := |\d u_0(s, t)|^2 + |\d u_1(s, t)|^2 + |\d u_2(s, t)|^2,
\end{align*}
where the functions \(|\d u_\ell(s,t)|\) are set to zero where they are not already defined (so \(|\d \ul u|\) is upper semi-continuous).
This convergence in particular implies
\begin{equation}\label{eq:Eto0}
\tint_{(-\infty, S] \times [-1/4, 1/4]} \tfrac 1 2|\d \ul u|^2 \;\underset{S\to-\infty}{\longrightarrow} \; 0 .
\end{equation}

Now assume for a contradiction that there exists a sequence \((s^\nu, t^\nu) \in (-\infty, 0] \times \left[-1/4, 1/4\right]\) with \(|\d \ul u(s^\nu, t^\nu)| \to \infty\).
Since the \(u_\ell\) are smooth, this is possible only for \(s^\nu\to-\infty\); passing to a further subsequence, we may in fact assume \(s^{\nu+1} \leq s^\nu - 1\) and \(s^1 \leq 1/4\).
Depending on whether \(t^\infty\) is \(\pm 1/4\) or is contained in \((-1/4,1/4)\), we derive a contradiction to \eqref{eq:Eto0}:

\medskip

\noindent \(\mathbf{t^\infty = \pm 1/4.}\) Assume \(t^\infty = -1/4\); the \(t^\infty = 1/4\) case can be treated in analogous fashion.
Define a sequence \((u_0^\nu)\) by: \begin{align*}
u_0^\nu\colon B_{1/8}(0) \cap \H \to M_0 \times M_0^-, \qquad u_0^\nu(s,t) := u_0(s + s^\nu, t - 1/4).
\end{align*}
The map \(u_0^\nu\) is \(\wh J_0\)-holomorphic and satisfies the Lagrangian boundary condition \(u_0(s,0) \in \Delta_{M_0}\) for \(s \in (-1/8,1/8)\).
Furthermore, \(|\d u_0^\nu(0, t^\nu + 1/4)| \to \infty\), \(t^\nu + 1/4 \to 0\) by assumption, and the energy of \(u_0^\nu\) is bounded by the energy of \(\ul v\), so \cite[Lemma~4.6.5]{ms:jh} implies the inequality \(\liminf_{\nu \to \infty} \tint_{B_{1/8}(0)} \tfrac 1 2|\d u_0^\nu|^2 > 0\), which contradicts \eqref{eq:Eto0}.

\medskip

\noindent \(\mathbf{t^\infty \in (-1/4,1/4).}\) Define a sequence \((u_0^\nu,u_1^\nu,u_2^\nu)\) of \((\wh J_0, \wh J_1, \wh J_2, i)\)-holomorphic size-\((\theta^\nu,\tfrac 1 4)\) squiggly strip quilts, with \begin{align*}
\theta^\nu\colon [-\tfrac 1 4,\tfrac 1 4] \to (0, \tfrac 1 8], \qquad \theta^\nu(s) := \tfrac 1 {2\pi}\arcsin(\tfrac 1 2e^{2\pi (s + s^\nu)}),
\end{align*}
by: \begin{align*}
u_\ell^\nu(s,t) := u_\ell(s + s^\nu, t).
\end{align*}
The energy \(\tint_{B_{1/8}(0)} \tfrac 1 2|\d \ul u^\nu|^2\) is bounded by the energy of \(\ul v\), and by assumption, the gradient \(|\d \ul u^\nu(0, t^\nu)|\) tends to \(\infty\).
In the following sublemma we establish the last hypothesis needed to apply Corollary~\ref{cor:rescale}.

\begin{sublemma} \label{sublem:obedient}
The functions \(\theta^\nu(s) = \tfrac 1 {2\pi}\arcsin(\tfrac 1 2e^{2\pi(s + s^\nu)})\) obediently shrink to zero as \(\nu \to \infty\).
\end{sublemma}

\begin{proof}[Proof of Sublemma~\ref{sublem:obedient}]
The convergence \(s^\nu \to -\infty\) implies \(\tfrac 1 2e^{2\pi(s + s^\nu)} \to 0\) in \(\cC^0\), so the equality \(\arcsin(0) = 0\) implies the \(\cC^0\)-convergence of \(\theta^\nu\) to zero.

To check the second condition for obedient shrinking, fix \(k \geq 1\) and note that \(\tfrac {\d^k\theta^\nu}{\d s^k}(s) = \tfrac {\d^k\theta}{\d s^k}(s + s^\nu)\), with \(\theta(s) = \tfrac 1 {2\pi}\arcsin(\tfrac 1 2e^{2\pi s})\) as above.
The derivative \(\tfrac {\d^k\theta}{\d s^k}(s)\) is a linear combination of the functions \(f_\ell(s) := (4 - e^{4\pi s})^{-(\ell-1/2)}e^{4\pi(\ell - 1/2)s}\) for \(\ell = 1, \ldots, m\).
(This can be seen by induction starting from \(\tfrac {\d \theta(s)}{\d s} = (4 - e^{4\pi s})^{-1/2}e^{2\pi s}\).)
This decomposition, the inequality \(\arcsin(x) \geq x\) for \(x \in [0,1]\), and the convergence \(s^k \to -\infty\) allows us to establish the second condition: \begin{align*}
\sup_{\nu \in \N} \frac {\max_{s \in [-1/4,1/4]} |f_\ell(s)|} {\min_{s \in [-1/4,1/4]} \theta^\nu(s)} &\leq \sup_{\nu \in \N} \frac {\exp(4\pi(\ell-\tfrac 1 2)(s^\nu + 1/4))} {\tfrac 1 {4\pi}\exp(2\pi(s^\nu-1/4))} \\
&= \sup_{\nu \in \N} 4\pi\exp(4\pi((\ell-1)s^\nu + \tfrac 1 4)) \\
&\leq 4\pi\exp(\pi).
\end{align*}

The arcsine function extends to a holomorphic function \(\arcsin\colon \ol B_1(0) \to \C\) by the power series \(\arcsin(z) := \sum_{k=0}^\infty \tfrac {{2k \choose k}z^{2k+1}} {4^k(2k+1)}\), so \(f^\nu\) extends to a holomorphic function \(F^\nu\) from \([-1/4,1/4]^2\) to \(\C\).
Since the functions \(\tfrac 1 2e^{2\pi(z + s^\nu)}\) tend \(\cC^\infty\) to zero and since \(\arcsin(0) = 0\), the extensions \(F^\nu\) also tend \(\cC^\infty\) to zero.
\end{proof}

\noindent Part (2) of Corollary~\ref{cor:rescale} now implies the inequality \(\liminf_{\nu\to\infty} \tint_{B_{1/8}(0)} \tfrac 1 2|\d \ul u^\nu|^2 > 0\), which contradicts \eqref{eq:Eto0}.
\end{proof}

\begin{proof}[Proof of Lemma~\ref{lem:lengthsgotozero}]
First, note that the domain \([1/2 - \theta(s), 1/2 + \theta(s)] \cup [1 - \theta(s), 1 + \theta(s)]\) of \(\gamma_s^1\) has total length \(4\theta(s) = \tfrac 2 {\pi} \arcsin\bigl( \tfrac 1 2 e^{2\pi s} \bigr)\), which converges to \(0\) as \(s \to -\infty\).
Hence the gradient bounds of Lemma~\ref{lem:cylgradbound} immediately imply that the \(L^2\)-length of \(\gamma_s^1\) converges to zero as \(s \to -\infty\).
Moreover, these gradient bounds imply that to show the \(L^2\)-lengths of \(\gamma_s^0, \gamma_s^2\) converge to zero, it suffices to fix an arbitrary \(\eps > 0\) and show that the \(L^2\)-lengths of \(\gamma_s^0|_{[\eps,1/2-\eps]}, \gamma_s^2|_{[1/2+\eps, 1-\eps]}\) converge to zero as \(s \to -\infty\).

Fix \(\eps>0\).
We will show that the \(L^2\)-length of \(\gamma_s^0|_{[1/2+\eps,1-\eps]}\) converges to zero as \(s \to -\infty\); the proof for \(\gamma_2\) is similar.
Choose \(s_0\) so that the domain of \(\gamma_s^0\) contains \([1/2+\eps/2, 1-\eps/2]\) for all \(s\leq s_0\).
Now the \(\cC^0\)-bound on \(|\d v_0|\) from Lemma~\ref{lem:cylgradbound} induces a \(\cC^m\)-bound on \(v_0|_{(-\infty, s_0 - 1] \times [1/2+\eps, 1-\eps]}\) for any \(m \geq 0\).
Indeed, we can apply the interior elliptic estimates (e.g.\ \cite[\S6.3]{ah}) on each of the precompactly-nested domains \begin{align*}
[s_0 - k - 1, s_0 - k] \times [1/2+\eps,1-\eps] \;\subset\; [s_0 - k - 2, s_0 - k + 1] \times [1/2+\eps/2,1-\eps/2]
\end{align*}
for \(k\in\N\).
Since for different \(k\) these domains are translations of one other, the constants in the elliptic estimates are independent of \(k\), and thus yield the desired \(\cC^m\)-bounds.

For \(s \leq s_0\), define 
\begin{align*}
\Phi(s) := \tfrac 1 2\tint_{-\infty}^s \tint_{1/2+\eps}^{1-\eps} \left(|\partial_s v_0|^2 + |\partial_t v_0|^2\right).
\end{align*}
Then \(\Phi\colon (-\infty, s_0]\to[0, \infty)\) is nondecreasing with \(\lim_{s \to -\infty} \Phi(s) = 0\) and 
\begin{align*}
\Phi'(s) &= \tfrac 1 2\int_{1/2+\eps}^{1-\eps} (|\partial_s v_0(s, \tau)|^2 + |\partial_t v_0(s, \tau)|^2) \,\d \tau = \int_{1/2+\eps}^{1-\eps} |\partial_s v_0(s, \tau)|^2 \,\d \tau, \\
\Phi''(s) &= 2\int_{1/2+\eps}^{1-\eps} \langle \partial_s v_0(s, \tau), \nabla_{\text{LC}, s}^2v_0(s, \tau)\rangle \,\d \tau,
\end{align*}
where in the last quantity we are using the Levi-Civita connection with respect to the metric \(g_0\) defined in \eqref{eq:metrics}.
By the previous paragraph, there exists a constant \(c > 0\) so that \(\Phi''(s) \leq c\) for all \(s \leq s_0 - 1\).
Now for any fixed \(\delta > 0\) we can choose \(s_1 \leq s_0 - 1\) such that \(\Phi(s_1) \leq \delta^2/4c\).
For \(s \leq s_1\), we obtain:
\begin{align*}
\frac {\delta^2} {4c} 
\;\geq\; 
\Phi(s_1)
\;\geq\; 
\Phi(s) - \Phi(s- \tfrac \delta {2c})
\;=\; 
\int_{s - \delta/2c}^s \Phi'(\sigma) \,\d \sigma 
\;\geq\; 
 \frac \delta {2c}(\Phi'(s) - \tfrac \delta 2),
\end{align*}
where the last step uses the bound on  \(\Phi''\) to deduce \(\Phi'(\sigma) \geq \Phi'(s) - c |s-\sigma|\).
This inequality can be rearranged to yield \(\Phi'(s) \leq \delta\) for all \(s \leq s_1\), and thus proves \(\lim_{s \to -\infty} \Phi'(s) = 0\).
Since \(\Phi'(s)\) is equal to \(\|\tfrac \d {\d t} \gamma_s^0\|^2_{L^2([1/2+\eps,1-\eps])}\) and since \(|\tfrac \d {\d t}\gamma_s^0|\) is uniformly bounded, we have now shown that the \(L^2\)-norm of \(\gamma^0_s\) converges to zero as \(s \to -\infty\).

The Cauchy--Schwarz inequality implies that the \(L^1\)-norm of \(\tfrac \d {\d t}\gamma_0^s\) --- i.e.\ the length \(\ell(\gamma_1^s)\) --- also tends to zero as \(s \to -\infty\).
\end{proof}

\subsection{An isoperimetric inequality and the proof of removal of singularity}\label{ss:isoineq}

In this subsection, we prove Theorem~\ref{thm:remsing}.
The crucial inputs will be Lemma~\ref{lem:lengthsgotozero} from \S\ref{ss:lengths} together with the following isoperimetric inequality for the energy on \((-\infty,s_0] \times \R/\Z\), \begin{align*}
E(\ul v; s_0) := \tint_{(-\infty,s_0] \times \R/\Z} \tfrac 1 2|\d \ul v|^2 \,\d s\d t.
\end{align*}

\begin{lemma} \label{lem:isoineq}
There exists \(C > 0\) depending only on \(M_\ell,L_{\ell(\ell+1)},\omega_\ell,J_\ell\) such that the following inequality holds for all \(s \leq 0\): 
\begin{align*}
E(\ul v; s) \leq C\!\! \sum_{i \in \{0,1,2\}} \ell(\gamma_s^i)^2.
\end{align*}
\end{lemma}

\noindent We defer the proof to later in \S\ref{ss:isoineq}; now, we turn to the proof of removal of singularity.
Throughout this subsection we denote \begin{align*}
M_{0112} := M_0^- \times M_1 \times M_1^- \times M_2, \qquad M_{02} := M_0^- \times M_2.
\end{align*}

\begin{proof}[Proof of Theorem~\ref{thm:remsing}]
\begin{step1}
There exist \(C_1, C_2 > 0\) such that the inequality \(E(\ul v; s) \leq C_1\exp(C_2s)\) holds for all \(s \leq 0\).
\end{step1}

\medskip

%\comment{should the second ``\(\leq\)'' be an ``\(=\)''?}

\noindent Fix \(s \leq 0\).
The following inequality follows from Lemma~\ref{lem:isoineq}: \begin{align*}
E(\ul v; s) \stackrel{\text{Lem.~\ref{lem:isoineq}}}{\leq} C\!\!\sum_{\ell \in \{0,1,2\}} \ell(v_\ell(s,-))^2 \leq \frac C 2 \left(\tint_0^1 |\d \ul v(s, t)| \,\d t\right)^2 \leq \frac C 2 \tint_0^1 |\d \ul v(s,t)|^2 \,\d t = C\frac {\d} {\d s}(E(\ul v; s)).
\end{align*}
Manipulating this inequality and integrating from \(s\) to \(0\), we obtain \(E(\ul v; s) \leq E(\ul v; 0)\exp(s/C)\).

\begin{step2}
The limit \(\lim_{s \to -\infty} v_0(s, -)\) exists in \(\cC^0([5/8, 7/8], M_0)\).
\end{step2}

\noindent Fix a \(\cC^1\) embedding \(i\colon M_0 \to \R^N\); we will show that \(\Lambda := \lim_{s \to -\infty} (i \circ v_0|_{[5/8,7/8]})\) exists in \(\cC^0\).

We will do so by showing that \(\Lambda\) exists in \(W^{1,2}\), where \(W^{1,2}([5/8,7/8],\R^N)\) is defined using the Euclidean metric on \(\R^N\).
Fix \(s_2 \leq s_1 < 0\).
Cauchy--Schwarz implies the following inequality: \begin{align} \label{eq:L2viaMinkowski}
\|(i \circ v_0)(s_1, -) - (i \circ v_0)(s_2, -) \|_{L^2([5/8, 7/8])} &= \left(\int_{5/8}^{7/8} \left| \int_{s_2}^{s_1} \partial_s (i \circ v_0) \,\d s \right|^2 \,\d t \right)^{1/2} \\
&\leq (s_1 - s_2)^{1/2}\left( \int_{5/8}^{7/8} \int_{s_2}^{s_1} \left|\partial_s(i \circ v_0)\right|_{g_{\euc}}^2 \,\d s \d t\right)^{1/2}. \nonumber
\end{align}
Since \(M_0\) is compact, there exists a constant of equivalence \(\mu > 0\) for the norms induced by \(g_{M_0}\) and \(i^*g_\euc\), so \eqref{eq:L2viaMinkowski} yields the following: \begin{align}
\|(i \circ v_0)(s_1, -) - (i \circ v_0)(s_2, -) \|_{L^2([5/8, 7/8])} &\stackrel{\mathclap{\eqref{eq:L2viaMinkowski}}}{\leq} \mu(s_1 - s_2)^{1/2}\left( \int_{5/8}^{7/8} \int_{s_2}^{s_1} |\partial_sv_0|_{g_{M_0}}^2 \,\d s \d t\right)^{1/2} \nonumber \\
&\stackrel{\mathclap{\text{Step 1}}}{\leq} \mu \: C_1^{1/2}(s_1 - s_2)^{1/2} \exp(C_2s_1/2) \label{eq:anotherremsingcalc} \\
&=: C_3(s_1 - s_2)^{1/2}\exp(C_2s_1/2). \nonumber
\end{align}
Write \(s_2 = (m+\epsilon)s_1\) for \(m \in \N\) and \(\epsilon \in [0, 1)\).
We have: \begin{align}
& \|(i \circ v_0)(s_1, -) - (i \circ v_0)((m + \epsilon)s_1, -)\|_{L^2([5/8, 7/8])} \nonumber \\
&\hspace{1.25in}\leq \|(i \circ v_0)(ms_1, -) - (i \circ v_0)((m+\epsilon) s_1, -)\|_{L^2([5/8, 7/8])} \nonumber \\
& \hspace{2.25in} + \sum_{j=1}^{m-1} \|(i \circ v_0)(js_1, -) - (i \circ v_0)((j+1)s_1, -)\|_{L^2([5/8, 7/8])} \label{eq:yetanotherremsingcalc} \\
&\hspace{1.25in}\stackrel{\mathclap{\eqref{eq:anotherremsingcalc}}}{\leq} C_3|s_1|^{1/2}\sum_{j=1}^m \exp(j C_2 s_1/2) \nonumber \\
&\hspace{1.25in}\leq \frac {C_3|s_1|^{1/2}\exp(C_2 s_1/2)} {1 - \exp(C_2 s_1/2)}. \nonumber
\end{align}
This estimate would be enough to show that $\Lambda$ exists in $L^2$; we now make a further estimate in order to upgrade this convergence to $W^{1,2}$.
Define \(f(s) := | \tfrac {\d} {\d t}(i \circ v_0)(s,-) |_{L^2([5/8,7/8])}\).
This quantity tends to zero as \(s \to -\infty\): \begin{align*}
\limsup_{s \to -\infty} f(s) \leq \limsup_{s \to -\infty} \mu |\tfrac {\d} {\d t}v_0(s, -) |_{L^2([5/8,7/8])} \stackrel{\text{Lem.~\ref{lem:lengthsgotozero}}}{=} 0.
\end{align*}
We can now show that \(\Lambda\) exists in \(W^{1,2}\): We have \begin{align*}
|(i \circ v_0)(s_1, -) - (i \circ v_0)(s_2, -)|_{W^{1,2}([5/8,7/8])} &\leq |(i \circ v_0)(s_1, -) - (i \circ v_0)(s_2, -)|_{L^2([5/8,7/8])} \\
&\hspace{2.5in} + f(s_1) + f(s_2) \\
&\stackrel{\mathclap{\eqref{eq:yetanotherremsingcalc}}}{\leq} \frac {C_3|s_1|^{1/2}\exp(C_2 s_1/2)} {1 - \exp(C_2 s_1/2)} + f(s_1) + f(s_2),
\end{align*}
which implies the equality \begin{align*}
\limsup_{s_1 \to -\infty} \sup_{s_2 \in (-\infty, s_1]} |(i \circ v_0)(s_1, -) - (i \circ v_0)(s_2, -)|_{W^{1,2}([5/8,7/8])} = 0.
\end{align*}
Since \(W^{1,2}([5/8,7/8], \R^N)\) is complete, \(\Lambda\) exists in \(W^{1,2}\).
The Sobolev embedding \(W^{1,2} \hookrightarrow \cC^0\) for one-dimensional domains now implies that \(\Lambda\) exists in \(\cC^0\).

\begin{step3}
We prove Theorem~\ref{thm:remsing}.
\end{step3}

\noindent By Lemma~\ref{lem:lengthsgotozero}, the first claim of Theorem~\ref{thm:remsing} would follow from the existence of the limits \begin{align*}
\Lambda_0 := \lim_{s \to -\infty} v_0(s, \tfrac 3 4), \qquad \Lambda_1 := \lim_{s \to -\infty} v_1(s, \tfrac 1 2), \qquad \Lambda_1' := \lim_{s \to -\infty} v_1(s, 1), \qquad \Lambda_2 := \lim_{s \to -\infty} v_2(s, \tfrac 1 4).
\end{align*}
It follows from Step 2 that \(\Lambda_0\) exists, and an analogous argument shows that \(\Lambda_2\) exists.
It remains to show that \(\Lambda_1, \Lambda_1'\) exist.

To show that \(\Lambda_1\) exists, we will show convergence of the path \begin{align*}
\gamma\colon s \mapsto (v_0(s, \tfrac 1 2+\theta(s)), v_1(s, \tfrac 1 2), v_1(s, \tfrac 1 2), v_2(s, \tfrac 1 2 -\theta(s))
\end{align*}
as \(s \to -\infty\).
This path takes values in \(M_0 \times \Delta_{M_1} \times M_2\) and \(\lim_{s \to -\infty} d_{M_{0112}}(\gamma(s), L_{01} \times L_{12}) = 0\) (by Lemma~\ref{lem:cylgradbound}), so the distances \( d_{M_{0112}}(\gamma(s), L_{01} \times_{M_1} L_{12})\) converge to zero.
Hence there exists a path \(\beta\colon (-\infty,0] \to L_{01} \times_{M_1} L_{12}\) satisfying the equality \begin{align} \label{eq:approxofgamma}
\lim_{s\to-\infty} d_{M_{0112}}(\gamma(s),\beta(s)) = 0.
\end{align}
(Indeed, define \(\beta\) by choosing a tubular neighborhood \(U\) of \(L_{01}\times_{M_1} L_{12}\), and compose \(\gamma\) with the projection \(U \to L_{01} \times_{M_1} L_{12}\).)  We will show that \(\lim_{s \to -\infty} \gamma(s)\) exists by showing that \(\lim_{s \to -\infty} \beta(s)\) exists.

Lemma~\ref{lem:lengthsgotozero}, the existence of \(\Lambda_0\) and \(\Lambda_2\), and \eqref{eq:approxofgamma} imply that \(x_{02} := \lim_{s \to -\infty} \pi_{02}(\beta(s))\) exists.
Since \(\pi_{02}\) restricts to an immersion of \(L_{01} \times_{M_1} L_{12}\) into \(M_{02}\), there exist finitely many preimages \(x_{0112}^1, \ldots, x_{0112}^k\) of \(x_{02}\) in \(L_{01} \times_{M_1} L_{12}\).
Choose \(\eps > 0\) small enough that the preimage of \(B_\eps(x_{02})\) under \(\pi_{02}|_{L_{01} \times_{M_1} L_{12}}\) consists of \(k\) connected components \(U^1, \ldots, U^k\), with \(x_{0112}^j\) contained in \(U^j\).
Now choose \(s_0 \in (-\infty, 0]\) such that \(\pi_{02}(\beta((-\infty, s_0]))\) is contained in \(B_\eps(x_{02})\).
The image \(\beta((-\infty, s_2])\) must then be contained in a single \(U_j\).
If \((s_\nu), (s_\nu')\) are sequences with limit \(-\infty\) such that \(x_{0112}^{j_1} := \lim_{\nu \to \infty} \beta(s_\nu)\) and \(x_{0112}^{j_2} := \lim_{\nu \to \infty} \beta(s_\nu')\) exist, then \(j_1\) and \(j_2\) must be equal; since \(L_{01} \times_{M_1} L_{12}\) is compact, this is enough to conclude that \(\lim_{s \to -\infty} \beta(s)\) exists.
As noted above, this is enough to conclude the first statement of Theorem~\ref{thm:remsing}.

The points \((\Lambda_0, \Lambda_1, \Lambda_1, \Lambda_2)\) and \((\Lambda_0, \Lambda_1', \Lambda_1', \Lambda_2)\) are lifts in \(L_{01} \times_{M_1} L_{12}\) of \((\Lambda_0, \Lambda_2)\), so if the projection from \(L_{01}\times_{M_1} L_{12}\) to \(M_{02}\) is injective, then \(\Lambda_1, \Lambda_1'\) are the same point.
\end{proof}

\begin{figure}
\centering
\def\svgwidth{0.9\columnwidth}
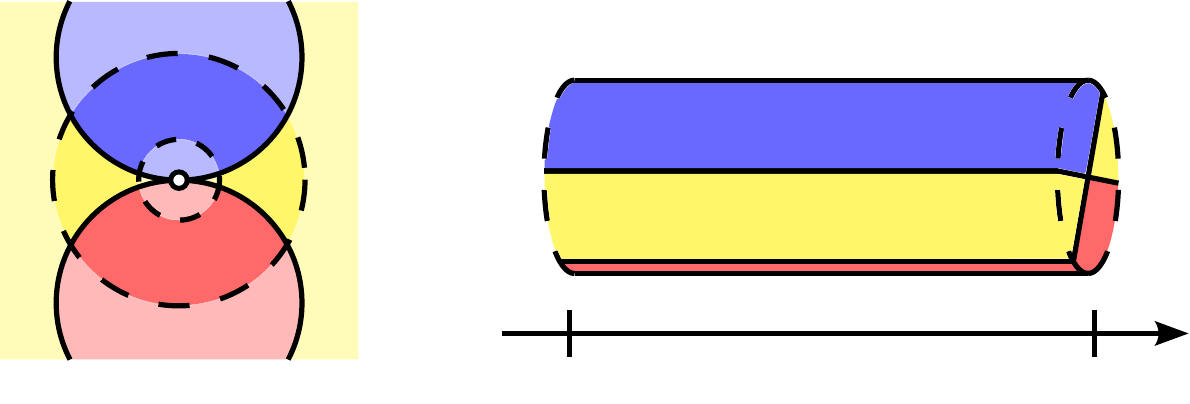
\caption{The start of our argument for Lemma~\ref{lem:isoineq} is to restrict an inverted figure eight to an annulus centered at the singular point (the portion in the left figure between the dotted circles), then reparametrize to a quilted tube with straight seams (the tubular part of the boundary of the cylinder on the right).
Next, we piecewise-smoothly extend to the interior of the cylinder.}
\label{fig:isoineq}
\end{figure}

Our proof of Lemma~\ref{lem:isoineq} is an adaptation to the quilted setting of \cite[Lemma~4.5.1]{ms:jh}, which is an isoperimetric inequality for the energy near an interior point of a \(J\)-holomorphic curve.
Their argument went like this: restricting the map to an annulus, then reparametrizing, yields a map defined on the curved part of the boundary of a cylinder.
By a lengths-go-to-zero result analogous to our Lemma~\ref{lem:lengthsgotozero}, they extend this map to the entire cylinder.
Their result now follows from Stokes' theorem, along with the isoperimetric inequality for the symplectic area applied to the top and bottom caps of the cylinder.
The difficulty in adapting this result to the quilted setting is in the extension to the cylinder (see Figure~\ref{fig:isoineq} for an illustration of the setup); the key will be the consequences of cleanly-immersed composition recorded in the following lemma.

\begin{lemma} \label{lem:cleanconsequences}
There exist \(C>0,\epsilon > 0\) such that:
\begin{itemize}
\item[(i)] If \(x_{02},y_{02}\in L_{01} \circ L_{12}\) have lifts \begin{align*}
x,x'\in\pi_{02}^{-1}\{x_{02}\} \cap
(L_{01}\times_{M_1}L_{12}), \qquad y,y'\in \pi_{02}^{-1}\{y_{02}\}\cap (L_{01}\times_{M_1}L_{12})
\end{align*}
with small distances \begin{align*}
\max \{ d_{M_{0112}}(x, y), d_{M_{0112}}(x', y') \} \leq \eps,
\end{align*}
then there exists a smooth path \(\gamma_{02}\colon [0, 1] \to M_{02}\) with image in \(L_{01}\circ L_{12}\) and smooth lifts \(\gamma, \gamma'\colon [0,1] \to L_{01}\times_{M_1}L_{12}\) that have bounded lengths \begin{align*}
\ell(\gamma_{02}) + \ell(\gamma) + \ell(\gamma') \leq C\, d_{M_{02}}(x_{02}, y_{02})
\end{align*}
and satisfy \(\gamma(0) = x\), \(\gamma(1) = y\), \(\gamma'(0) = x'\), and \(\gamma'(1) = y'\).

\item[(ii)] For \(x, x'\in L_{01}\times_{M_1}L_{12}\) with \(d_{M_{02}}(\pi_{02}(x), \pi_{02}(x')) \leq \epsilon\), there exists a point \(y_{02} \in L_{01} \circ L_{12}\) and preimages \(y, y' \in \pi_{02}^{-1}(y_{02}) \cap L_{01}\times_{M_1}L_{12}\) such that the following inequality holds: \begin{align*}
 d_{M_{02}}(\pi_{02}(x'), y_{02}) + d_{M_{02}}(\pi_{02}(x) , y_{02}) + d_{M_{0112}}(x, y) + d_{M_{0112}}(x', y') \leq C\,d_{M_{02}}(\pi_{02}(x), \pi_{02}(x')).
 \end{align*}
\end{itemize}
\end{lemma}

\noindent We will give only a brief sketch, since a formal proof is no more enlightening.
The key is that the cleanly-immersed hypothesis implies that any two branches of \(L_{01} \circ L_{12}\) meet like two vector subspaces.
\begin{itemize}
\item[(i)] If \(x,x',y,y'\) lie in the same local branch of \(L_{01}\circ L_{12}\), then the conclusion is immediate.
Otherwise, \(x\) and \(y\) lie in one branch, and \(x'\) and \(y'\) lie in another.
Represent these branches as open subsets of vector subspaces \(V, V' \subset \R^N\).
Then \(x_{02}, y_{02}\) lie in \(V \cap V'\), and we may define \(\gamma_{02}\) to be a path in \(V \cap V'\) from \(x_{02}\) to \(y_{02}\) and \(\gamma\) (resp.\ \(\gamma'\)) to be the lift to the portion of \(L_{01} \times_{M_1} L_{12}\) corresponding to \(V\) (resp.\ to \(V'\)).

\item[(ii)] If \(x,x'\) lie in the same local branch of \(L_{01}\circ L_{12}\), the conclusion is again immediate.
Otherwise, represent the branches containing \(x,x'\) as open subsets of \(V, V' \subset \R^N\).
Set \(y_{02}\) to be the nearest point in \(V \cap V'\) to \(x\), and let \(y\) (resp.\ \(y'\)) be the lift to the portion of \(L_{01} \times_{M_1} L_{12}\) corresponding to \(V\) (resp.\ to \(V'\)).
\end{itemize}

\begin{figure}
\centering
\def\svgwidth{0.3\columnwidth}
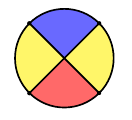
\caption{The domains used in the proof of Lemma~\ref{lem:isoineq}.}
\label{fig:isoineqdomains}
\end{figure}

\begin{proof}[Proof of Lemma~\ref{lem:isoineq}]
\begin{step1}
We prove Lemma~\ref{lem:isoineq} up to an extension result, which we defer to Steps 2 and 3.
\end{step1}

\noindent It suffices to prove the lemma for \(s \leq s_0 \leq 0\), where \(s_0\) is chosen so that \(\sup_{s \leq s_0} \ell(\gamma_s^i), \: i \in \{0,1,2\}\) is bounded by a constant \(\delta > 0\) to be determined later.
As illustrated in Figure~\ref{fig:isoineqdomains}, partition the unit circle \(S_1(0)\) into four segments by \begin{gather*}
A_0 := \{(x, y) \in S_1(0) \: | \: y \leq x, \: y \leq -x\}, \qquad A_1 := \{(x, y) \in S_1(0) \: | \: x \geq y, \: x \geq -y\}, \\
A_2 := \{(x, y) \in S_1(0) \: | \: y \geq x, \: y \geq -x\}, \qquad A_3 := \{(x, y) \in S_1(0) \: | \: x \leq y, \: x \leq -y \}
\end{gather*}
and set \(p_{i(i+1)} := A_i \cap A_{i+1}\) for \(i \in \Z/4\Z\).
Given \(s_1,s_2\) with \(s_2 < s_1 \leq s_0\), define maps\footnote{The maps $\sigma_i$ are simply the reparametrizations of $v_0,v_1,v_2$ from the intersections of $V_0,V_1,V_2$ with $\{(s,t)\:|\:s_2\leq s\leq s_1\}$ to the domains $A_i \times [s_2,s_1]$.
We are doing nothing in the $s$ factor and rescaling in the $t$ factor.
See Fig.~\ref{fig:isoineq} for an illustration of this reparametrization.} \(\sigma_i\colon A_i \times [s_2,s_1] \to M_i\), \(i \in \{0,1,2,3\}\) (where we set \(M_3 := M_1\)) like so: \begin{gather*}
\sigma_0(\exp(2\pi i t), s) := v_0\bigl(s, \tfrac 1 2 + \theta(s) + 4(\tfrac 1 2 - 2\theta(s))(t - \tfrac 5 8)\bigr), \quad \sigma_1(\exp(2\pi i t), s) := v_1\bigl(s, 8\theta(s)t \bigr), \\
\sigma_2(\exp(2\pi i t), s) := v_2\bigl(s, \theta(s) + 4(\tfrac 1 2 - 2\theta(s))(t - \tfrac 1 8)\bigr), \quad \sigma_3(\exp(2\pi i t), s) := v_1\bigl(s, \tfrac 1 2 + 8\theta(s)(t - \tfrac 1 2 )\bigr),
\end{gather*}
where we take \(t \in [-1/8,7/8]\).
These maps satisfy the seam condition \begin{align*}
(\sigma_i(p_{i(i+1)},s), \sigma_{i+1}(p_{i(i+1)},s)) \in L_{i(i+1)}, \qquad \forall \: i \in \Z/4\Z, \: s \in [s_2,s_1],
\end{align*}
where we set \(L_{23} := L_{12}^T\), \(L_{30} := L_{01}^T\).

In order to apply Stokes' theorem, we will extend the maps \(\sigma_i\) to the domains $U_i \times [s_2,s_1]$, where $U_i$ are the following four quadrants of the closed unit disk (refer again to Fig.~\ref{fig:isoineqdomains}): \begin{gather*}
U_0 := \{(x, y) \in \ol B(0,1) \: | \: y \leq x, \: y \leq -x\}, \qquad U_1 := \{(x, y) \in \ol B(0,1) \: | \: x \geq y, \: x \geq -y\}, \\
U_2 := \{(x, y) \in \ol B(0,1) \: | \: y \geq x, \: y \geq -x\}, \qquad U_3 := \{(x, y) \in \ol B(0,1) \: | \: x \leq y, \: x \leq -y \}.
\end{gather*}
Choose \(s_2 = t_0 < t_1 < \cdots < t_k = s_1\) such that for every \(j\), the diameters of the images \(\sigma_i( A_i \times [t_j, t_{j+1}] )\) are bounded by
\(\delta\).
As long as \(\delta\) is small enough, Steps 2 and 3 below allow us to extend \(\sigma_i\) to a continuous map \(\wt \sigma_i\colon U_i \times [s_2,s_1] \to M_i\) that is smooth on \(U_i \times [t_j,t_{j+1}]\), such that the extended maps satisfy the Lagrangian seam conditions \begin{align} \label{eq:isoperseamconds}
(\wt\sigma_i(p,s), \wt\sigma_{i+1}(p,s)) \in L_{i(i+1)} \qquad \forall \: p \in U_i \cap U_{i+1}, \: s \in [s_2,s_1].
\end{align}
Indeed, use Step 2 to define the maps \(\wt\sigma_i\) on the slices \(U_i \times \{t_j\}\), then use Step 3 to extend \(\wt\sigma_i\) to all of \(U_i \times [s_2,s_1]\).

Since \(\om_0, \om_1, \om_2\) are closed, Stokes' theorem yields the following: \begin{align*}
E(\ul v; [s_2, s_1] \times \R/\Z) &\leq \sum_{i \in \{1, 2\}} \Bigl| \sum_{j \in \{0,1,2,3\}} \tint_{U_j \times \{s_i\}} \wt\sigma_j^*\om_j \Bigr| \leq C \!\! \sum_{i \in \{1, 2\}} \sum_{j \in \{1,2,3\}} \ell(\gamma_{s_i}^j)^2,
\end{align*}
where in the first inequality we have used the seam conditions \eqref{eq:isoperseamconds}, and in the second inequality we have used the isoperimetric inequality for the symplectic area \cite[Theorem 4.4.1]{ms:jh}.
Taking the limit as \(s_2\) goes to \(-\infty\) and applying Lemma~\ref{lem:lengthsgotozero} yields the conclusion of the lemma.

Throughout the final two steps, the constants \(C_i\) may depend on the geometry of $L_{01}$, $L_{12}$, $\omega_\ell$, and $J_\ell$, but are independent of $\kappa$.

\begin{step2}
There exist \(C > 0\), \(\kappa_0 > 0\) so that if \(\sigma_0, \sigma_1, \sigma_2, \sigma_3\) are smooth maps with \begin{gather*}
\sigma_i\colon A_i \to M_i, \qquad (\sigma_i(p_{i(i+1)}), \sigma_{i+1}(p_{i(i+1)})) \in L_{i(i+1)}, \qquad \kappa := \max_{i \in \{0,1,2,3\}} \diam \sigma_i(A_i) \leq \kappa_0,
\end{gather*}
then there exist extensions \(\wt\sigma_i\colon U_i \to M_i\) of \(\sigma_i\) such that: \begin{gather*}
(\wt\sigma_i(p), \wt\sigma_{i+1}(p)) \in L_{i(i+1)} \quad \forall \: p \in U_i \cap U_{i+1}, \qquad \max_{i \in \{0,1,2,3\}} \ell(\wt\sigma_i|_{\partial U_i}) + \max_{i \in \{0,1,2,3\}} \diam \wt\sigma_i(U_i) \leq C\kappa.
\end{gather*}
\end{step2}

\noindent 
The points
\begin{align*}
z := (\sigma_0(p_{01}), \sigma_1(p_{01}), \sigma_1(p_{12}), \sigma_2(p_{12})), \qquad z' := (\sigma_0(p_{30}), \sigma_3(p_{30}), \sigma_3(p_{23}), \sigma_2(p_{23}))
\end{align*}
lie in \(L_{01} \times L_{12}\).
Since the intersection \((L_{01} \times L_{12}) \cap (M_0 \times \Delta_{M_1} \times M_2)\) defining \(L_{01} \times_{M_1} L_{12}\) is transverse, there are points \(x, x' \in L_{01} \times_{M_1} L_{12}\) that are close to \(z\) resp.\ \(z'\), \begin{align} \label{eq:xtoz}
d_{M_{0112}}(x, z) \leq C_1\kappa, \qquad d_{M_{0112}}(x', z') \leq C_1\kappa,
\end{align}
for a uniform constant \(C_1 >0\).
The triangle inequality bounds the distance between the projections of \(z, z'\): \begin{align*}
d_{M_{02}}(\pi_{02}(x), \pi_{02}(x')) &\leq d_{M_{02}}(\pi_{02}(x), \pi_{02}(z)) + d_{M_{02}}(\pi_{02}(z), \pi_{02}(z')) + d_{M_{02}}(\pi_{02}(z'), \pi_{02}(x')) \\
&\leq 2(C_1 + 1)\kappa.
\end{align*}
As long as \(\kappa_0\) is chosen to be small enough, it follows from Lemma~\ref{lem:cleanconsequences}(ii) that there exist lifts \(y, y' \in L_{01}\times_{M_1} L_{12}\) of a single point \(y_{02} \in L_{01} \circ L_{12}\) with small distances to \(x\) resp.\ \(x'\): \begin{align} \label{eq:xtoy}
d_{M_{0112}}(x, y) \leq C_2\kappa, \qquad d_{M_{0112}}(x', y') \leq C_2\kappa,
\end{align}
where \(C_2>0\) is another constant.
We can now define the extensions \(\wt\sigma_i\) at the origin: \begin{align*}
(\wt\sigma_0(0), \wt\sigma_1(0), \wt\sigma_1(0), \wt\sigma_2(0)) := y, \qquad (\wt\sigma_0(0), \wt\sigma_3(0), \wt\sigma_3(0), \wt\sigma_2(0)) := y'.
\end{align*}
Inequalities \eqref{eq:xtoz} and \eqref{eq:xtoy} and the triangle inequality yield: \begin{align*}
d_{M_{0112}}(y, z) \leq (C_1 + C_2)\kappa, \qquad d_{M_{0112}}(y', z') \leq (C_1 + C_2)\kappa.
\end{align*}
The local triviality of smooth submanifolds implies that there exists a constant \(C_3>0\) such that after redefining \(\kappa_0\) if necessary, we may extend the maps \(\wt\sigma_i\) to the set \(\{(a,b) \in \ol B(0,1) \: | \: b = \pm a\}\) such that the seam conditions \eqref{eq:isoperseamconds} hold and the length of the loop \(\wt\sigma_i|_{\partial U_i}\) is bounded by \(C_3\kappa\).
Once more redefining \(\kappa_0\) if necessary, we may extend each map \(\wt\sigma_i\) to \(U_i\) in such a way that the diameter of \(\wt\sigma_i(U_i)\) is bounded by \(C_4\kappa\) for \(C_4>0\) another constant.

\begin{step3}
There exists \(\lambda>0\) such that the following holds.
Assume that \(\sigma_0, \sigma_1, \sigma_2, \sigma_3\) are smooth maps and \(a < b\) are real numbers with: \begin{gather*}
\sigma_i\colon A_i \times [a,b] \cup U_i \times \{a, b\} \to M_i, \qquad \max_{i \in \{0,1,2,3\}} \diam \im \sigma_i \leq \lambda, \\
(\sigma_i(q),\sigma_{i+1}(q)) \in L_{i(i+1)} \qquad \forall \: q \in \bigl(p_{i(i+1)} \times [a,b]\bigr) \cup  \bigl((U_i \cap U_{i+1}) \times \{a,b\}\bigr) .
\end{gather*}
Then each \(\sigma_i\) can be extended to a smooth map \(\wt\sigma_i\colon U_i \times[a,b] \to M_i\) such that the following seam conditions hold: \begin{align*}
(\wt\sigma_i(q), \wt\sigma_{i+1}(q)) \in L_{i(i+1)} \qquad \forall \: q \in (U_0 \cap U_1) \times [a,b].
\end{align*}
\end{step3}

\noindent Define \(x, x', y, y' \in L_{01} \times_{M_1} L_{12}\) like so: \begin{align*}
x &:= (\sigma_0, \sigma_1, \sigma_1, \sigma_2)( 0 ,a), \qquad x' := (\sigma_0, \sigma_3, \sigma_3, \sigma_2)( 0 ,a), \\
y &:= (\sigma_0, \sigma_1, \sigma_1, \sigma_2)( 0 ,b), \qquad y' := (\sigma_0, \sigma_3, \sigma_3, \sigma_2)( 0 ,b).
\end{align*}
Then \(\pi_{02}(x) = \pi_{02}(x')\) and \(\pi_{02}(y) = \pi_{02}(y')\), and \(x\) resp.\ \(x'\) are close to \(y\) resp.\ \(y'\): \begin{align*}
d_{M_{0112}}(x, y) \leq 4\lambda, \qquad d_{M_{0112}}(x',y') \leq 4\lambda.
\end{align*}
It follows from Lemma~\ref{lem:cleanconsequences}(i) that as long as \(\lambda\) is chosen to be small enough, there exists a path \(\gamma_{02}\colon [a,b] \to L_{01} \circ L_{12}\) and lifts \(\gamma, \gamma'\colon [a,b] \to L_{01} \times_{M_1} L_{12}\) from \(x\) to \(y\) resp.\ from \(x'\) to \(y'\) of small lengths: \begin{align*}
\ell(\gamma) + \ell(\gamma') \leq C_5\lambda
\end{align*}
for \(C_5>0\) a constant.
Define \(\wt\sigma_0,\wt\sigma_1,\wt\sigma_2,\wt\sigma_3\) on \(\{0\} \times [a,b]\) like so: \begin{align*}
(\wt\sigma_0,\wt\sigma_1,\wt\sigma_1, \wt\sigma_2)(0, t) := \gamma(t), \qquad (\wt\sigma_0,\wt\sigma_3,\wt\sigma_3,\wt\sigma_2)(0,t) := \gamma'(t).
\end{align*}
The diameter of the loop \((\wt\sigma_0, \wt\sigma_1)|_{\partial((U_0 \cap U_1) \times [a,b])}\) is bounded by \(2(C_5+1)\lambda\), so by redefining \(\lambda\) if necessary, we may extend \((\wt\sigma_0, \wt\sigma_1)\) to a map \((U_0 \cap U_1) \times [a,b] \to M_0^- \times M_1\) with small diameter: \begin{align*}
\diam\, ((\wt\sigma_0,\wt\sigma_1)((U_0 \cap U_1) \times[a,b])) \leq C_6\lambda
\end{align*}
for \(C_6 > 0\) a constant.
Extend \((\wt\sigma_1, \wt\sigma_2)\), \((\wt\sigma_2,\wt\sigma_3)\), \((\wt\sigma_3,\wt\sigma_0)\) to \((U_1\cap U_2)\times[a,b]\), \((U_2\cap U_3) \times[a,b]\), \((U_3\cap U_0)\times[a,b]\) in the same fashion.
Finally, \(\wt\sigma_i|_{\partial(U_i \times [a,b])}\) is a map to \(M_i\) from a domain homeomorphic to \(S^2\), and its diameter is small: \begin{align*}
\diam\, (\wt\sigma_i(\partial(U_i \times [a,b]))) \leq (2C_6 + 1)\lambda.
\end{align*}
Redefining \(\lambda\) if necessary, we may extend \(\wt\sigma_i\) to all of \(U_i \times [a,b]\).
\end{proof}

\section{Convergence modulo bubbling for strip-shrinking} \label{sec:convergence}

The purpose of this section is to prove a convergence-mod-bubbling result, which we state as Thm.~\ref{thm:nonfoldedstripshrink} below.
It is a strengthening of the strip-shrinking analysis of \cite{isom} from \(H^2\cap W^{1,4}\)-convergence to \(\cC^k\)-convergence; we also allow the domain to be equipped with nonstandard complex structures and the geometric composition \(L_{01} \circ L_{12}\) to be immersed, rather than embedded.
Thm.~\ref{thm:nonfoldedstripshrink} is used to prove the Gromov Compactness Theorem in \cite{bw:compactness}, which we in turn rely on in \S\ref{sec:remsing} of the current paper to prove the Removal of Singularity Theorem \ref{thm:remsing}.
The proof of Thm.~\ref{thm:nonfoldedstripshrink} (which we will give in \S\ref{ss:compactness_proof}) relies on a collection of $\delta$-independent elliptic estimates, which we will formulate and prove in \S\ref{ss:estimate}.

\begin{theorem} \label{thm:nonfoldedstripshrink}
There exists \(\eps > 0\) such that the following holds: Fix \(k \in \N_{\geq 1}\), positive reals \(\delta^\nu \to 0\) and \(\rho > 0\), symmetric complex structures\footnote{See \S\ref{ss:CsandACs} for the definition of a symmetric complex structure.} \(j^\nu\) on \([-\rho,\rho]^2\) that converge \(\cC^\infty\) to \(j^\infty\) with \(\| j^\infty - i \|_{\cC^0} \leq \eps\), and \(\cC^{k+2}_\loc\)-bounded sequences of domain-dependent compatible almost complex structures \(J_\ell^\nu: [-\rho,\rho]^2 \to \J_\ell(M_\ell,\om_\ell)\), \(\ell \in \{0,1,2\}\) such that the \(\cC^{k+1}\)-limit of each \((J_\ell^\nu)\) is a compatible \(\cC^\infty\) almost complex structure \(J_\ell^\infty\colon [-\rho,\rho]^2 \to \J(M_\ell,\omega_\ell)\).

Then if \((v_0^\nu,v_1^\nu,v_2^\nu)\) is a sequence of size-\((\delta^\nu,\rho)\) \((J_0^\nu,J_1^\nu,J_2^\nu, j^\nu)\)-holomorphic squiggly strip quilts for \((L_{01}, L_{12})\) with uniformly bounded gradients, \begin{align*}
\sup_{\nu \in \N, \: (s,t) \in [-\rho,\rho]^2}  |\d v^\nu|(s,t) < \infty,
\end{align*}
then there is a subsequence in which \((v_0^\nu(t - \delta^\nu))\), \((v_1^\nu|_{t=0})\), \((v_2^\nu(t + \delta^\nu))\) converge \(\cC^k_\loc\) to a \((J_0^\infty, J_2^\infty,j^\infty)\)-holomorphic size-\(\rho\) degenerate strip quilt \((v_0^\infty, v_1^\infty, v_2^\infty)\) for \(L_{01} \times_{M_1} L_{12}\).

If the inequality \(\liminf_{\nu\to\infty, (s,t) \in [-\rho,\rho]^2} | \d v^\nu |(s,t) > 0\) holds, then \(v_0^\infty, v_2^\infty\) are not both constant.
\end{theorem}

We now fix some data and explain the basic setup we will use for the proof of Thm.~\ref{thm:nonfoldedstripshrink}.

\medskip

\begin{center}
\parbox{0.9\columnwidth}{\it Fix for \S\ref{sec:convergence} closed symplectic manifolds \(M_0, M_1, M_2\) and compact Lagrangians \(L_{01} \subset M_0^- \times M_1\), \(L_{12} \subset M_1^- \times M_2\) with immersed composition as defined in the beginning of \S\ref{sec:remsing}.}
\end{center}

\medskip

\noindent For convenience, we will denote by \((M_{02}, \om_{02})\), \((M_{0211},\om_{0211})\) the symplectic manifolds \begin{gather*}
(M_{0211},\om_{0211}) := M_0 \times M_2^- \times M_1^- \times M_1 = (M_0\times M_2 \times M_1 \times M_1, \om_0 \oplus (-\om_2) \oplus (-\om_1) \oplus \om_1), \\
(M_{02}, \om_{02}) := M_0^- \times M_2 = (M_0 \times M_2, (-\om_0) \oplus \om_2)
\end{gather*}
and by \((L_{01}\times L_{12})^T\subset M_{0211}\) the transposed Lagrangian gotten by permuting the factors of \(M_{0211}\) by \((x_0,x_1,y_1,x_2) \mapsto (x_0,x_2,x_1,y_1)\).

The analysis in our proof of Theorem~\ref{thm:nonfoldedstripshrink} will be phrased in terms of pairs of smooth maps\footnote{This ``folded'' setup was first used in \cite{isom}.
It is more convenient to work with maps of this form, e.g.\ when we construct the compatible connection in Lem.~\ref{lem:conn} and prove the first estimate in Lem.~\ref{lem:k=0ellipticineq}.} \((w_{02},\wh w) = ((w_0,w_2), (w_0',w_2',w_1',w_1))\):
\begin{gather} \label{eq:foldedmaps}
w_{02}\colon (-\rho,\rho) \times [0,\rho-2\delta) \to M_{02}, \qquad \wh w\colon (-\rho,\rho) \times [0,\delta] \to M_{0211}, \\
(w_{02},\wh w)(s, 0) \in \Delta_{M_{02}} \times \Delta_{M_1}, \qquad \wh w(s, \delta) \in (L_{01}\times L_{12})^T \qquad \forall \: s \in (-\rho,\rho), \nonumber
\end{gather}
where \(\delta\) is nonnegative.
From now on we denote the domains of \(w_{02}\) and \(\wh w\) by \begin{align*}
Q_{02,\delta,\rho} := (-\rho,\rho) \times [0,\rho-2\delta), \qquad \wh Q_{\delta,\rho} := (-\rho,\rho) \times [0,\delta],
\end{align*}
and combine them into the notation \(\bQ_{\delta,\rho} := (Q_{02,\delta,\rho}, \wh Q_{\delta,\rho})\).
We denote the closures in \(\R^2\) by \begin{align*}
\ol Q_{02,\delta,\rho} := [-\rho,\rho] \times [0,\rho-2\delta], \qquad \ol{\wh Q}_{\delta,\rho} := [-\rho,\rho] \times [0,\delta].
\end{align*}

For \(\delta > 0,\rho>0\) (resp.\ \(\delta = 0,\rho>0\)), the setup\footnote{Here we use the notation \eqref{eq:foldedmaps}\(_{\delta,\rho}\) to explicitly indicate the dependence of \eqref{eq:foldedmaps} on $\delta$ and $\rho$.
We will use similar notation elsewhere; it will be a succinct way to refer to equations with the parameters specialized in various ways.}
\eqref{eq:foldedmaps}\(_{\delta,\rho}\) is equivalent to a triple of smooth maps \((v_0,v_1,v_2)\) with the same domain and targets as a size-\((\delta,\rho)\) squiggly strip quilt for \((L_{01}, L_{12})\) \eqref{eq:squigglymaps}\(_{f = \delta}\) (resp.\ as a size-\(\rho\) degenerate strip quilt for \(L_{01} \times_{M_1} L_{12}\) \eqref{eq:degenmaps}) and that fulfill the seam conditions \eqref{eq:squigglyseams}\(_{f=\delta}\) (resp.\ \eqref{eq:degenseams}) but are not necessarily pseudoholomorphic or of finite energy.
Indeed, given such \((v_0,v_1,v_2)\), define \((w_{02}, \wh w)\) like so: \begin{gather} \label{eq:fromtripletofolded}
w_{02}(s,t) := ( v_0(s, -t - 2\delta), v_2(s, t + 2\delta)), \\
\wh w(s,t) := (v_0(t - 2\delta), v_2(s, -t + 2\delta), v_1(s, -t), v_1(s, t)). \nonumber
\end{gather}
Conversely, for \(\delta \geq 0\) and \((w_{02}, \wh w)\) satisfying \eqref{eq:foldedmaps}\(_{\delta,\rho}\), define \((v_0,v_1,v_2)\) satisfying \eqref{eq:squigglymaps}\(_{f=\delta}\), \eqref{eq:squigglyseams}\(_{f=\delta}\) (for \(\delta > 0\)) or \eqref{eq:degenmaps}, \eqref{eq:degenseams} (for \(\delta = 0\)) like so: \begin{gather}
v_0(s,t) := \begin{cases}
w_0'(s, t + 2\delta), & -2\delta \leq t \leq -\delta, \\
w_0(s, -t - 2\delta), & t \leq -2\delta,
\end{cases} \qquad v_2(s,t) := \begin{cases}
w_2'(s, -t + 2\delta), & \delta \leq t \leq 2\delta, \\
w_2(s, t - 2\delta), & 2\delta \leq t,
\end{cases} \nonumber \\
v_1(s,t) := \begin{cases}
w_1'(s, -t), & -\delta \leq t \leq 0, \\
w_1(s, t), & 0 \leq t \leq \delta. \label{eq:fromfoldedtotriple}
\end{cases}
\end{gather}
The transformations \eqref{eq:fromtripletofolded}, \eqref{eq:fromfoldedtotriple} are inverse to one another.

\subsection{Complex and almost complex structures in the folded and unfolded setups} \label{ss:CsandACs}

The Gromov Compactness Theorem in \cite{bw:compactness} is proved by ``straightening'' the seams of a squiggly strip quilt.
Pushing forward the standard complex structure from the squiggly strip quilt to the new quilt with horizontal seams produces a nonstandard complex structure, which is symmetric under conjugation.
We axiomatize this property in the following definition.

\begin{definition}
Fix \(\rho > 0\).
A {\bf symmetric complex structure on \(\mathbf{[-\rho,\rho]^2}\)} is a complex structure \(j\) such that the equality \begin{align*}
j(s,t) = -\sigma \circ j(s,-t) \circ \sigma
\end{align*}
holds for any \((s,t) \in [-\rho,\rho]^2\), where \(\sigma\) is the conjugation \(\alpha\partial_s + \beta\partial_t \mapsto \alpha\partial_s - \beta\partial_t\).
\end{definition}

When a symmetric complex structure, almost complex structures, and a pseudoholomorphic squiggly strip quilt are ``pushed forward'' by the folding operation \eqref{eq:fromtripletofolded}, the result is a ``coherent system of complex structures'', a ``coherent pair of almost complex structures'', and a ``pseudoholomorphic folded strip quilt'', defined as follows.

%In fact, we can ``push forward'' the complex structure on \(\C\) and the almost complex structures on \(M_0, M_1, M_2\) by \eqref{eq:fromtripletofolded}, so that a \((J_0, J_1, J_2, j)\)-holomorphic size-\((\delta,\rho)\) squiggly strip quilt of width \(\delta\) is equivalent to a {\bf \(\mathbf{(\bJ, \bj)}\)-holomorphic size-\(\mathbf{(\delta,\rho)}\) folded strip quilt} for \(\bj\) a {\bf coherent collection of complex structures} and \(\bJ\) a {\bf coherent pair of almost complex structures}.  After defining these notions, we will explain how these setups are equivalent.

\begin{definition} \label{def:coherent}
Fix \(\delta > 0\) and \(\rho > 0\).
\begin{itemize}
\item A {\bf coherent collection of complex structures \(\bj\) on \(\mathbf{\ol \bQ_{\delta,\rho}}\)
} is a pair \(\bj = (j_{02}, \wh j) = ((j_0, j_2), (j_0',j_2',j_1',j_1))\), where \(j_0, j_2\) (resp.\ \(j_0', j_2', j_1', j_1\)) are complex structures on \(\ol Q_{02,\delta,\rho}\) (resp.\ on \(\ol{\wh Q}_{\delta,\rho}\)) such that the following equalities hold for all \(s \in (-\rho,\rho)\): \begin{gather}
j_\ell(s,0) = -\sigma \circ j_\ell'(s,0) \circ \sigma, \label{eq:coherentCs0} \\
j_0'(s,\delta) = j_2'(s,\delta), \qquad j_1'(s,\delta) = j_1(s,\delta), \qquad j_1'(s,\delta) = -\sigma \circ j_0'(s,\delta) \circ \sigma. \label{eq:coherentCsdelta}
\end{gather}

\item A {\bf coherent pair of almost complex structures \(\bJ\) on \(\mathbf{\ol\bQ_{\delta,\rho}}\)} is a pair \(\bJ = (J_{02}, \wh J)\), where \(J_{02}, \wh J\) are almost complex structures \begin{align*}
J_{02}: \ol Q_{02,\delta,\rho} \to \J(M_{02},\om_{02}), \qquad \wh J: \ol{\wh Q}_{\delta,\rho} \to \J(M_{0211},\om_{0211})
\end{align*}
satisfying the following compatibility condition: For \(s \in (-\rho,\rho)\), $\wh J(s,0)$ decomposes as
\begin{align*}
\wh J(s,0) = (-J_{02}(s,0)) \oplus \wh J_{11}(s),
\end{align*}
where $\wh J_{11}(s), s \in (-\rho,\rho)$ is some almost-complex structure on $M_{11}$.

\item Fix a coherent collection \(\bj\) of complex structures and a coherent pair \(\bJ\) of almost complex structures on \(\ol\bQ_{\delta,\rho}\).
A {\bf \(\mathbf{(\bJ,\bj)}\)-holomorphic size-\(\mathbf{(\delta,\rho)}\) folded strip quilt} is a collection of smooth maps \(w = (w_{02}, \wh w) = ((w_0,w_2), (w_0',w_2',w_1',w_1))\) satisfying \eqref{eq:foldedmaps} that have finite energy, \begin{align*}
\tint_{Q_{02,\delta,\rho}} u_{02}^*\om_{02} < \infty, \qquad \tint_{\wh Q_{\delta,\rho}} \wh u^*\om_{0211} < \infty,
\end{align*}
and satisfy the Cauchy--Riemann equations \begin{align*}
\dbar_{\bJ,\bj}w &= (\dbar_{02,J_{02},j_{02}}w_{02}, \wh \dbar_{\wh J,\wh j}\wh w) = 0,
\end{align*}
where \(\dbar_{\bJ,\bj} = (\dbar_{02,J_{02},j_{02}}, \wh \dbar_{\wh J,\wh j})\) is the pair of operators defined by: \begin{align}
\dbar_{02,J_{02},j_{02}}w_{02} &:= (\d w_0,\d w_2) \circ (j_0,j_2)(\partial_s) - J_{02}(-, w_{02})\circ (\partial_s w_0,\partial_s w_2), \nonumber \\
\wh \dbar_{\wh J, \wh j}\wh w &:= (\d w_0',\d w_2',\d w_1',\d w_1) \circ (j_0',j_2',j_1',j_1)(\partial_s) - \wh J(-,\wh w) \circ (\partial_s w_0',\partial_s w_2',\partial_s w_1',\partial_s w_1). \label{eq:delbar}
\end{align}
\end{itemize}
\end{definition}

\noindent Given a \((J_0, J_1, J_2, j)\)-holomorphic squiggly strip quilt \((v_0,v_1,v_2)\) with \(j\) symmetric, we can produce a folded strip quilt like this: Define a coherent collection \(\bj\) of complex structures by \begin{gather} \label{eq:tripletofoldedCs}
j_{02}(s,t) = (j_0, j_2)(s,t) := (-\sigma \circ j(s,-t-2\delta) \circ \sigma, j(s,t+2\delta)), \\
\wh j(s,t) = (j_0', j_2', j_1', j_1)(s,t) := (j(s,t-2\delta), -\sigma\circ j(s,-t+2\delta) \circ \sigma, -\sigma \circ j(s,-t) \circ\sigma, j(s,t)) \nonumber
\end{gather}
and a coherent pair \(\bJ\) of almost complex structures by \begin{gather} \label{eq:tripletofoldedACs}
J_{02}(s,t) := (-J_0(s,-t-2\delta)) \oplus J_2(s, t+2\delta), \\
\wh J := J_0(t-2\delta) \oplus (-J_2(-t+2\delta)) \oplus (-J_1(s,-t)) \oplus J_1(s,t). \nonumber
\end{gather}
If \((w_{02}, \wh w)\) is defined by applying \eqref{eq:fromtripletofolded} to \((v_0,v_1,v_2)\), then \((w_{02},\wh w)\) is a \((\bJ, \bj)\)-holomorphic size-\((\delta,\rho)\) folded strip quilt.
Indeed, \((w_{02},\wh w)\) have the correct domains and codomains and satisfy the seam conditions, as discussed earlier, and the finite-energy hypothesis on \((v_0,v_1,v_2)\) implies that \((w_{02},\wh w)\) has finite energy.
The Cauchy--Riemann equation \eqref{eq:squigglyCR} for \(v_0\) on \((-\rho,\rho)\times(-\rho,-2\delta]\) can be rewritten as \begin{align*}
\d w_0(s,t) \circ (-\sigma \circ j(s,-t-2\delta) \circ \sigma) - (-J_0(s,-t-2\delta,w_0(s,t))) \circ \d w_0(s,t) = 0
\end{align*}
for \(w_0(s,t) := v_0(s,-t-2\delta)\) as in \eqref{eq:fromtripletofolded}, so \(w_0\) is \((-J_0(s,-t-2\delta), j_0(s,t))\)-holomorphic on \(Q_{02,\delta,\rho}\).
Five similar calculations complete the check that \((w_{02},\wh w)\) is \((\bJ, \bj)\)-holomorphic.

Finally, we consider the coordinate representation of a coherent collection of complex structures.  Fix a coherent collection \(\bj = ((j_0, j_2), (j_0', j_2', j_1', j_1))\) of complex structures on \(\ol\bQ_{\delta,\rho}\).
Define \(a_0(s,t), c_0(s,t) \in \R\) by \begin{align} \label{eq:Cscoordsdef}
j_0(s,t)(\partial_s) =: a_0(s,t)\partial_s + c_0(s,t)\partial_t,
\end{align}
and define \(a_j(s,t), c_j(s,t)\) for \(j \in \{1,2\}\) and \(a_k'(s,t), c_k'(s,t)\) for \(k \in \{0,1,2\}\) in the same way.
Then \eqref{eq:coherentCs0} and \eqref{eq:coherentCsdelta} translate into the following conditions on these coefficients: \begin{gather}
a_j(s,0) = -a_j'(s,0), \quad c_j(s,0) = c_j'(s,0) \qquad \forall \: j \in \{0,1,2\}, \label{eq:Cscoordconds} \\
a_0(s,\delta) = a_2(s,\delta), \qquad a_1'(s,\delta) = a_1(s,\delta), \qquad a_0(s,\delta) = -a_1'(s,\delta), \nonumber \\
c_0(s,\delta) = c_2(s,\delta), \qquad \: c_1'(s,\delta) = c_1(s,\delta), \qquad c_0(s,\delta) = c_1'(s,\delta). \nonumber
\end{gather}
We will use this coordinate representation in \S\ref{ss:estimate}.

\subsection{A collection of \(\delta\)-independent elliptic estimates} \label{ss:estimate}

This subsection is devoted to proving Lemma~\ref{lem:masterestimate}, which is the crucial \(\delta\)-independent elliptic estimate needed for the proof of Theorem~\ref{thm:nonfoldedstripshrink}.

\medskip

\begin{center}
\parbox{0.9\columnwidth}{\it In addition to the data fixed at the beginning of \S\ref{sec:convergence}, fix for \S\ref{ss:estimate} \(\rho>0\) and a pair of maps \(u = (u_{02}, \ol u)\) satisfying \eqref{eq:foldedmaps}\(_{\delta=0,\rho}\).}
\end{center}

\medskip

\noindent Furthermore, we continue to denote by \(\bi\) the standard coherent collection of complex structures defined in \eqref{eq:standardj}, and for any \(\delta \in (0,\rho/4]\) we define a pair \(u_\delta = (u_{02,\delta}, \wh u_\delta)\) of smooth maps satisfying \eqref{eq:foldedmaps}\(_{\delta,\rho}\) by: \begin{align} \label{eq:uextend}
u_{02,\delta} := u_{02}|_{Q_{02,\delta,\rho}}, \qquad \wh u_\delta(s,t) := \ol u(s).
\end{align}

Our approach is inspired by \cite{isom}, but we deviate from that approach by working with a special connection which allows us to drop boundary terms from the \(H^2\)-estimate \cite[Lemma~3.2.1(b)]{isom}.
This special connection is constructed in the following lemma, which is a generalization to the immersed case of a connection constructed in \cite{w:ss}.

\begin{lemma}\label{lem:conn}
There is an assignment \(\delta \mapsto \nabla_\delta = (\nabla_{02,\delta}, \wh \nabla_\delta)\) that sends \(\delta \in (0,\rho/4]\) to a pair of connections \(\nabla_{02,\delta}\) resp.\ \(\wh \nabla_\delta\) on \(u_{02,\delta}^*\rT M_{02} \to Q_{02,\delta,\rho}\) resp.\ \(\wh u_\delta^*\rT M_{0211} \to \wh Q_{\delta,\rho}\) such that the following hold:
\begin{itemize}
\item Parallel transport under \(\wh\nabla_\delta\) preserves \(\wh u_\delta^*\rT(L_{01}\times L_{12})^T\) and \(\wh u_\delta^*\rT(M_{02}\times\Delta_{M_1})\);

\item For a section \(\wh\zeta \in \Gamma(\wh u_\delta^*\rT(M_{02}\times\Delta_{M_1}))\)
we have
\(\nabla_{02,\delta,s} (p \circ \wh\zeta) = p \circ \wh\nabla_{\delta,s}\wh\zeta\),
where
\(p\colon \wh u_\delta^*\rT(M_{02}\times\Delta_{M_1})\to u_{02,\delta}^*\rT M_{02}|_{t=0}\)
is the projection;

\item For \(\delta_1 < \delta_2\), the restrictions of \(\nabla_{\delta_1}, \nabla_{\delta_2}\) agree: \begin{align*}
\nabla_{02,\delta_1}|_{Q_{02,\delta_2,\rho}} = \nabla_{02,\delta_2}, \qquad \wh\nabla_{\delta_2}|_{\wh Q_{\delta_1,\rho}} = \wh\nabla_{\delta_1}.
\end{align*}
\end{itemize}
\end{lemma}

\begin{proof}
Fix metrics on \(u_{02}^*\rT M_{02}\) and \(\ol u^*\rT M_{0211}\) so that given a smooth subbundle, we may form its orthogonal complement.
For any fixed \(s \in (-\rho,\rho)\) we denote:
\begin{gather*}
\Lambda_{0211} := \rT_{\ol u(s)} (L_{01} \times L_{12})^T, \quad\! \Delta := \rT_{\ol u(s)} (M_{02} \times \Delta_{M_1}), \quad\! \wh \Lambda_{02} := \Lambda_{0211} \cap \Delta, \quad\! \Lambda_{02} := \rT\pi_{02, \ol u(s)}(\wh \Lambda_{02}).
\end{gather*}
The transversality of \(L_{01} \times L_{12} \pitchfork M_0 \times \Delta_{M_1} \times M_2\) implies
\(\wh \Lambda_{02} = \rT_{\ol u(s)} \wh L_{02}\),
so the projection from \(\wh \Lambda_{02}\) to \(\Lambda_{02}\) is injective (see e.g.\ \cite[Lemma~2.0.5]{quiltfloer}).
Hence the intersection of \(\wh \Lambda_{02}\) and \(\{0\} \times \rT_{(\ol u_1(s), \ol u_1(s))}\Delta_{M_1}\) is trivial.
It follows that if we let \(C_1\) denote the complement of \(\wh \Lambda_{02} + (\{0\} \times \rT_{(\ol u_1(s), \ol u_1(s))}\Delta_{M_1})\) in \(\Delta\), the diagonal decomposes as \(\Delta = \wh \Lambda_{02} \oplus C_1 \oplus (\{0\} \times \rT_{(\ol u_1(s), \ol u_1(s))}\Delta_{M_1})\).
Let \(C_2\) be the complement of \(\wh \Lambda_{02}\) in \(\Lambda_{0211}\).
Transversality implies
\(\rT_{\ol u(s)} M_{0211} = \Lambda_{0211} + \Delta\), so we have deduced the following decomposition: \begin{align*}
\rT_{\ol u(s)} M_{0211} = C_2 \oplus \wh \Lambda_{02} \oplus C_1 \oplus (\{0\} \times \rT_{(\ol u_1(s), \ol u_1(s))}\Delta_{M_1}).
\end{align*}
The subspace \(\Lambda_{0211}\) (resp.\ \(\Delta\)) is given by the sum of the first two factors (resp.\ the sum of the last three factors) in this decomposition.
Therefore, if we choose connections on each of these four subbundles and set \(\ol\nabla\) to be the product connection, then extend \(\ol\nabla\) to a connection \(\wh \nabla_\delta\) on \(\wh u_\delta^*\rT M_{0211} \to \wh Q_{\delta,\rho}\) by defining \(\wh\nabla_{\delta,s}((s,t) \mapsto \wh\zeta(s,t)) := \ol\nabla_s(s \mapsto \wh\zeta(s,t))\) and defining \(\wh\nabla_{\delta,t}((s,t) \mapsto \wh\zeta(s,t)) := \nabla_{\wh g,t}(t \mapsto \wh\zeta(s,t))\) in terms of the Levi-Civita connection \(\nabla_{\wh g}\), \(\wh\nabla_\delta\) satisfies the first bullet.

Denote by \(p\colon \ol u^* \rT(M_{02} \times \Delta_{M_1}) \to u_{02}^* \rT M_{02}|_{t = 0}\) projection
and by \(i\colon u_{02}^* \rT M_{02}|_{t = 0} \to u_{02}^* \rT(M_{02} \times \Delta_{M_1})\) the inclusion defined by sending \(v \in T_{u_{02}(s,0)}M_{02}\) to \((v, 0) \in T_{\ol u(s)}(M_{02} \times \Delta_{M_1})\).
Define a connection \(p_*\ol \nabla\) on \(u_{02}^* \rT M_{02}|_{t = 0}\) by \((p_* \ol \nabla)(\zeta_{02}) := p \circ \ol \nabla (i\circ \zeta_{02})\).
Extend \(p_*\ol \nabla\) in any way to a connection \(\nabla_{02}\) on \(u_{02}^*\rT M_{02}\); for \(\delta \in (0,\rho/4]\), define \(\nabla_{02,\delta} := \nabla_{02}|_{Q_{02,\delta,\rho}}\).
The second bullet now follows from a computation, in which \((\zeta_{02}, \wh \zeta_1, \wh\zeta_1)\) is an arbitrary section of \(\wh u_\delta^*\rT(M_{02}\times\Delta_{M_1})\):
\begin{align*}
p \circ \wh \nabla_{\delta,s} \wh\zeta \;=\; p \circ \wh \nabla_{\delta,s} (\wh\zeta_{02}, \wh\zeta_1, \wh\zeta_1) \;=\; p \circ \wh \nabla_{\delta,s} (i \circ p \circ \wh\zeta) + p \circ \wh \nabla_{\delta,s} (0, \wh\zeta_1, \wh\zeta_1) \;= \nabla_{02,\delta,s}(p \circ \wh\zeta).
\end{align*}
The term \(p \circ \wh \nabla_{\delta,s}(0,\wh\zeta_1,\wh\zeta_1)\) in the third quantity vanishes since the subbundle \(\{0\} \times \rT_{(\wh w_{\delta,1}, \wh w_{\delta,1})} \Delta_{M_1}\) is preserved under parallel transport by \(\wh \nabla_{\delta,s}\).
\end{proof}

\noindent We will use the connections \(\nabla_\delta\) just constructed throughout the rest of \S\ref{ss:estimate}.
Due to the third property in Lemma~\ref{lem:conn}, it is unambiguous to drop the subscript and refer to \(\nabla_\delta\) simply as \(\nabla\).
Note that this pair of connections induce connections on the pullbacks by \(u_{02,\delta}\) or \(\wh u_\delta\) of any tensor bundle of \(\rT M_{02}\) or \(\rT M_{0211}\) in a canonical way.

Before we state the elliptic estimate Lemma~\ref{lem:masterestimate}, we need to define our function spaces and delbar operators.

\begin{definition} \label{def:functionspaces}
Fix \(r \in (0,\rho)\), \(\delta > 0\), and \(k \geq 2\).
Define the space of sections \(\mathbf{\Gamma_{u_\delta}^k}(\bQ_{\delta,r})\) and the norms \(\mathbf{ \| - \|_{H^k(\bQ_{\delta,r})}, \| - \|_{\wt H^k(\bQ_{\delta,r})} }\) as follows.
\begin{itemize}
\item Define \(\mathbf{\Gamma_{u_\delta}^k}(\bQ_{\delta,r})\) by: \begin{align*}
\Gamma_{u_\delta}^k(\bQ_{\delta,r})
:= \left\{ \left.\left(
\begin{aligned}
\xi_{02}\in H^k(Q_{02,\delta,r},u_{02,\delta}^* \rT M_{02}), \\
\wh\xi\in H^k(\wh Q_{\delta,r}, \wh u_\delta^* \rT M_{0211})
\end{aligned} 
\right) \right|
\eqref{eq:linearizedBCs}
\right\} ,
\end{align*}
where \eqref{eq:linearizedBCs} denotes the following linearized boundary conditions: \begin{align} \label{eq:linearizedBCs}
(\xi_{02}(s,0), \wh \xi(s,0)) \in \rT \Delta_{M_{02}} \times \rT \Delta_{M_1}, \qquad \wh\xi(s,\delta) \in \rT(L_{01} \times L_{12})^T \qquad \forall \: s \in (-r,r).
\end{align}

\item Define two norms \( \mathbf{ \| - \|_{H^k(\bQ_{\delta,r})}, \| - \|_{\wt H^k(\bQ_{\delta,r}) } } \) on \(\Gamma^k_{u_\delta}\) by: \begin{align*}
\| (\xi_{02}, \wh \xi) \|_{H^k(\bQ_{\delta,r})}^2 &:= \| \xi_{02} \|_{H^k(Q_{02,r},u_{02,\delta}^*\rT M_{02})}^2 + \|\wh \xi\|_{H^k(\wh Q_{\delta,r},\wh u_\delta^*\rT M_{0211})}^2, \\
\| (\xi_{02}, \wh \xi) \|_{\wt H^k(\bQ_{\delta,r})}^2 &:= \|(\xi_{02},\wh\xi)\|_{H^k(\bQ_{\delta,r})}^2 + \sum_{l=0}^{k-2} \|(\nabla^l \xi_{02}, \nabla^l \wh \xi)\|_{\cC^0H^1(\bQ_{\delta,r})}^2 \\
&:= \|(\xi_{02},\wh\xi)\|_{H^k(\bQ_{\delta,r})}^2 + \sum_{l=0}^{k-2} \Bigl( \sup_{t \in [0,r-2\delta) } \| \nabla_{02}^l\xi_{02}( -, t )\|_{H^1((-r, r),u_{02,\delta}(-,t)^*\rT M_{02})}^2 \\
&\hspace{2.5in} + \sup_{t \in [0, \delta]} \| \wh \nabla^l\wh \xi( -, t ) \|_{H^1((-r,r),\wh u_\delta(-,t)^*\rT M_{0211})}^2 \Bigr) .
\end{align*}
\end{itemize}
\end{definition}
\noindent Note that \(\|-\|_{\wt H^k(\bQ_{\delta,r})}\) is a well-defined norm on \(\Gamma^k_{u_\delta}(\bQ_{\delta,r})\) due to the embedding \(H^1 \hookrightarrow \cC^0\) for one-dimensional domains.
However, the constant in the bound \(\|-\|_{\wt H^k(\bQ_{\delta,r})} \leq C(\delta,r)\|-\|_{H^k(\bQ_{\delta,r})}\) is \(\delta\)-dependent.

In \cite{isom}, Wehrheim--Woodward introduced an exponential map with quadratic corrections, which allowed them to treat the Lagrangian boundary conditions as totally geodesic.
Wehrheim--Woodward assumed the composition \(L_{01} \circ L_{12}\) to be embedded, but their construction of the corrected exponential map only used the immersedness of that composition.
We may therefore import their corrected exponential map into our setting:

\begin{definition} \label{def:exp}
Given \(r > 0\) and \(\delta > 0\), define the \textbf{corrected exponential map \(\mathbf{e_{u_\delta}}\)} and its {\bf linearization \(\mathbf{\rd e_{u_\delta}}\)} and {\bf\(\mathbf{s}\)- and \(\mathbf{t}\)-derivatives as follows}.
\begin{itemize}
\item Let \(e_{u_\delta} = (e_{u_{02,\delta}}, e_{\wh u_\delta})\) be the pair of maps defined in \cite[Lemma 3.1.2]{isom}; \(e_{u_\delta}\) sends \(\zeta \in \Gamma^2_{u_\delta}(\bQ_{\delta,r})\) with \(\|\zeta\|_{\cC^0(\bQ_{\delta,r})}\) sufficiently small to a pair of maps \(e_{u_\delta}(\zeta) = (e_{u_{02,\delta}}(\zeta_{02}), e_{\wh u_\delta}(\wh \zeta)) \) satisfying \eqref{eq:foldedmaps}.

\item For \(p_{02} \in u_{02,\delta}^*\rT M_{02}|_{(s,t)}\), \(\rd e_{u_{02,\delta}}(p_{02})\colon u_{02,\delta}^*\rT M_{02}|_{(s,t)} \to \rT_{e_{u_{02,\delta}}(p_{02})}M_{02}\) is defined by including the fiber \(u_{02,\delta}^*\rT M_{02}|_{(s,t)}\) into \(\rT_{p_{02}}u_{02,\delta}^*\rT M_{02}\) as the vertical vectors, then postcomposing with the tangent map \(\rT (e_{u_{02,\delta}})_{p_{02}}\colon \rT_{p_{02}}u_{02,\delta}^*\rT M_{02} \to \rT_{e_{u_{02,\delta}}(p_{02})} M_{02}\).
The linearization \(\rd e_{\wh u_\delta}(\wh p)\) is defined analogously.

\item For \(p_{02} \in u_{02,\delta}^*\rT M_{02}|_{(s,t)}\), define \(\mathbf{\rD_s e_{w_{02}}(p_{02}) \in \rT_{e_{w_{02}}(p_{02})}M_{02} }\) to be the vector gotten by choosing a flat section \(\sigma\) of \(w_{02}^*\rT M_{02}|_{(s-\eps,s+\eps)\times\{t\}}\) for \(\eps\) small, then setting \(\rD_s e_{w_{02}}(p_{02}) := \rT_s (e_w(\sigma)) (\partial_s)\).
The derivatives \(\rD_t e_{w_{02}}(p_{02}), \rD_s e_{\wh w}(\wh p), \rD_t e_{\wh w}(\wh p)\) are defined analogously, and each of these derivatives depends smoothly on the argument \(p_{02}\) or \(\wh p\).
\end{itemize}
\end{definition}

\noindent This exponential map will allow us to define fiberwise complex structures in the following, which are parametrized by vector fields rather than by maps.

In the following definition of the linear delbar operator, we must go into coordinates.
Fix \(\delta > 0\) and a coherent collection \(\bj = ((j_0, j_2), (j_0', j_2', j_1', j_1))\) of complex structures on \(\ol\bQ_{\delta,\rho}\).
Then \(\bj\) induces via \eqref{eq:Cscoordsdef} two pairs of endomorphisms \(A = (A_{02}, \wh A)\), \(C = (C_{02}, \wh C)\) of \(u_{02,\delta}^*\rT M_{02}, \wh u_\delta^*\rT M_{0211}\), with \(C_{02}, \wh C\) defined as follows and \(A_{02}, \wh A\) defined in analogous fashion: \begin{gather}
C_{02}(s,t)\colon \rT_{u_{02,\delta}(s,t)}M_{02} \to \rT_{u_{02,\delta}(s,t)}M_{02}, \qquad (v_0, v_2) \mapsto (c_0(s,t)v_0, c_2(s,t)v_2), \label{eq:endos} \\
\wh C(s,t)\colon \rT_{\wh u_\delta(s,t)}M_{0211} \to \rT_{\wh u_\delta(s,t)}M_{0211}, \hspace{0.1in} (v_0', v_2', v_1', v_1) \mapsto (c_0'(s,t)v_0', c_2'(s,t)v_2', c_1'(s,t)v_1', c_1(s,t)v_1). \nonumber
\end{gather}
Note that the conditions \eqref{eq:Cscoordconds} (which are equivalent to the coherence conditions \eqref{eq:coherentCs0}, \eqref{eq:coherentCsdelta}) imply that for any \(s \in [-\rho,\rho]\), the endomorphisms \begin{align*}
\wh C(s,\delta), \qquad C_{02}(s,0) \times (\wh C|_{(u_0',u_2')^*\rT M_{02}})(s,0), \qquad (\wh C|_{(u_1',u_1)^*\rT M_{11}})(s,0)
\end{align*}
are scalar multiples of the identity; we will use this fact later in \S\ref{ss:estimate}.
In addition, the reader may find it helpful to note that in the case that $\bj$ is the standard collection $\bi$, $A_{02}$ and $\wh A$ are zero and $C_{02}$ and $\wh C$ are identity operators.

\begin{definition} \label{def:operators}
For \(\delta > 0\), \(r>0\), \(k \geq 2\), a coherent collection \(\bj\) of complex structures and a coherent pair of almost complex structures \(\bJ\) on \(\bQ_{\delta,r}\), and \(\xi \in \Gamma^2_{u_\delta}(\bQ_{\delta,r})\), define the \textbf{linear delbar operator \(\mathbf{\D_\xi}\)} to be the following map from \(H^1(Q_{02,\delta,r},u_{02,\delta}^*\rT M_{02}) \times H^1(\wh Q_{\delta,r},\wh u_\delta^*\rT M_{0211})\) to \(H^0(u_{02,\delta}^*\rT M_{02}) \times H^0(\wh u_\delta^*\rT M_{0211})\):
\begin{align*}
\D_\xi\zeta &:= A\nabla_s\zeta + C\nabla_t\zeta - \bJ(\xi)\nabla_s\zeta \\
&:= \bigl(A_{02}\nabla_{02,s}\zeta_{02} + C_{02}\nabla_{02,t}\zeta_{02} - J_{02}(\xi_{02})\nabla_{02,s}\zeta_{02}, \wh A\wh\nabla_s\wh \zeta + \wh C\wh\nabla_t\wh \zeta - \wh J(\wh \xi)\wh \nabla_s\wh \zeta\bigr),
\end{align*}
where \(\mathbf{J(\xi)}\) is the {\bf pulled-back complex structure} \begin{align*}
\bJ(\xi)(s,t) &:= \rd e_{u_\delta}(\xi(s,t))^{-1} \bJ(s,t,e_{u_\delta}(\xi(s,t))) \rd e_{u_\delta}(\xi(s,t)) \\
&:= \bigl(\rd e_{u_{02,\delta}}(\xi_{02}(s,t))^{-1} J_{02}(s,t,e_{u_{02,\delta}}(\xi_{02}(s,t))) \rd e_{u_{02,\delta}}(\xi_{02}(s,t)), \\
&\hspace{2.85in} \rd e_{\wh u_\delta}(\wh \xi(s,t))^{-1} \wh J(s,t,e_{\wh u_\delta}(\wh\xi(s,t))) \rd e_{\wh u_\delta}(\wh \xi(s,t))\bigr).
\end{align*}
\end{definition}

\noindent If \(\zeta = (\zeta_{02}, \wh \zeta)\) is a pair of sections in \(\Gamma^2_{u_\delta}(\bQ_{\delta,r})\), we can write \(\partial_s(e_w(\zeta))\) and \(\partial_t(e_w(\zeta)\) in terms of \(\rd e_{u_\delta}, \rD_s e_{u_\delta}, \rD_t e_{u_\delta}\): \begin{align}
\partial_s(e_{u_\delta}(\zeta)) &:= (\partial_s(e_{u_{02,\delta}}(\zeta_{02})), \partial_s(e_{\wh u_\delta}(\wh\zeta))) := ( \rd e_{u_{02,\delta}}(\zeta_{02})(\nabla_{02,s}\zeta_{02}) \nonumber \\
&\hspace{2.75in} + \rD_s e_{u_{02,\delta}}(\zeta_{02}), \rd e_{\wh u_\delta}(\wh \zeta)(\wh\nabla_s\wh \zeta) + \rD_s e_{\wh u_\delta}(\wh \zeta) ), \label{eq:expdecomp} \\
\partial_t(e_{u_\delta}(\zeta)) &:= (\partial_t(e_{u_{02,\delta}}(\zeta_{02})), \partial_t(e_{\wh u_\delta}(\wh\zeta))) := ( \rd e_{u_{02,\delta}}(\zeta_{02})(\nabla_{02,t}\zeta_{02}) \nonumber \\
&\hspace{2.75in} + \rD_t e_{u_{02,\delta}}(\zeta_{02}), \rd e_{\wh u_\delta}(\wh\zeta)(\wh \nabla_t\wh \zeta) + \rD_t e_{\wh u_\delta}(\wh \zeta) ). \nonumber
\end{align}
This decomposition allows us to relate the delbar operator \(\dbar_{\bJ,\bj}\) from \eqref{eq:delbar} with the linear delbar operator \(\D_\xi\) just defined: \begin{align} \label{eq:dbarvsD}
\dbar_{\bJ,\bj}(e_{u_\delta}(\zeta)) &= A\partial_s(e_{u_\delta}(\zeta)) + C\partial_t(e_{u_\delta}(\zeta)) - \bJ(s,t,e_{u_\delta}(\zeta))\partial_s(e_{u_\delta}(\zeta)) \nonumber \\
&= \rd e_{u_\delta}(\zeta)\bigl ( A\nabla_s\zeta + C\nabla_t\zeta - \rd e_{u_\delta}(\zeta)^{-1} \bJ(s,t,e_w(\zeta)) \rd e_{u_\delta}(\zeta) \nabla_s\zeta \bigr ) \\
&\hspace{1.8in} + \bigl ( A \, \rD_s e_{u_\delta}(\zeta) + C \, \rD_t e_{u_\delta}(\zeta) - \bJ(s,t,e_{u_\delta}(\zeta)) \rD_se_{u_\delta}(\zeta) \bigr ) \nonumber \\
&=: \rd e_{u_\delta}(\zeta) \D_\zeta \zeta + F(\zeta). \nonumber
\end{align}
The inhomogeneous term \(F\) depends smoothly on \(\zeta\), which is crucial for the proof of Theorem~\ref{thm:nonfoldedstripshrink}.

The following is the main result of \S\ref{ss:estimate}.
It generalizes \cite[Lemma 3.2.1]{isom}, which bounds the \(H^1\)-norm of \(\zeta\) when the domain complex structure is standard.

\begin{lemma} \label{lem:masterestimate}
There is a constant \(\eps > 0\) and for every \(C_0 > 0\), \(k \geq 0\), and \(r_1,r_2\) with \(0<r_1<r_2<\rho\) there is a constant \(C_1\) such that the inequality \begin{align} \label{eq:masterestimate}
\| \zeta\|_{\wt H^{k+1}(\bQ_{\delta,r_1})} \leq C_1 \bigl( \|\D_\zeta\zeta\|_{\wt H^k(\bQ_{\delta,r_2})} + \|\zeta\|_{H^0(\bQ_{\delta,r_2})} \bigr)
\end{align}
holds for any choice of \(\delta \in (0, r_1/4]\), a coherent collection \(\bj\) of complex structures on \(\ol \bQ_{\delta,\rho}\) with \(\|\bj - \bi\|_{\cC^0} \leq \eps\) and \(\|\bj - \bi\|_{\cC^{\max\{k,1\}}} \leq C_0\), a coherent pair \(\bJ\) of almost complex structures on \(\ol \bQ_{\delta,\rho}\) which are contained in a \(\cC^{\max\{k,1\}}\)-ball of radius \(C_0\) and which induce by \eqref{eq:metrics} metrics whose pairwise constants of equivalence are bounded above by \(C_0\), and a pair of sections \(\zeta \in \Gamma_{u_\delta}^{k+2}(\bQ_{\delta,r_2})\) with \(\|\zeta\|_{\cC^0} \leq \eps\), \(\|\zeta\|_{\cC^1} \leq C_0\), and \(\|\zeta\|_{\wt H^k(\bQ_{\delta,r_2})} \leq C_0\).
\end{lemma}

We begin by establishing \(\delta\)-independent Sobolev estimates for elements of \(\Gamma^k_{u_\delta}(\bQ_{\delta,r})\).

\begin{lemma} \label{lem:improvedsobolev}
Fix \(C_0 > 0\), \(k \geq 0\), and \(r_1,r_2\) with \(0<r_1<r_2<\rho\).
Then there is a constant \(C_1\) and a polynomial \(P\) such that the inequality \begin{align} \label{eq:improvedsobolev}
&\| \nabla^k \zeta \|_{\cC^0H^1(\bQ_{\delta,r}) } \leq C_1 \bigl( \|\zeta\|_{H^{k+2}(\bQ_{\delta,r}) } + \|\nabla^{k-1}\D_\xi\zeta\|_{\cC^0H^1(\bQ_{\delta,r})} \bigr) \nonumber \\
&\hspace{1.45in} + P \left(\sum_{l=1}^{k-1} \|\nabla^l\xi\|_{\cC^0H^1(\bQ_{\delta,r}) }\right) \left(\|\zeta\|_{H^{k+1}(\bQ_{\delta,r}) } + \sum_{l=0}^{k-2} \|\nabla^l\D_\xi\zeta\|_{\cC^0H^1(\bQ_{\delta,r}) }\right)
\end{align}
(where the term \(\|\nabla^{k-1}\D_\xi\zeta\|_{\cC^0H^1(\bQ_{\delta,r}) }\) is to be omitted when \(k=0\)) holds for any choice of \(\delta \in (0, r_1/4]\), \(r \in [r_1, r_2]\), a coherent collection \(\bj\) of complex structures on \(\ol \bQ_{\delta,\rho}\) with \(\|\bj - \bi\|_{\cC^k} \leq C_0\), a coherent pair \(\bJ\) of compatible almost complex structures on \(\bQ_{\delta,\rho}\) which are contained in a \(\cC^k\)-ball of radius \(C_0\) and which induce by \eqref{eq:metrics} metrics whose pairwise constants of equivalence are bounded above by \(C_0\), and pairs of sections \(\zeta, \xi \in \Gamma_{u_\delta}^{k+2}(\bQ_{\delta,r})\) with \(\|\xi\|_{\cC^1} \leq C_0\).
\end{lemma}

\noindent Here is the idea of the proof: \cite[Lemma~3.1.4]{isom} is a uniform Sobolev inequality for sections \(\zeta\) satisfying the linearized boundary conditions.
Since the special connection constructed in Lemma~\ref{lem:conn} preserves the linearized boundary conditions, \cite[Lemma~3.1.4]{isom} immediately gives a bound on \(\| \nabla_s^k\zeta \|_{\cC^0H^1(\bQ_{\delta,r}) }\).
To derive a bound on \(\| \nabla_\alpha\zeta \|_{\cC^0H^1(\bQ_{\delta,r}) }\) for \(\alpha \in \{s, t\}^k\), we trade indices using the operator \(\D_\xi\).

\begin{proof}
We prove this lemma in two steps: first, we prove a slightly different inequality, which has terms of the form \(\|\nabla^l \zeta\|_{\cC^0H^1}\) on the right-hand side.
Then, we prove the desired inequality by inductively removing these unwanted terms.

Throughout this proof, \(C_1\) and \(P\) will denote a \(\delta\)-independent constant and \(\delta\)-independent polynomial that may change from line to line.

\begin{step1}
We prove the following inequality:
\begin{align} \label{eq:sobolev}
\|\nabla^k\zeta\|_{\cC^0H^1} &\leq C_1\left( \|\zeta\|_{H^{k+2}} + \|\nabla^{k-1} \D_\xi \zeta \|_{\cC^0H^1} + P \left(\sum_{l = 1}^{k-1} \|\nabla^l \xi\|_{\cC^0H^1} \right) \cdot \sum_{l = 0}^{k - 1} \| \nabla^l\zeta \|_{\cC^0H^1}\right).
\end{align}
\end{step1}

\noindent We begin by proving the \(k = 0\) case of \eqref{eq:sobolev}, which is essentially a consequence of \cite[Lemma 3.1.4]{isom}.
One modification must be made to that lemma: we must relax the hypothesis that the composition \(L_{01} \circ L_{12}\) is embedded to the hypothesis that this composition is immersed.
To make this modification, change the proof of \cite[Lemma 3.1.4]{isom} like so: instead of using \cite[Lemma 3.1.3(c)]{isom}, use the fact that for \(\wh\xi = (\xi_{02}', \xi_1', \xi_1) \in \cC^\infty((-r,r), \ol u^*TM_{0211})\), \begin{align*}
\|\wh \xi\|_{H^1((-r,r))} \leq C_1\bigl(\|\xi'_{02}\|_{H^1((-r,r))} + \|\xi_1' - \xi_1\|_{H^1((-r,r))} + \|\pi_{0211}^\perp\wh\xi\|_{H^1((-r,r))}\bigr),
\end{align*}
where \(\pi_{0211}^\perp\) is the projection onto the orthogonal complement of the tangent space of \((L_{01} \times L_{12})^T\).
This inequality follows from the pointwise estimate \(|\wh \xi| \leq C(|\xi_{02}'| + |\xi_1' - \xi_1| + |\pi_{0211}^\perp\wh\xi|)\),
which can be proved like \cite[Lemma 3.1.3b]{isom}.

Next, fix \(k \geq 1\); let us prove \eqref{eq:sobolev} for this \(k\).
Let \(\zeta,\xi\) be sections in \(\Gamma^{k+2}_{u_\delta} \), and assume that the other hypotheses of the lemma are satisfied.
We will show that for every tuple \(\alpha = (\alpha_1, \ldots, \alpha_k) \in \{s, t\}^k\), there is a polynomial \(P_\alpha\) so that the following inequality holds:
\begin{align}
\label{eq:induct_bound}
\| \nabla_\alpha \zeta \|_{\cC^0H^1} &\leq C_1\left(\|\zeta\|_{H^{k+2}} + \|\nabla^{k-1} \D_\xi \zeta \|_{\cC^0H^1} + P_\alpha \left( \sum_{l = 1}^{k-1} \|\nabla^l \xi\|_{\cC^0H^1} \right) \cdot \sum_{l = 0}^{k - 1} \| \nabla^l\zeta \|_{\cC^0H^1}\right).
\end{align}
We prove this by induction on \(n_t(\alpha) := \#\{m \in [1, k] \: | \: \alpha_m = t \} \).
\begin{enumlist}
\item[ \( \boldsymbol{n_t(\alpha) = 0.} \) ] If \(\alpha = (s, \ldots, s)\), then since the special connection we have constructed preserves the boundary conditions of \(\Gamma^{k+2}_{u_\delta}\), the desired inequality follows immediately from the \(k=0\) case of the current lemma: \(\| \nabla_s^k \zeta \|_{\cC^0H^1} \leq C_1\| \nabla_s^k \zeta \|_{H^2}\).

\item[ \( \boldsymbol{n_t(\alpha) \in [1, k]. } \) ] Let us prove the inductive step (i.e.\ that there is a polynomial $P_\alpha$ for which \eqref{eq:induct_bound} holds) for some \(n_t(\alpha) \in [1, k]\).
Write \(\alpha = (\alpha', \alpha_m = t, s, \ldots, s)\).
Using the assumed bound on \(\bj\), we estimate: \begin{align*}
\|\nabla_\alpha \zeta \|_{\cC^0H^1} &= \|\nabla_{\alpha'}(C^{-1}( \D_\xi(\nabla_s^{k - m}\zeta) - (A - \bJ(\xi))\nabla_s^{k-m+1}\zeta) ) \|_{\cC^0H^1} \\
&\leq C_1\Bigl( \|\nabla_{\alpha'}\D_\xi(\nabla_s^{k-m}\zeta)\|_{\cC^0H^1} + \|\nabla_{\alpha'}\nabla_s^{k-m+1}\zeta\|_{\cC^0H^1} + \|\nabla_{\alpha'}(\bJ(\xi)\nabla_s^{k-m+1}\zeta)\|_{\cC^0H^1} \\
&\hspace{0.5in} + \sum_{l = 0}^{m-2} \|\nabla^{k-m + l +1}\zeta\|_{\cC^0H^1} + \sum_{l=0}^{m-2} \|\nabla^l(\bJ(\xi)\nabla^{k-m+1}\zeta)\|_{\cC^0H^1}\Bigr).
\end{align*}
Let us bound separately the five terms in the last expression.
\begin{enumlist}
\item[ \(\boxed{ \|\nabla_{\alpha'} \D_\xi (\nabla_s^{k-m}\zeta) \|_{\cC^0H^1}.} \) ] We estimate: \begin{align*}
&\|\nabla_{\alpha'}\D_\xi(\nabla_s^{k-m}\zeta) \|_{\cC^0H^1} \\
&\hspace{0.4in}\leq \|\nabla_{\alpha'}\nabla_s^{k-m}\D_\xi\zeta\|_{\cC^0H^1} + \sum_{l = 0}^{k-m-1} \|\nabla_{\alpha'}\nabla_s^l(\partial_sA\nabla_s^{k-m-l}\zeta + \partial_sC\nabla_s^{k-m-l-1}\nabla_t\zeta)\|_{\cC^0H^1} \\
&\hspace{0.7in} + \sum_{l = 1}^{k-m} \|\nabla_{\alpha'}\nabla_s^l(\bJ(\xi))\nabla_s^{k-m-l+1}\zeta\|_{\cC^0H^1} + \sum_{l = 0}^{k-m-1} \|\nabla_{\alpha'}(C\nabla_s^l[\nabla_s, \nabla_t]\nabla_s^{k-m-l-1}\zeta)\|_{\cC^0H^1}.
\end{align*}
Let us bound each of the four terms on the right-hand side.
The first term on the right-hand side, \(\|\nabla_{\alpha'}\nabla_s^{k-m}\D_\xi\zeta\|_{\cC^0H^1}\), is bounded by \(\|\nabla^{k-1}\D_\xi\zeta\|_{\cC^0H^1}\).
Due to the assumed bound on \(\bj\), the term \(\sum_{l = 0}^{k-m-1} \|\nabla_{\alpha'}\nabla_s^l(\partial_sA\nabla_s^{k-m-l}\zeta + \partial_sC\nabla_s^{k-m-l-1}\nabla_t\zeta)\|_{\cC^0H^1}\) is bounded by a constant times \(\sum_{l=0}^{k-1} \|\nabla^l\zeta\|_{\cC^0H^1}\).
To bound the term \(\sum_{l = 1}^{k-m} \|\nabla_{\alpha'}\nabla_s^l(\bJ(\xi))\nabla_s^{k-m-l+1}\zeta\|_{\cC^0H^1}\), observe that the assumed bound on \(\bJ\) yields: \begin{align*}
\sum_{l = 1}^{k-m} \|\nabla_{\alpha'}\nabla_s^l(\bJ(\xi))\nabla_s^{k-m-l+1}\zeta\|_{\cC^0H^1} &\leq \sum_{\beta, \gamma \geq 0, \atop
\beta + \gamma = k-2} \|\nabla^{\beta+1}(\bJ(\xi))\nabla^{\gamma + 1}\zeta\|_{\cC^0H^1} \\
&\leq P\left(\sum_{l=1}^{k-1} \|\nabla^l\xi\|_{\cC^0H^1}\right) \cdot \sum_{l=1}^{k-1} \|\nabla^l\zeta\|_{\cC^0H^1}.
\end{align*}
(In the last inequality we have used the Banach algebra property of \(\cC^0H^1\).)  Finally, the curvature of \(\nabla\) is a tensor, so the term \(\sum_{l=0}^{k - m - 1} \|\nabla_{\alpha'}(C\nabla_s^l[\nabla_s, \nabla_t]\nabla_s^{k-m-l-1}\zeta)\|_{\cC^0H^1}\) can be bounded by a constant times \(\sum_{l = 0}^{k-2} \|\nabla^l\zeta\|_{\cC^0H^1}\).

\item[ \( \boxed{ \|\nabla_{\alpha'}\nabla_s^{k - m + 1} \zeta\|_{\cC^0H^1}.}\) ] By the inductive hypothesis, this term is bounded appropriately: \begin{align*}
\|\nabla_{\alpha'}\nabla_s^{k-m+1}\zeta\|_{\cC^0H^1} &\leq C_1\left( \|\zeta\|_{H^{k+2}} + \|\nabla^{k-1}\D_\xi\zeta\|_{\cC^0H^1} + P_{(\alpha',s,\ldots,s)}\left(\sum_{l = 1}^{k-1} \|\nabla^l \xi\|_{\cC^0H^1}\right) \times\right. \\
&\hspace{3.65in} \left. \times \sum_{l=0}^{k-1} \|\nabla^l\zeta\|_{\cC^0H^1}\right).
\end{align*}

\item[ \(\boxed{ \|\nabla_{\alpha'}(\bJ(\xi)\nabla_s^{k - m + 1} \zeta)\|_{\cC^0H^1}.}\) ] To bound this term, it suffices to bound \(\|\bJ(\xi)\nabla_{\alpha'}\nabla_s^{k-m+1}\zeta\|_{\cC^0H^1}\) \break and \(\|\nabla^{\beta+1}(\bJ(\xi))\nabla^{\gamma+1}\zeta\|_{\cC^0H^1}\) separately, where in the second term \(\beta\) and \(\gamma\) are nonnegative integers with \(\beta + \gamma = k-2\).
The quantity \(\|\bJ(\xi)\nabla_{\alpha'}\nabla_s^{k-m+1}\zeta\|_{\cC^0H^1}\) can be bounded using the Banach algebra property of \(\cC^0H^1\), the assumed \(\cC^1\)-bounds on \(\xi\), and the inductive hypothesis.
Using the Banach algebra property of \(\cC^0H^1\), the quantity \(\|\nabla^{\beta+1}(\bJ(\xi))\nabla^{\gamma+1}\zeta\|_{\cC^0H^1}\) can be bounded by \(P\left(\sum_{l = 1}^{k-1} \|\nabla^l\xi\|_{\cC^0H^1}\right) \cdot \sum_{l=1}^{k-1} \|\nabla^l\zeta\|_{\cC^0H^1}\).

\item[ \(\boxed{ \sum_{l = 0}^{m-2} \|\nabla^{k-m + l +1}\zeta\|_{\cC^0H^1}. }\)] This term is already bounded appropriately.

\item[ \(\boxed{ \sum_{l=0}^{m-2} \|\nabla^l(\bJ(\xi)\nabla^{k-m+1}\zeta)\|_{\cC^0H^1}. }\)] By the Banach algebra property of \(\cC^0H^1\), this term is bounded by \(P\left(\sum_{l=1}^{k-2} \|\nabla^l\xi\|_{\cC^0H^1}\right)\cdot\sum_{l=1}^{k-1} \|\nabla^l\zeta\|_{\cC^0H^1}\).
\end{enumlist}
\end{enumlist}

\noindent This establishes the inductive step, so we have proven \eqref{eq:sobolev} for all \(k \geq 0\).

\begin{step2}
We prove \eqref{eq:improvedsobolev} by induction on \(k\).
\end{step2}

\noindent As in Step 1, the \(k = 0\) case follows from \cite[Lemma 3.1.4]{isom}.
Next, say that \eqref{eq:improvedsobolev} holds up to, but not including, some \(k \geq 1\).
By \eqref{eq:sobolev}, we have: \begin{align*}
\|\nabla^k\zeta\|_{\cC^0H^1} &\leq C_1\left( \|\zeta\|_{H^{k+2}} + \|\nabla^{k-1} \D_\xi \zeta \|_{\cC^0H^1} + P\left( \sum_{l = 1}^{k-1} \|\nabla^l \xi\|_{\cC^0H^1} \right) \cdot \sum_{l = 0}^{k - 1} \| \nabla^l\zeta \|_{\cC^0H^1} \right) .
\end{align*}
Replacing the sum \(\sum_{l=0}^{k-1} \|\nabla^l\zeta\|_{\cC^0H^1}\) appearing in the last term using the inductive hypothesis finishes the inductive step.
\end{proof}

We now turn to the proof of Lemma~\ref{lem:masterestimate}.
Here is our strategy:
in Lemma~\ref{lem:k=0ellipticineq}, we bound \(\|\zeta\|_{H^1}\) in terms of \(\|\zeta\|_{H^0}\) and \(\|\D_\zeta\zeta\|_{H^0}\), for \(\zeta\) supported in \(\bQ_{\delta,r} \).
In Lemma~\ref{lem:ellipticineq}, we use Lemma~\ref{lem:k=0ellipticineq} to bound \(\|\eta\nabla^k\zeta\|_{H^1}\) in terms of \(\|\zeta\|_{\wt H^k}\) and \(\|\D_\zeta\zeta\|_{\wt H^k}\), where \(\eta\) is supported in \(Q_{02, \delta,r}\) and \(\zeta\) has arbitrary support.
Finally, we use Lemma~\ref{lem:ellipticineq} to prove Lemma~\ref{lem:masterestimate}.

\begin{lemma}[elliptic estimate for \(k = 0\) and \(\zeta\) compactly supported] \label{lem:k=0ellipticineq}
There is a constant \(\eps > 0\) and for every \(C_0 > 0\), \(k \geq 0\), and \(r_1, r_2\) with \(0 < r_1 < r_2 < \rho\) there is a constant \(C_1\) such that the inequality \begin{align} \label{eq:k=0ellipticineq}
\| \nabla\zeta\|_{H^0(\bQ_{\delta,r})} \leq C_1 \bigl( \|\D_\xi\zeta\|_{H^0(\bQ_{\delta, r})} + \|\zeta\|_{H^0(\bQ_{\delta,r})} \bigr)
\end{align}
holds for any choice of \(\delta \in (0, r_1/4]\), \(r \in [r_1, r_2]\), a coherent collection \(\bj\) of complex structures on \(\ol \bQ_{\delta,\rho}\) with \(\|\bj - \bi\|_{\cC^0} \leq \eps\) and \(\|\bj - \bi\|_{\cC^1} \leq C_0\), a coherent pair \(\bJ\) of almost complex structures on \(\ol\bQ_{\delta,\rho}\) which are contained in a \(\cC^1\)-ball of radius \(C_0\) and which induce by \eqref{eq:metrics} metrics whose pairwise constants of equivalence are bounded above by \(C_0\), and sections \(\zeta, \xi \in \Gamma_{u_\delta}^2(\bQ_{\delta,r})\) with \(\|\xi\|_{\cC^0} \leq \eps\), \(\|\xi\|_{\cC^1} \leq C_0\), and \(\supp \zeta_{02}\), \(\supp \wh\zeta\) compact subsets of \(Q_{02,\delta,r}\), \(\wh Q_{\delta,r}\).
\end{lemma}
  
\begin{proof}
Throughout this proof, \(C_1\) will denote a \(\delta\)-independent constant that may change from line to line, and \(A = (A_{02}, \wh A)\), \(C = (C_{02}, \wh C)\) will be the endomorphisms of \(u_{02,\delta}^*\rT M_{02}\) and \(\wh u_\delta^*\rT M_{0211}\) defined in \eqref{eq:endos}.

We begin by fixing convenient metrics on \(M_{02}\) and \(M_{0211}\) that will be used for the pointwise norms in the definition of the Sobolev norms.
Via \eqref{eq:metrics}, \(\bJ\) induces fiberwise metrics \(g_{02}, \wh g\) on \(u_{02,\delta}^*\rT M_{02}\) and \(\wh u_\delta^*\rT M_{0211}\).
In this proof, however, we will use the pullback metrics \(g_\xi = (g_{02,\xi}, \wh g_\xi)\) of \(g_{02}, \wh g\) under \(\rd e_{u_{02,\delta}}(\xi_{02})\), \(\rd e_{\wh u_\delta}(\wh\xi)\); note that \(g_\xi\) is \(\bJ(\xi)\)-invariant.
If we pick \(\eps > 0\) to be sufficiently small, then \(\d e_{u_\delta}(\xi)\) is \(\cC^0\)-close to the identity, and hence the induced norm
\(\| - \|_{\xi, H^k} \coloneqq \bigl(\int_{\bQ_{\delta,r}} |-|_\xi^2 \,\d s\d t\bigr)^{1/2}\)
on \(\Gamma^k_{u_\delta}(\bQ_{\delta,r })\) is equivalent to the standard norms \(\| - \|_{H^k} = \| - \|_{0, H^k}\).
(Here we have denoted $|-|_\xi \coloneqq g_\xi(-,-)^{1/2}$.)

With these metrics we calculate for \(\zeta\in \Gamma_{u_\delta}^2\) compactly supported and \(\xi \in \Gamma_{u_\delta}^2\) satisfying \(\|\xi\|_{\cC^0(\bQ_{\delta,r})} \leq \eps\) and \(\|\nabla\xi\|_{\cC^0(\bQ_{\delta,r})} \leq C_0\):
\begin{align} \label{eq:k=0ellipticineq1}
\| \D_\xi\zeta \|_{\xi, H^0 }^2 &=  
 \tint_{\bQ_{\delta,r}} \bigl( (|\nabla_s\zeta|_\xi^2 + |A\nabla_s\zeta|_\xi^2) + 2g_\xi(A \nabla_s \zeta, C\nabla_t \zeta ) + | C \nabla_t \zeta |_{\xi}^2 \bigr) \, \d s \d t \nonumber \\
 &\hspace{2in} + \tint_{\bQ_{\delta,r}} \bigl( g_\xi( C\nabla_s\zeta, \bJ(\xi)\nabla_t\zeta ) - g_\xi(C\nabla_t\zeta, \bJ(\xi)\nabla_s\zeta ) \bigr) \, \d s \d t .
\end{align}
Let us estimate the two integrals on the right-hand side separately.
We begin with the first integral: \begin{align} \label{eq:k=0ellipticineq2}
&\tint_{\bQ_{\delta,r}}\bigl( (|\nabla_s\zeta|_\xi^2 + |A\nabla_s\zeta|_\xi^2) + 2g_\xi(2A \nabla_s \zeta, \tfrac 1 2C\nabla_t \zeta ) + | C \nabla_t \zeta |_{\xi}^2 \bigr) \, \d s \d t \\
&\hspace{1.5in}\sr{\text{AM-GM}}{\geq} \tint_{\bQ_{\delta,r}}\bigl( (|\nabla_s\zeta|_\xi^2 - 3|A\nabla_s\zeta|_\xi^2) + \tfrac 3 4| C \nabla_t \zeta |_{\xi}^2 \bigr) \, \d s \d t \geq \tfrac 5 8\|\nabla\zeta\|_{\xi,H^0}^2, \nonumber
\end{align}
where the last inequality follows from the hypothesis \(\|\bj - \bi\| \leq \eps\) as long as \(\eps\) is chosen small enough.

To bound the second integral on the right-hand side of \eqref{eq:k=0ellipticineq1}, we first derive a convenient formula for its integrand:
\begin{align}
 &g_\xi( C\nabla_s\zeta, \bJ(\xi)\nabla_t\zeta ) - g_\xi(C\nabla_t\zeta, \bJ(\xi)\nabla_s\zeta )\label{eq:k=0ellipticineq3} \\
 &\hspace{0.5in} = \bigl( \partial_s( g_\xi(C \zeta, \bJ(\xi)\nabla_t\zeta) ) - (\nabla_s g_\xi)(C\zeta,\bJ(\xi)\nabla_t\zeta) - g_\xi((\nabla_sC)\zeta,\bJ(\xi)\nabla_t\zeta) \nonumber \\
 &\hspace{3in} - g_\xi(C\zeta, \nabla_s(\bJ(\xi))\nabla_t\zeta) - g_\xi(C\zeta,\bJ(\xi)\nabla_s\nabla_t\xi) \bigr) \nonumber \\
 &\hspace{1.5in} - \bigl( \partial_t( g_\xi(C\zeta, \bJ(\xi)\nabla_s\zeta) ) - \nabla_t(g_\xi)(C\zeta,\bJ(\xi)\nabla_s\zeta) - g_\xi((\nabla_tC)\zeta, \bJ(\xi)\nabla_s\zeta) \nonumber \\
&\hspace{1.85in} - g_\xi(C\zeta, \nabla_t(\bJ(\xi))\nabla_s\zeta) + g_\xi(C\zeta,\bJ(\xi)[\nabla_s,\nabla_t]\zeta) - g_\xi(C\zeta,\bJ(\xi)\nabla_s\nabla_t\xi) \bigr) \nonumber \\
 &\hspace{0.1in}= \bigl( \partial_s( g_\xi(C\zeta,\bJ(\xi)\nabla_t\zeta)) - \partial_t( g_\xi(C\zeta,\bJ(\xi)\nabla_s\zeta)) \bigr) - \bigl( (\nabla_sg_\xi)(C\zeta, \bJ(\xi)\nabla_t\zeta) - (\nabla_tg_\xi)(C\zeta, \bJ(\xi)\nabla_s\zeta) \bigr) \nonumber \\
 & \hspace{0.65in} - \bigl(g_\xi((\nabla_sC)\zeta,\bJ(\xi)\nabla_t\zeta) - g_\xi((\nabla_tC)\zeta,\bJ(\xi)\nabla_s\zeta)\bigr) - g_\xi(C\zeta, \nabla_s(\bJ(\xi))\nabla_t\zeta - \nabla_t(\bJ(\xi))\nabla_s\zeta) \nonumber \\
 &\hspace{4.8in} - g_\xi(C\zeta,\bJ(\xi)[\nabla_s,\nabla_t]\zeta). \nonumber
\end{align}
We can now use Green's formula and the assumed \(\cC^1\)-bounds on \(\bj\), \(\bJ\), and \(\xi\) to bound the second integral on the right-hand side of \eqref{eq:k=0ellipticineq2}: \begin{align}
&\tint_{\bQ_{\delta,r}} \bigl( g_\xi( C\nabla_s\zeta, \bJ(\xi)\nabla_t\zeta ) - g_\xi(C\nabla_t\zeta, \bJ(\xi)\nabla_s\zeta ) \bigr) \, \d s \d t \nonumber \\
&\hspace{0.3in}\sr{\mathclap{\eqref{eq:k=0ellipticineq3}}}{=} \tint_{(-r,r) \times \{0\}} g_\xi(C\zeta,\bJ(\xi)\nabla_s\zeta) \, \d s \d t - \tint_{(-r,r) \times \{\delta\}} \wh g_\xi(\wh C\wh\zeta, \wh \bJ(\wh\xi)\wh\nabla_s\wh\zeta) \, \d s \d t \label{eq:k=0ellipticineq4} \\
&\hspace{0.5in} - \tint_{\bQ_{\delta,r}} \bigl( (\nabla_sg_\xi)(C\zeta, \bJ(\xi)\nabla_t\zeta) - (\nabla_tg_\xi)(C\zeta, \bJ(\xi)\nabla_s\zeta) \bigr) \, \d s \d t \nonumber \\
&\hspace{0.5in} - \tint_{\bQ_{\delta,r}} \bigl(g_\xi((\nabla_sC)\zeta,\bJ(\xi)\nabla_t\zeta) - g_\xi((\nabla_tC)\zeta,\bJ(\xi)\nabla_s\zeta)\bigr) \, \d s \d t \nonumber \\
&\hspace{0.5in} - \tint_{\bQ_{\delta,r}} g_\xi(C\zeta, \nabla_s\bJ(\xi))\nabla_t\zeta - \nabla_t(\bJ(\xi))\nabla_s\zeta) \, \d s \d t - \tint_{\bQ_{\delta,r}} g_\xi(C\zeta,\bJ(\xi)[\nabla_s,\nabla_t]\zeta) \, \d s \d t \nonumber \\
&\hspace{0.3in}\geq -\tint_{\bQ_{\delta,r}} C_1|\zeta|_\xi ( |\zeta|_\xi + |\nabla\zeta|_\xi ) \, \d s \d t \nonumber \sr{\text{AM-GM}}{\geq} -\tfrac 1 2\|\nabla \zeta\|_{\xi,H^0}^2 - C_1\|\zeta\|_{\xi, H^0}^2, \nonumber
\end{align}
where in the first inequality we have eliminated the integrals over the \(t = 0\) and \(t = \delta\) boundary via the coherence condition on \(\bj\) and the fact that \(g_\xi(\zeta,\bJ(\xi)\nabla_s\zeta)|_{t=0}\) and \(\wh g_\xi(\wh\zeta, \wh J(\wh\xi)\wh\nabla_s\wh\zeta)|_{t=\delta}\) vanish.
Indeed, \(\wh g_\xi(\wh\zeta, \wh J(\wh\xi)\wh\nabla_s\wh\zeta)|_{t=\delta}\) vanishes by the Lagrangian boundary condition:
\begin{align*} 
\lan \wh\zeta , \wh J(\wh\xi)\wh\nabla_s\wh\zeta \ran_{\wh\xi} |_{t=\delta}
&= 
 \om_{0211}(\rd e_{\wh u_\delta}(\wh\xi)\wh\zeta, \wh J(e_{\wh u_\delta}(\wh\xi))^2 \rd e_{\wh u_\delta}(\wh\xi) \wh\nabla_s\wh\zeta ) |_{t=\delta} \\
 &= - \om_{0211}(\rd e_{\wh u_\delta}(\wh\xi)\wh\zeta, \d e_{\wh u_\delta}(\wh\xi) \wh\nabla_s\wh\zeta ) |_{t=\delta} \;=\; 0 ,
\end{align*}
where we crucially used the fact that both the exponential map \(\rd e_{\wh u_\delta}(\wh\xi)\) and the connection \(\wh\nabla\) preserve \(\rT(L_{01}\times L_{12})^T\).
The boundary term \(g_\xi(\zeta,\bJ(\xi)\nabla_s\zeta)|_{t=0}\) vanishes due to the facts that \(\rd e_{u_\delta}(\xi)\) preserves \(\rT\Delta_{M_{02}} \times \rT\Delta_{M_1}\), \(\nabla\) satisfies \(\nabla_{02,s}\zeta_{02}|_{t=0} = p \circ \wh\nabla_s\wh\zeta|_{t=0}\) for \(p\colon M_{0211} \to M_{02}\) the projection, and \(\omega_{02},\omega_{0211}\) satisfy \(\omega_{0211}|_{\rT M_{02} \times \rT \Delta_{M_1}} = -p^*\omega_{02}\):
\begin{align*}
 \lan \zeta , \bJ(\xi)\nabla_s\zeta \ran_{\zeta} |_{t=0}
 &= -\omega_{02}( \rd e_{u_{02,\delta}}(\xi_{02})\zeta_{02}, \rd e_{u_{02,\delta}}(\xi_{02})\nabla_{02,s}\zeta_{02} )|_{t=0} - \omega_{0211}( \rd e_{\wh u_\delta}(\wh \xi)\wh\zeta, \rd e_{\wh u_\delta}(\wh\xi)\wh\nabla_s\wh\xi)|_{t=0} \\
 &= -\omega_{02}( \rd e_{u_{02,\delta}}(\xi_{02})(p \circ \wh \zeta), \rd e_{u_{02,\delta}}(\xi_{02})(p \circ \wh\nabla_s\wh\zeta))|_{t=0} \\
 &\hspace{2.5in} + p^*\omega_{02}( \rd e_{\wh u_\delta}(\wh\xi)\wh\zeta, \rd e_{\wh u_\delta}(\wh\xi)\wh\nabla_s\wh\zeta)|_{t=0} = 0.
\end{align*}

Combining \eqref{eq:k=0ellipticineq1}, \eqref{eq:k=0ellipticineq2}, and \eqref{eq:k=0ellipticineq4} yields the following inequality: \begin{align*}
\| \D_\xi\zeta \|_{\xi,H^0}^2 \geq \tfrac 1 8\|\nabla\zeta\|_{\xi,H^0}^2 - C_1\|\zeta\|_{\xi,H^0}^2.
\end{align*}
Adding \(C_1\|\zeta\|^2_{\xi, H^0 }\) to both sides of this inequality and taking the square root of the result, we obtain: \begin{align*}
\|\nabla\zeta\|_{\xi, H^0 } \leq C_1(\|\D_\xi\zeta\|^2_{\xi, H^0 } + \|\zeta\|^2_{\xi, H^0 })^{1/2} \leq C_1(\|\D_\xi\zeta\|_{\xi, H^0 } + \|\zeta\|_{\xi, H^0 }).
\end{align*}
In this estimate, we may replace \(\|-\|_{\xi,H^0}\) with \(\|-\|_{H^0}\) by using the \(\delta\)-independent uniform equivalence of these norms, which yields \eqref{eq:k=0ellipticineq}.
\end{proof}

\begin{lemma}[elliptic estimate for \(k \geq 0\)] \label{lem:ellipticineq}
There is a constant \(\eps > 0\) and for every \(C_0 > 0\), \(k \geq 0\), and \(0 < r_1 < r_2 < \rho\) there is a constant \(C_1\) such that the inequality \begin{align} \label{eq:ellipticineq}
\| \eta\nabla^k \zeta\|_{H^1(\bQ_{\delta,r})} \leq C_1 \bigl( \|\D_\zeta\zeta\|_{\wt H^k(\bQ_{\delta,r})} + \|\zeta\|_{\wt H^k(\bQ_{\delta,r})} \bigr)
\end{align}
holds for any choice of \(\delta \in (0, r_1/4]\), \(r \in [r_1, r_2]\), a coherent collection \(\bj\) of complex structures on \(\ol \bQ_{\delta,\rho}\) with \(\|\bj - \bi\|_{\cC^0} \leq \eps\) and \(\|\bj - \bi\|_{\cC^{\max\{k,1\}}} \leq C_0\), a pair \(\bJ\) of compatible almost complex structures on \(\ol \bQ_{\delta,\rho}\) which are contained in a \(\cC^{\max\{k,1\}}\)-ball of radius \(C_0\) and which induce by \eqref{eq:metrics} metrics whose pairwise constants of equivalence are bounded above by \(C_0\), a pair of sections \(\zeta \in \Gamma_{u_\delta}^{k+2}(\bQ_{\delta,r})\) with \(\|\zeta\|_{\cC^0} \leq \eps\), \(\|\zeta\|_{\cC^1} \leq C_0\), and \(\|\zeta\|_{\wt H^k(\bQ_{\delta,r})} \leq C_0\), and a smooth function \(\eta\colon \ol Q_{02,\delta,r} \to \R\) with \(\|\eta\|_{\cC^{k+1}} \leq c_0\) and \(\supp \eta \subset Q_{02,\delta,r}\).
\end{lemma}

\begin{proof}
Throughout this proof, \(C_1\) will denote a \(\delta\)-independent constant and \(P\) will denote a \(\delta\)-independent polynomial, and both may change from line to line.

We break down the proof into several steps: in Step 1, we establish \eqref{eq:ellipticineq}, but with an extra term on the right-hand side.
In Step 2, we bound this extra term, using different arguments in the \(k \neq 3\) and \(k = 3\) cases.
In Step 3, we establish \eqref{eq:ellipticineq}.

\begin{step1a}
We prove the following inequality: \begin{align} \label{eq:ellipticineqstep1a}
\| \eta \nabla_\alpha \zeta \|_{H^1} \leq C_1 \Bigl( \|\D_\zeta\zeta\|_{H^k} + \|\zeta\|_{H^k} + \sum_{\beta \geq 1,\gamma \geq 0, \atop
\beta+\gamma=k} \| \eta \nabla^\beta(\bJ(\zeta)) \nabla^\gamma\nabla_s\zeta \|_{H^0} \Bigr)
\end{align}
for \(\alpha = \underbrace{(s,\ldots,s)}_k\).
\end{step1a}

\noindent Since the connection \(\nabla\) preserves the linearized boundary conditions and \(\eta\) is supported in \(Q_{02,\delta,r}\), we may estimate \(\| \eta\nabla_s^k\zeta\|_{H^1}\) using Lemma~\ref{lem:k=0ellipticineq}: \begin{align*}
\|\eta\nabla_s^k\zeta\|_{H^1 } &\leq C_1( \|\D_\zeta(\eta\nabla_s^k\zeta) \|_{H^0} + \|\eta\nabla_s^k\zeta\|_{H^0} ) \nonumber \\
&= C_1 \Bigl( \|\eta\nabla_s^k\zeta\|_{H^0} + \Bigl\| \eta \nabla_s^k\D_\zeta\zeta - \sum_{l=1}^k {k \choose l} \eta(\partial_s^l A \nabla_s^{k-l+1}\zeta + \partial_s^l C \nabla_s^{k-l}\nabla_t\zeta ) \\
&\hspace{1.5in} + \sum_{l=1}^k {k\choose l}\eta\nabla_s^l(\bJ(\zeta)) \nabla_s^{k-l+1}\zeta - \sum_{l=1}^k C\eta\nabla_s^{l-1} [\nabla_s,\nabla_t] \nabla_s^{k-l}\zeta \nonumber \\
&\hspace{3.35in} - (\partial_s\eta(A - \bJ(\zeta)) + C\partial_t\eta)\nabla_s^k\zeta\Bigr\|_{H^0} \Bigr) \nonumber \\
&\leq C_1\Bigl( \|\D_\zeta\zeta\|_{H^k} + \|\zeta\|_{H^k} + \sum_{\beta\geq1,\gamma \geq 0, \atop
\beta + \gamma = k} \bigl\| \eta \nabla^\beta(\bJ(\zeta))\nabla^\gamma\nabla_s\zeta \bigr\|_{H^0} \Bigr). \nonumber
\end{align*}

\begin{step1b}
We prove \eqref{eq:ellipticineqstep1a} for a general multiindex \(\alpha\) of length \(k\).
\end{step1b}

\noindent We establish Step 1b by induction on \(n_t(\alpha) := \{ \#m \in [1,k] \: | \: \alpha_m = t\}\).
Step 1a is the base case for this induction.
For the inductive step, fix \(\alpha\) with \(n_t(\alpha) \geq 1\), and write \(\alpha = (\alpha', \alpha_m = t, \underbrace{s,\ldots,s}_{k-m})\).
We estimate: \begin{align*}
\|\eta \nabla_\alpha \zeta\|_{H^1 } &= \|\eta\nabla_{\alpha'}(C^{-1}( \D_\zeta(\nabla_s^{k-m}\zeta) - (A - \bJ(\zeta))\nabla_s^{k-m+1}\zeta)) \|_{H^1 } \\
&\hspace{-0.5in} \leq C_1( \|\zeta\|_{H^k } + \| \eta \nabla_{\alpha'}\D_\zeta(\nabla_s^{k-m}\zeta)\|_{H^1 } + \|\eta\nabla_{\alpha'}\nabla_s^{k-m+1}\zeta\|_{H^1} + \|\eta\nabla_{\alpha'}(\bJ(\zeta)\nabla_s^{k-m+1}\zeta)\|_{H^1} ) \\
&\hspace{-0.5in}= C_1 \Bigl( \|\zeta\|_{H^k} + \Bigl\| \eta \nabla_{\alpha'}\Bigl( \nabla_s^{k-m}\D_\zeta\zeta - \sum_{l=1}^{k-m} {{k-m} \choose l} (\partial_s^l A \nabla_s^{k-m-l+1}\zeta + \partial_s^l C \nabla_s^{k-m-l}\nabla_t\zeta ) \\
&\hspace{0.6in} + \sum_{l=1}^{k-m} {{k-m}\choose l} \nabla_s^l(\bJ(\zeta)) \nabla_s^{k-m-l+1}\zeta - \sum_{l=1}^{k-m} C \nabla_s^{l-1} [\nabla_s,\nabla_t] \nabla_s^{k-m-l}\zeta \Bigr) \Bigr\|_{H^1} \\
&\hspace{2.15in} + \|\eta\nabla_{\alpha'}\nabla_s^{k-m+1}\zeta\|_{H^1} + \|\eta\nabla_{\alpha'}(\bJ(\zeta)\nabla_s^{k-m+1}\zeta)\|_{H^1} \Bigr) \nonumber \\
&\hspace{-0.5in}\leq C_1 \Bigl( \|\D_\zeta\zeta\|_{H^k} + \|\zeta\|_{H^k} + \sum_{\beta\geq1,\gamma\geq0, \atop
\beta+\gamma=k} \| \eta \nabla^\beta(\bJ(\zeta)) \nabla^\gamma\nabla_s\zeta \|_{H^0} \Bigr),
\end{align*}
where in the last inequality we have used the inductive hypothesis to bound \(\|\eta\nabla_{\alpha'}\nabla_s^{k-m+1}\zeta\|_{H^1}\).

\begin{step2a}
In the \(k \neq 3\) case, we prove the following inequality: \begin{align} \label{eq:ellipticineqstep2a}
\sum_{\beta\geq1,\gamma\geq0, \atop
\beta+\gamma=k} \| \eta \nabla^\beta(\bJ(\zeta))\nabla^\gamma\nabla_s\zeta\|_{H^0} \leq C_1\|\zeta\|_{H^k}.
\end{align}
\end{step2a}

\noindent It follows from the assumption \(k \neq 3\) that if \(\beta, \gamma \geq 1\) satisfy \(\beta + \gamma = k+1\), then \(\min\{\beta, \gamma\} \leq \max\{k-2, 1\}\).
Furthermore, the assumption \(\|\zeta\|_{\wt H^k} \leq C_0\) implies the inequality \(\|\zeta\|_{\cC^{k-2}} \leq C_1\) by the embedding of \(H^1 \hookrightarrow \cC^0\) for one-dimensional domains whose lengths are bounded away from zero.
This, along with the assumed \(\cC^1\)-bound on \(\zeta\), yields \eqref{eq:ellipticineqstep2a} in the \(k \neq 3\) case.

\begin{step2b}
In the \(k=3\) case, we prove the following inequality: \begin{align} \label{eq:ellipticineqstep2b}
\sum_{\beta\geq 1,\gamma\geq 0, \atop
\beta+\gamma=3} \| \eta \nabla_\beta(\bJ(\zeta))\nabla_\gamma\nabla_s\zeta\|_{H^0} \leq C_1( \|\D_\zeta\zeta\|_{\wt H^3} + \| \zeta \|_{\wt H^3} + \delta^{1/2}\|\eta\nabla^3\zeta\|_{H^1} ).
\end{align}
\end{step2b}

\noindent The assumed \(\cC^1\)-bound on \(\zeta\) implies that the only term in the left-hand side of \eqref{eq:ellipticineqstep2b} that is not immediately bounded by \(C_1\|\zeta\|_{H^3}\) is \(\| \eta \nabla^2(\bJ(\zeta)) \nabla\nabla_s\zeta\|_{H^0}\).

Choose smooth maps \begin{gather*}
S, U\colon \wh u^*\rT M_{0211} \to \wh u^*\hom((\rT M_{0211})^{\otimes 2}, \rT M_{0211}), \quad T\colon \wh u^*\rT M_{0211} \to \wh u^*\hom((\rT M_{0211})^{\otimes 3}, \rT M_{0211}), \\
V\colon \wh u^*\rT M_{0211} \to u^*\hom(\rT M_{0211},\rT M_{0211})
\end{gather*}
so that the formula \begin{align} \label{eq:Jchainrule}
\wh\nabla^2(\wh J(\wh \zeta)) = S(\wh\zeta)(\wh\nabla^2\wh\zeta) + T(\wh\zeta)(\wh\nabla\wh\zeta, \wh\nabla\wh\zeta) + U(\wh\zeta)(\wh\nabla\zeta) + V(\wh\zeta)
\end{align}
holds, where the maps \(S, T, U, V\) preserve fibers but may not respect their linear structure.
Since \(\bJ\) is bounded in \(\cC^3\), \(S, T, U, V\) must be bounded in \(\cC^1\).
We may now use \eqref{eq:Jchainrule} to bound the hat-part of \( \|\eta\nabla^2(\bJ(\zeta))\nabla\nabla_s\zeta\|_{H^0}\): \begin{align} \label{eq:ellipticineqstep2bfirstineq}
\|\eta\wh\nabla^2(\wh J(\wh \zeta))\wh \nabla\wh \nabla_s\wh \zeta\|_{H^0} &\leq C_1( \|\wh\zeta\|_{H^2} + \| S(\wh\zeta)(\wh\nabla^2\wh\zeta)\wh\nabla\wh\nabla_s\wh\zeta\|_{H^0}) \nonumber \\
&= C_1 \bigl( \|\wh\zeta\|_{H^2} + \|\wh\nabla_s(S(\wh\zeta)(\eta\wh\nabla^2\wh\zeta)\wh\nabla\wh\zeta) - \wh\nabla_s(S(\wh\zeta)(\eta\wh\nabla^2\wh\zeta))\wh\nabla\wh\zeta \\
&\hspace{3in} + S(\wh\zeta)(\eta\wh\nabla^2)[\wh\nabla, \wh\nabla_s]\wh\zeta \|_{H^0} \bigr) \nonumber \\
&\leq C_1( \|\wh\zeta\|_{H^3} + \delta^{1/2}\| S(\zeta)(\eta\wh\nabla^2\wh\zeta)\wh\nabla\wh\zeta\|_{\cC^0H^1} ) \nonumber \\
&\leq C_1( \|\wh\zeta\|_{H^3} + \delta^{1/2}\| S(\zeta)(\eta\wh\nabla^2\wh\zeta)\|_{\cC^0H^1}\|\wh\nabla\wh\zeta\|_{\cC^0H^1} ), \nonumber
\end{align}
where in the last inequality we have used the \(\delta\)-independent Banach algebra property of \(\cC^0H^1\).
By Lemma~\ref{lem:improvedsobolev}, \(\|\wh\nabla\wh\zeta\|_{\cC^0H^1}\) is bounded by \(C_1(\|\D_\zeta\zeta\|_{\wt H^2} + \|\zeta\|_{H^3})\) and therefore by \(C_1\|\zeta\|_{\wt H^3}\); on the other hand, the \(\cC^1\)-bound on \(S\) and the \(\cC^1\)-bound on \(\zeta\) implies the inequality \(\| S(\wh \zeta)(\eta\wh\nabla^2\wh\zeta) \|_{\cC^0H^1} \leq C_1\|\eta\wh\nabla^2\wh\zeta\|_{\cC^0H^1}\).
Substituting these inequalities into \eqref{eq:ellipticineqstep2bfirstineq}, we obtain: \begin{align} \label{eq:ellipticineqstep2bsecondineq}
\| \eta \wh\nabla^2(\wh J(\wh \zeta)) \wh\nabla\wh\nabla_s\wh\zeta \|_{H^0} &\leq C_1( \|\zeta\|_{H^3} + \delta^{1/2} \|\zeta\|_{\wt H^3} \| \eta \nabla^2\zeta\|_{\cC^0H^1}) \leq C_1( \|\zeta\|_{H^3} + \delta^{1/2}\| \eta\nabla^2\zeta \|_{\cC^0H^1} ).
\end{align}
Next, we use Lemma~\ref{lem:improvedsobolev} to bound \(\| \eta \nabla^2 \zeta \|_{\cC^0H^1} \): \begin{align} \label{eq:ellipticineqstep2bthirdineq}
\| \eta\nabla^2\zeta \|_{\cC^0H^1} &\leq C_1( \|\zeta\|_{\wt H^3} + \|\nabla^2(\eta\zeta)\|_{\cC^0H^1}) \nonumber \\
&\leq C_1( \|\eta\zeta\|_{H^4} + \|\nabla\D_\zeta(\eta\zeta)\|_{\cC^0H^1} + \|\zeta\|_{\wt H^3}) + P( \|\nabla\zeta\|_{\cC^0H^1})(\|\zeta\|_{H^3_{\delta,\rho}} + \|\D_\zeta\zeta\|_{\cC^0H^1}) \\
&\leq C_1( \|\D_\zeta\zeta\|_{\wt H^3} + \|\zeta\|_{\wt H^3} + \|\eta\nabla^3\zeta\|_{H^1} ) + P( \|\zeta\|_{\wt H^3} )\|\zeta\|_{\wt H^3} \nonumber \\
&\leq C_1( \|\D_\zeta\zeta\|_{\wt H^3} + \|\zeta\|_{\wt H^3} + \|\eta\nabla^3\zeta\|_{H^1} ), \nonumber
\end{align}
where the last inequality follows from the assumed bound on \(\|\zeta\|_{\wt H^3}\).
Substituting \eqref{eq:ellipticineqstep2bthirdineq} into \eqref{eq:ellipticineqstep2bsecondineq}, we obtain: \begin{align} \label{eq:ellipticineqstep2bfourthineq}
\| \eta \wh\nabla^2(\wh J(\wh \zeta)) \wh\nabla\wh\nabla_s\wh\zeta \|_{H^0} &\leq C_1( \|\zeta\|_{H^3} + \delta^{1/2}( \|\D_\zeta\zeta\|_{\wt H^3} + \|\zeta\|_{\wt H^3} + \|\eta\nabla^3\zeta\|_{H^1} ) ) \\
&\leq \nonumber C_1( \|\D_\zeta\zeta\|_{\wt H^3} + \|\zeta\|_{\wt H^3} + \delta^{1/2}\|\eta\nabla^3\zeta\|_{H^1} ). \nonumber
\end{align}

To bound the \(02\)-part of \(\|\eta \nabla^2(\bJ(\zeta))\nabla\nabla_s\zeta\|_{H^0 }\), we use the the fact that the domains \(Q_{02,\delta,r}\) satisfy a uniform cone condition: \begin{align} \label{eq:ellipticineqstep2bfifthineq}
\|\eta \nabla_{02}^2(J_{02}(\zeta_{02}))\nabla_{02}\nabla_{02,s}\zeta_{02}\|_{H^0) } &\sr{\mathclap{\text{H\"{o}lder}}}{\leq} C_1\|\nabla_{02}^2(J_{02}(\zeta_{02}))\|_{L^4}\|\nabla_{02}^2\zeta\|_{ L^4 } \\
&\leq C_1(1 + \|\zeta\|_{H^3 })\|\zeta\|_{H^3 }, \nonumber
\end{align}
where the second inequality follows from the Sobolev embedding \(H^1 \hookrightarrow L^4\) for two-dimensional domains satisfying a cone condition.
Combining \eqref{eq:ellipticineqstep2bfourthineq} and \eqref{eq:ellipticineqstep2bfifthineq} and using the assumed bound on \(\|\zeta\|_{\wt H^3}\) yields the desired bound: \begin{align*}
\|\eta\nabla^2(\bJ(\zeta))\nabla\nabla_s\zeta \|_{H^0 } &\leq C_1( \|\D_\zeta\zeta\|_{\wt H^3} + \|\zeta\|_{\wt H^3 } + \delta^{1/2} \|\eta\nabla^3\zeta\|_{H^1 }).
\end{align*}

\begin{step3}
We prove Lemma~\ref{lem:ellipticineq}.
\end{step3}

The \(k \neq 3\) case of Lemma~\ref{lem:ellipticineq} is an immediate consequence of Steps 1b and 2a.

Toward the \(k = 3\) case of Lemma~\ref{lem:ellipticineq}, let us show that there exists \(\delta_0 \in (0, r_1]\) such that \eqref{eq:ellipticineq} holds for \(\delta \in (0, \delta_0]\).
Combining \eqref{eq:ellipticineqstep1a} and \eqref{eq:ellipticineqstep2b} yields the following inequality: \begin{align} \label{eq:ellipticineqstep3firstineq}
\| \eta\nabla^3\zeta \|_{H^1} \leq C_1( \|\D_\zeta\zeta\|_{\wt H^3} + \|\zeta\|_{\wt H^3} + \delta^{1/2}\| \eta\nabla^3\zeta \|_{H^1} ).
\end{align}
If we set \(\delta_0 := \min \{ (2C_1)^{-2} , r_1\}\), where \(C_1\) is the constant appearing in \eqref{eq:ellipticineqstep3firstineq}, then \eqref{eq:ellipticineqstep3firstineq} yields the uniform inequality \(\|\eta \nabla^3\zeta\|_{H^1} \leq C_1(\|\D_\zeta\zeta\|_{\wt H^3 } + \|\zeta\|_{\wt H^3 })\) for all \(\delta \in (0, \delta_0]\).

It remains to establish the \(k=3\) case of \eqref{eq:ellipticineq} for \(\delta \in [\delta_0,r_1]\).
To do so, we begin by bounding \( \| \nabla^2(\bJ(\zeta)) \nabla^2 \zeta\|_{ H^0}\), using the fact that the domains \(\bQ_{\delta,r}\) satisfy a uniform cone condition for \(\delta \in [\delta_0,r_1/4]\): \begin{align} \label{eq:ellipticineqstep3secondineq}
\| \nabla^2(\bJ(\zeta)) \nabla^2 \zeta\|_{ H^0} \sr{\text{H\"{o}lder}}{\leq} C_1\|\nabla^2(\bJ(\zeta))\|_{L^4}\|\nabla^2\zeta\|_{L^4} &\sr{\mathclap{\text{Sobolev}}}{\leq} C_1(1+\|\zeta\|_{H^{2,4} })\|\zeta\|_{H^{2,4} } \\
&\leq C_1(1 + \|\zeta\|_{H^3 })\|\zeta\|_{H^3 } \leq C_1\|\zeta\|_{H^3 }. \nonumber
\end{align}
Substituting \eqref{eq:ellipticineqstep3secondineq} into \eqref{eq:ellipticineqstep1a} yields the \(k=3\) case of \eqref{eq:ellipticineq} for \(\delta \in [\delta_0,r_1/4]\): \begin{align*}
\| \eta \nabla^3 \zeta \|_{H^1 } \leq C_1 \Bigl( \|\D_\zeta\zeta\|_{H^3 } + \|\zeta\|_{H^3 } + \sum_{\beta \geq 1,\gamma \geq 0, \atop
\beta+\gamma=3} \| \eta \nabla^\beta(\bJ(\zeta)) \nabla^\gamma\nabla_s\zeta \|_{H^0 } \Bigr) \leq C_1( \| \D_\zeta\zeta \|_{H^3 } + \|\zeta\|_{H^3 }).
\end{align*}
\end{proof}

\begin{proof}[Proof of Lemma~\ref{lem:masterestimate}]
Lemma~\ref{lem:masterestimate} follows immediately from Lemmata~\ref{lem:improvedsobolev} and \ref{lem:ellipticineq}.
Indeed, choose \(\eta\colon \ol Q_{02,\delta,r_2} \to \R\) to be a smooth function with \(\eta|_{\ol Q_{02,\delta,r_1}} \equiv 1\) and \(\supp \eta \subset Q_{02,\delta,r_2}\).
\(C_1\) and \(P\) will denote a \(\delta\)-independent constant and a \(\delta\)-independent polynomial that may change from line to line.
Lemma~\ref{lem:ellipticineq} yields a bound on \(\|\zeta\|_{H^{k+1}(\bQ_{\delta,r_1}) }\): \begin{align} \label{eq:masterestimatefirstineq}
\|\zeta\|_{H^{k+1}(\bQ_{\delta,r_1}) } \leq \|\eta\zeta\|_{H^{k+1}(\bQ_{\delta,r_2}) } \leq C_1\bigl( \|\zeta\|_{\wt H^k(\bQ_{\delta,r_2}) } + \|\D_\zeta\zeta\|_{\wt H^k(\bQ_{\delta,r_2}) }\bigr).
\end{align}
Lemma~\ref{lem:improvedsobolev} yields a bound on \(\sum_{l = 0}^{k-1} \|\nabla^l\zeta\|_{\cC^0H^1(\bQ_{\delta,r_1}) }\): \begin{align} \label{eq:masterestimatesecondineq}
\sum_{l = 0}^{k-1} \|\nabla^l\zeta\|_{\cC^0H^1(\bQ_{\delta,r_1}) } &\leq C_1\bigl( \|\zeta\|_{H^{k+1}(\bQ_{\delta,r_1})} + \|\D_\zeta\zeta\|_{\wt H^k(\bQ_{\delta,r_1})}\bigr) + \\
&\hspace{1in} + P\bigl(\|\zeta\|_{\wt H^k(\bQ_{\delta,r_1})}\bigr) \cdot \bigl( \|\zeta\|_{H^k(\bQ_{\delta,r_1})} + \|\D_\zeta\zeta\|_{\wt H^{k-1}(\bQ_{\delta,r_1})} \bigr) \\
&\sr{\mathclap{\eqref{eq:masterestimatefirstineq}}}{\leq} C_1\bigl( \|\D_\zeta\zeta\|_{\wt H^k(\bQ_{\delta,r_2})} + \|\zeta\|_{\wt H^k(\bQ_{\delta,r_2})}\bigr), \nonumber
\end{align}
where in the second inequality we have used the assumed bound on \(\|\zeta\|_{\wt H^k(\bQ_{\delta,r_1})}\).
Combining \eqref{eq:masterestimatefirstineq} and \eqref{eq:masterestimatesecondineq} yields \(\|\zeta\|_{\wt H^{k+1}(\bQ_{\delta,r_1}) } \leq C_1( \|\D_\zeta\zeta\|_{\wt H^k(\bQ_{\delta,r_2})} + \|\zeta\|_{\wt H^k(\bQ_{\delta,r_2})} )\), which can be used to inductively prove the desired inequality \eqref{eq:masterestimate}.
\end{proof}

We will not use the following proposition in this paper.
However, it will be used in \cite{b:thesis} to show that the linearized Cauchy--Riemann operator defines a Fredholm section.
 
\begin{proposition}[linear elliptic estimate for \(k = 2\)] \label{prop:k=2linearineq}
There is a constant \(\eps > 0\) and for every \(C_0 > 0\), \(k \geq 0\), and \(0 < r_1 < r_2 < \rho\) there is a constant \(C_1\) such that the inequality \begin{align*}
\| \zeta\|_{H^{k+1}(\bQ_{\delta,r_1})} \leq C_1 \bigl( \|\D_\xi\zeta\|_{H^k(\bQ_{\delta,r_2})} + \|\zeta\|_{H^0(\bQ_{\delta,r_2})} \bigr)
\end{align*}
holds for any choice of \(\delta \in (0,r_1/4]\), a coherent collection \(\bj\) of complex structures on \(\ol \bQ_{\delta,\rho}\) with \(\|\bj - \bi\|_{\cC^0} \leq \eps\) and \(\|\bj - \bi\|_{\cC^2} \leq C_0\), a pair \(\bJ\) of compatible almost complex structures on \(\ol \bQ_{\delta,\rho}\) which are contained in a \(\cC^2\)-ball of radius \(C_0\) and which induce by \eqref{eq:metrics} whose pairwise constants of equivalence are bounded above by \(C_0\), and two pairs of sections \(\zeta, \xi \in \Gamma_{u_\delta}^{k+2}(\bQ_{\delta,r_2})\) with \(\|\xi\|_{\cC^0} \leq \eps\) and \(\|\xi\|_{\cC^1} \leq C_0\).
\end{proposition}

\noindent The proof is an easier version of the proof of Lemma~\ref{lem:masterestimate}.

\subsection{Proof of Thm.~\ref{thm:nonfoldedstripshrink}}
\label{ss:compactness_proof}

Now that we have established the necessary definitions and estimates in \S\S\ref{ss:CsandACs}--\ref{ss:estimate}, we are finally ready to prove Thm.~\ref{thm:nonfoldedstripshrink}.

\begin{proof}[Proof of Theorem~\ref{thm:nonfoldedstripshrink}]
We divide the proof into steps: in Step 1, we show that the squiggly strip quilts converge \(\cC^0_\loc\) in a subsequence.
In Step 2, we upgrade this convergence to \(\cC^k_\loc\).
Finally, we prove in Step 3 that if the gradient satisfies a lower bound at a sequence of points with limit on the boundary, then at least one of \(v_0^\infty,v_2^\infty\) is nonconstant.
Throughout this proof, \(C_1\) will be a constant that may change from line to line.

\begin{step1}
After passing to a subsequence, \((v_0^\nu(t - \delta^\nu))\), \((v_1^\nu|_{t=0})\), \((v_2^\nu(t + \delta^\nu))\) converge \(\cC^0_\loc\) to a \((J_0^\infty,J_2^\infty,i)\)-holomorphic size-\(\rho\) degenerate strip quilt \((v_0^\infty,v_1^\infty,v_2^\infty)\) for \(L_{01} \times_{M_1} L_{12}\).
\end{step1}

\noindent The Arzel\`{a}--Ascoli theorem implies that there exist continuous maps \begin{align*}
v_0^\infty\colon (-\rho,\rho) \times (-\rho,0] \to M_0, \qquad v_1^\infty\colon (-\rho,\rho) \to M_1, \qquad v_2^\infty\colon (-\rho,\rho) \times [0,\rho) \to M_2
\end{align*}
such that after passing to a subsequence, \((v_0^\nu(s,t-\delta^\nu))\), \((v_1^\nu|_{t=0})\), \((v_2^\nu(s, t+\delta^\nu))\) converge \(\cC^0_\loc\) to \(v_0^\infty\), \(v_1^\infty\), \(v_2^\infty\).
Standard compactness for pseudoholomorphic curves (e.g.\ \cite[Theorem~B.4.2]{ms:jh}) implies that this convergence takes place in \(\cC^k_\loc\) on the interior (i.e.\ away from the line \(t = 0\)); in particular, \(v_0^\infty\) resp.\ \(v_2^\infty\) are \(J_0^\infty\)- resp.\ \(J_2^\infty\)-holomorphic on the interior, hence \(\cC^\infty\) by \cite[Theorem~B.4.1]{ms:jh}.
In fact, we claim that \(v_0^\infty\) and \(v_2^\infty\) are \(\cC^\infty\) on their full domains, and that they satisfy a generalized Lagrangian boundary condition in \(L_{01} \times_{M_1} L_{12}\) at \(t = 0\).

Denote by \(\ol v\) the map \begin{align*}
\ol v := (v_0^\infty(-,0), v_1^\infty(-), v_1^\infty(-), v_2^\infty(-,0))\colon (-\rho,\rho) \to M_0^- \times M_1 \times M_1^- \times M_2.
\end{align*}
To show that \(v_0^\infty\), \(v_2^\infty\) satisfy a generalized Lagrangian boundary condition in \(L_{01} \circ L_{12}\), we will show that for any \(s \in (-\rho,\rho)\), \(\ol v(s)\) lies in \(L_{01} \times_{M_1} L_{12}\).
The containment \(\ol v(s) \in M_0 \times \Delta_{M_1} \times M_2\) is clear.
To show the containment \(\ol v(s) \in L_{01} \times L_{12}\), we will show that \((v_0^\infty(s,0), v_1^\infty(s))\) lies in \(L_{01}\); the proof that \((v_1^\infty(s), v_2^\infty(s,0))\) lies in \(L_{12}\) is analogous.
Since \((v_0^\nu(s,-\delta^\nu), v_1^\nu(s,-\delta^\nu))\) lies in \(L_{01}\), and since \((v_1^\nu|_{t = 0})\) converges \(\cC^0_\loc\) to \(v_1^\infty\), it suffices to show that the distances \(d(v_1^\nu(s,-\delta^\nu), v_1^\nu(s, 0))\) converge to zero.
This follows from the uniform gradient bound on \((v_1^\nu)\) and the convergence of \(\delta^\nu\) to zero.

Let us show that \(v_0^\infty\) and \(v_2^\infty\) are \(\cC^\infty\).
We have already concluded that these maps are \(\cC^\infty\) on the interior, so it only remains to show that they are \(\cC^\infty\) at the boundary points, w.l.o.g.\ at \((0,0)\).
For that purpose we choose a neighborhood \(U \subset L_{01} \times_{M_1} L_{12}\) of \(\ol v(0)\) such that \(\pi_{02}|_U\colon U \to M_{02}\) is a smooth embedding, hence \(\pi_{02}(U) \subset M_{02}\) is a noncompact embedded Lagrangian.
Since \(v_0^\infty\) and \(v_2^\infty\) are continuous we find \(\eps > 0\) such that \(\ol v((-\eps, \eps))\) is contained in \(U\), which implies that \((v_0^\infty, v_2^\infty)((-\eps, \eps) \times \{0\})\) is contained in \(\pi_{02}(U)\).
The maps $v_0^\nu$ and $v_2^\nu$ have uniformly-bounded derivatives and converge $\cC^1_\loc$ to $v_0^\infty, v_2^\infty$ on the interior of their domains, hence $(v_0^\infty(s,-t),v_2^\infty(s,t))$ is in $W^{1,4}((-\eps,\eps)\times[0,\eps))$.
Standard elliptic regularity (e.g.\ \cite[Theorem B.4.1]{ms:jh}\footnote{
The hypothesis of \cite{ms:jh} that the Lagrangian submanifold is closed can be removed.}) applied to \((v_0^\infty(s,-t), v_2^\infty(s,t))\) now shows that \(v_0^\infty\) and \(v_2^\infty\) are \(\cC^\infty\) at \((0,0)\).
Since \(\pi_{02}|_U\) is a diffeomorphism onto its image, \(\ol v\) is \(\cC^\infty\) at \(0\) and thus we have shown that \(v_0^\infty, v_1^\infty, v_2^\infty\) are \(\cC^\infty\).

\begin{step2}
After passing to a further subsequence, the convergence of \((v_0^\nu(s,t-\delta^\nu))\), \((v_1^\nu|_{t = 0})\), \((v_2^\nu(s, t + \delta^\nu))\) takes place in \(\cC^k_\loc\).
\end{step2}

\noindent In order to establish \(\cC^k_\loc\) convergence near \((-\rho,\rho)\times\{0\}\), we cannot rely on \cite[Theorem B.4.2]{ms:jh}.
Rather, we will establish uniform Sobolev bounds for all three sequences of maps.
The compact Sobolev embeddings \(H^{k+2}\hookrightarrow\cC^k\) resp.\ \(H^{k+1} \hookrightarrow \cC^k\) for two-dimensional resp.\ one-dimensional domains will then provide \(\cC^k_\loc\)-convergent subsequences.

Set \(\bJ^\nu\) resp.\ \(\bj^\nu\) to be the coherent pair of almost complex structures resp.\ coherent collection of complex structures resulting from the transformations \eqref{eq:tripletofoldedACs} resp.\ \eqref{eq:tripletofoldedCs} applied to \(J_0^\nu,J_1^\nu,J_2^\nu\) resp.\ \(j^\nu\), and set \((w_{02}^\nu, \wh w^\nu)\) to be the \((\bJ^\nu,\bj^\nu)\)-holomorphic size-\((\delta^\nu,\rho)\) folded strip quilt resulting from the transformation \eqref{eq:fromtripletofolded} applied to \((v_0^\nu,v_1^\nu,v_2^\nu)\).
Then \(w_{02}^\nu\) resp.\ \(\wh w^\nu|_{t = 0}\) converge \(\cC^0_\loc\) to \(u_{02}(s,t) := (v_0^\infty(s,-t), v_2^\infty(s,t))\) resp.\ \(\ol u(s,t) := (v_0^\infty(s,0), v_2^\infty(s,0), v_1^\infty(s), v_1^\infty(s))\), where we have used the assumed \(\cC^1\)-bounds on \((v_0^\nu)\), \((v_2^\nu)\).
Since \((J_0^\nu)\) resp.\ \((J_1^\nu|_{t=0})\) resp.\ \((J_2^\nu)\) converge \(\cC^{k+1}\) to \(J_0^\infty\) resp.\ \(J_1^\infty\) resp.\ \(J_2^\infty\), and since \((J_0^\nu), (J_1^\nu), (J_2^\nu)\) are \(\cC^{k+2}\)-bounded, \((J_{02}^\nu)\) resp.\ \((\wh J^\nu|_{t=0})\) converge \(\cC^{k+1}\) to \(J_{02}^\infty\) resp.\ \(\wh J^\infty\); since \(j^\nu\) converges in \(\cC^\infty_\loc\) to the standard complex structure \(i\colon \tfrac \partial {\partial s} \mapsto \tfrac \partial {\partial t}, \tfrac \partial {\partial t} \mapsto -\tfrac \partial {\partial s}\), the components of \(\bj^\nu\) converge in \(\cC^\infty_\loc\) to the standard coherent collection \(\bi\) of complex structures, \begin{align} \label{eq:standardj}
\bi := ( (i,i), (i,i,i,i) ).
\end{align}

Fix \(\rho' \in (0, \rho)\) and choose \(\rho > \rho_1 > \rho_2 > \cdots > \rho_{k+2} = \rho'\).
Set \(u_{\delta^\nu}\) to be the restriction and extension to \(\bQ_{\delta^\nu,\rho_1}\) of \(u\) as defined in \eqref{eq:uextend}.
Due to the \(\cC^0_\loc\)-convergence of \(w^\nu_{02}\) resp.\ \(\wh w^\nu|_{t=0}\) to \(u_{02}\) resp.\ \(\ol u\) and the uniform \(\cC^1\)-bounds on \(\wh w^\nu\), we can express \(w^\nu_{02}\) resp.\ \(\wh w^\nu\) on \(Q_{02,\delta,\rho_1}\) resp.\ \(\wh Q_{\delta^\nu,\delta,\rho_1}\) for sufficiently large \(\nu\) in terms of the corrected exponential maps \(e_{u_{02,\delta^\nu}}\) resp.\ \(e_{\wh u_{\delta^\nu}}\) and sections \((\zeta_{02}^\nu,\wh\zeta^\nu)\in \Gamma^{k+1}_{u_{\delta^\nu}}\) as introduced in \S\ref{ss:estimate}:
\begin{align*}
w_{02}^\nu = e_{u_{02,\delta^\nu}}(\zeta_{02}), \qquad \wh w^\nu = e_{\wh u_{\delta^\nu}}(\wh\zeta).
\end{align*}
The sections \(\zeta_{02}^\nu, \wh\zeta^\nu\) converge to zero in \(\cC^0\) as \(\nu \to \infty\), are uniformly bounded in \(\cC^1\), and satisfy boundary conditions \eqref{eq:linearizedBCs} in the linearizations of \((L_{01} \times L_{12})^T\) and \(M_0 \times \Delta_{M_1} \times M_2\).

\medskip

\noindent {\bf Iteration claim.} {\it We bound \(\|\D^\nu_{\zeta^\nu}\zeta^\nu\|_{\wt H^l(\bQ_{\delta^\nu,\rho_l})}\) and \(\|\zeta^\nu\|_{\wt H^l(\bQ_{\delta^\nu,\rho_l})}\) for \(l \in [1, k+2]\) by induction on \(l\), where \(\wt H^l\) and \(\D^\nu\) are the modified Sobolev space and the linear delbar operator defined in \S\ref{ss:estimate} using \(\bJ^\nu\), \(\bj^\nu\), and the pair of connections \(\nabla = (\nabla_{02}, \wh \nabla)\) constructed in Lemma~\ref{lem:conn}.}

\medskip

\noindent The first key fact for this claim is the formula \begin{align} \label{eq:Dzetaisalmostzero}
\D^\nu_{\zeta^\nu}\zeta^\nu = \rd e_{u_{\delta^\nu}}(\zeta^\nu)^{-1}\bigl( \dbar_{\bJ^\nu,\bj^\nu}(e_{u_{\delta^\nu}}(\zeta^\nu)) - F^\nu(\zeta^\nu) \bigr) =: G^\nu(\zeta^\nu),
\end{align}
justified in \eqref{eq:dbarvsD}, where \(\dbar_{\bJ^\nu,\bj^\nu}\) is the nonlinear delbar operator defined in \eqref{eq:delbar}.
The relevant fact here is that \(G^\nu\) is a pair of smooth maps \begin{align*}
G_{02}^\nu\colon u_{02,\delta^\nu}^*\rT M_{02} \to u_{02,\delta^\nu}^*\rT M_{02}, \qquad \wh G^\nu\colon \wh u_{\delta^\nu}^*\rT M_{0211} \to \wh u_{\delta^\nu}^*\rT M_{0211}
\end{align*}
that preserve fibers but do not necessarily respect their linear structure.
Furthermore, for any \(k\), \(G^\nu\) is uniformly bounded in \(\cC^k\).
The second key fact is Lemma~\ref{lem:masterestimate}, which is a collection of \(\delta\)-independent elliptic estimates.

Since \(\zeta^\nu\) is uniformly bounded in \(\cC^1\), \(\|\zeta^\nu\|_{H^1(\bQ_{\delta^\nu,\rho_1})}\) and \(\|\D^\nu_{\zeta^\nu}\zeta^\nu\|_{H^1(\bQ_{\delta^\nu,\rho_1})} = \|G^\nu(\zeta^\nu)\|_{H^1(\bQ_{\delta^\nu,\rho_1})}\) are uniformly bounded.
This establishes the base case of the iteration.

Next, say that \(\zeta^\nu\) and \(\D^\nu_{\zeta^\nu}\zeta^\nu\) are uniformly bounded in \(\wt H^l(\bQ_{\delta^\nu,\rho_l})\) for some \(l \in [1,k+1]\).
Lemma~\ref{lem:masterestimate} yields: \begin{align} \label{eq:stripshrinkstep2afirstpart}
\|\zeta^\nu\|_{\wt H^{l+1}(\bQ_{\delta^\nu,\rho_{l+1}})} \leq C_1 \bigl( \|\D^\nu_{\zeta^\nu}\zeta^\nu\|_{\wt H^l(\bQ_{\delta^\nu,\rho_l})} + \|\zeta^\nu\|_{H^0(\bQ_{\delta^\nu,\rho_l})} \bigr).
\end{align}
It remains to bound \(\|\D^\nu_{\zeta^\nu}\zeta^\nu\|_{\wt H^{l+1}(\bQ_{\delta^\nu,\rho_{l+1}})}\).
Since \(\zeta^\nu\) is uniformly bounded in \(\wt H^{l+1}(\bQ_{\delta^\nu,\rho_{l+1}})\), it is uniformly bounded in \(\cC^{l-1}(\bQ_{\delta^\nu,\rho_{l+1}})\) by Lemma~\ref{lem:improvedsobolev}, which allows us to bound \(\|\D^\nu_{\zeta^\nu}\zeta^\nu\|_{\wt H^{l+1}(\bQ_{\delta^\nu,\rho_{l+1}})}\): \begin{align*}
\|\D^\nu_{\zeta^\nu}\zeta^\nu\|_{\wt H^{l+1}(\bQ_{\delta^\nu,\rho_{l+1}})} &\sr{\mathclap{\eqref{eq:Dzetaisalmostzero}}}{\leq} C_1 \Bigl( \sum_{\lambda_1,\ldots,\lambda_m \geq 1, \atop
\lambda_1 + \cdots + \lambda_m \leq l+1} \bigl\| |\nabla^{\lambda_1}\zeta^\nu| \cdots |\nabla^{\lambda_m}\zeta^\nu| \bigr\|_{H^0(\bQ_{\delta^\nu,\rho_{l+1}})} \\
&\hspace{1in} + \sum_{\lambda_1, \ldots, \lambda_m \geq 1, \atop
\lambda_1 + \cdots + \lambda_m} \bigl\||\nabla^{\lambda_1}\zeta^\nu|\cdots |\nabla^{\lambda_m}\zeta^\nu|\bigr\|_{\cC^0H^0(\bQ_{\delta^\nu,\rho_{l+1}})} \\
&\hspace{0.85in} + \sum_{\lambda_1\geq0,\lambda_2,\ldots,\lambda_m \geq 1, \atop
\lambda_1 + \cdots + \lambda_m \leq l-1} \bigl\||\nabla_s\nabla^{\lambda_1}\zeta^\nu| |\nabla^{\lambda_2}\zeta^\nu|\cdots |\nabla^{\lambda_m}\zeta^\nu| \bigr\|_{\cC^0H^0(\bQ_{\delta^\nu,\rho_{l+1}})} \Bigr) \nonumber \\
&\leq C_1 \Bigl( \|\zeta^\nu\|_{H^{l+1}(\bQ_{\delta^\nu,\rho_{l+1}})} + \sum_{m=0}^{l-1} \|\nabla^m\zeta^\nu\|_{\cC^0H^1(\bQ_{\delta^\nu,\rho_{l+1}})} + 1\Bigr) \\
&\leq C_1 \bigl(\|\zeta^\nu\|_{\wt H^{l+1}(\bQ_{\delta^\nu},\rho_{l+1})} + 1\bigr). \nonumber
\end{align*}
This, together with \eqref{eq:stripshrinkstep2afirstpart}, establishes the iteration step and completes the Iteration Claim.

\medskip

The uniform bounds on \(\|\zeta^\nu\|_{\wt H^{k+2}(\bQ_{\delta^\nu,\rho_{k+2}})}\) and the \(\cC^k\)-bounds that result from Lemma~\ref{lem:improvedsobolev} yield uniform bounds on \(\|w_{02}^\nu\|_{H^{k+2}(\bQ_{\delta^\nu,\rho_{k+2}})}\), \(\|\wh w^\nu\|_{H^{k+2}(\bQ_{\delta^\nu,\rho_{k+2}})}\), and \(\|\wh w^\nu|_{t=0}\|_{H^{k+1}((-\rho_{k+2},\rho_{k+2}))}\).
These bounds induce uniform bounds on the \(H^{k+2}\)-norms of \(v_0^\nu, v_2^\nu\) on the relevant subdomains of \((-\rho_{k+2},\rho_{k+2})^2\) and on the \(H^{k+1}\)-norms of \(v_1^\nu|_{(-\rho_{k+2},\rho_{k+2})\times\{0\}}\).
The compact embeddings \(H^{k+2} \hookrightarrow \cC^k\) resp.\ \(H^{k+1} \hookrightarrow \cC^k\) for two-dimensional resp.\ one-dimensional domains implies the desired \(\cC^k_\loc\)-convergence of \((v_0^\nu(s, t - \delta^\nu))\) resp.\ \((v_1^\nu(s, 0))\) resp.\ \((v_2^\nu(s,t+\delta^\nu))\) to \(v_0^\infty\) resp.\ \(v_1^\infty\) resp.\ \(v_2^\infty\).

\begin{step3}
We show that if for some \(\ell \in \{0,1,2\}\) and \(\kappa > 0\) the gradient satisfies a lower bound \(|\d v_\ell^\nu(0, \tau^\nu)| \geq \kappa\) for some \(\tau^\nu \to \tau^\infty \in (-\rho,\rho)\),
then at least one of \(v_0^\infty, v_2^\infty\) is nonconstant.
\end{step3}

In the notation of Step 2, it suffices to show that if for some \(\tau^\nu \to \tau^\infty \in [0,\rho)\) and \(\kappa > 0\) the inequality \(| \d w^\nu(0,\tau^\nu) | := | \d w_{02}^\nu(0,\tau^\nu) | + | \d \wh w^\nu(0,\tau^\nu) | \geq \kappa\) is satisfied, then \(u_{02}\) is not constant.
We prove the contrapositive of this statement: assuming that \(u_{02}\) is constant, we will show that the quantities \(\lim_{\nu\to\infty} \sup_{t\in [0,\rho)} | \d w_{02}^\nu(0,t) |\) and \(\lim_{\nu\to\infty} \sup_{t \in [0,\delta^\nu]} | \d \wh w^\nu(0,t) |\) are both zero.

Since the convergence of \((w_{02}^\nu)\) to \(u_{02}\) takes place in \(\cC^1_\loc\), the quantity \(\lim_{\nu\to\infty} \sup_{t \in [0,\rho)} |\d w_{02}^\nu(0,t)|\) is zero.
To see that the quantity \(\lim_{\nu\to\infty} \sup_{t \in [0,\delta^\nu]} |\d \wh w^\nu(0,t)|\) is also zero, note that by the last paragraph of Step 1, the limit \(\ol u\) of \((\wh w^\nu)\) is also constant, which implies the formula \(\d \wh w^\nu =  \d e_{\wh u_{\delta^\nu}}(\wh \zeta^\nu)(\nabla\wh\zeta^\nu)\).
It follows that to prove the equality \(\lim_{\nu\to\infty} \sup_{t\in[0,\delta^\nu]} | \wh\nabla\wh w^\nu(0,t)| = 0\), it suffices to prove the equality \(\lim_{\nu\to\infty} \sup_{t\in[0,\delta^\nu]} | \wh\nabla \wh \zeta^\nu(0,t)| = 0\).
We can now estimate, using the Sobolev inequality \(\|-\|_{\cC^0} \leq C_1\|-\|_{H^1}\) for one-dimensional domains whose lengths are bounded away from zero: \begin{align*}
\limsup_{\nu\to\infty} \sup_{t\in[0,\delta^\nu]} |\wh\nabla \wh \zeta^\nu(0,t)| &\leq \lim_{\nu\to\infty}\sup_{t \in [0,\delta^\nu]} |\wh\nabla\wh \zeta^\nu(0,0)| + \lim_{\nu\to\infty} \sup_{t \in [0,\delta^\nu]} |\wh\nabla\wh \zeta(0,t) - \wh\nabla\wh \zeta(0,0)| \\
&= \lim_{\nu\to\infty} \sup_{t\in[0,\delta^\nu]} |\wh\nabla \wh \zeta^\nu(0,t) - \wh\nabla \wh \zeta^\nu(0,0)| \\
&\leq \lim_{\nu\to\infty} C_1\tint_0^{\delta^\nu} |\wh\nabla_t\wh\nabla\wh\zeta^\nu(0,t)| \,\d t \\
&\leq \lim_{\nu\to\infty} C_1(\delta^\nu)^{1/2} \left(\tint_0^{\delta^\nu} |\wh\nabla_t\wh\nabla\wh\zeta^\nu(0,t)|^2 \,\d t\right)^{1/2} \\
&\sr{\mathclap{\text{Sobolev}}}{\leq} \lim_{\nu\to\infty} C_1(\delta^\nu)^{1/2} \|\wh \zeta\|_{H^3(\wh Q_{\delta^\nu,\rho})} = 0.
\end{align*}
This completes the contrapositive of Step 3, which concludes our proof of Theorem~\ref{thm:nonfoldedstripshrink}.
\end{proof}

\appendix

\section{Removal of singularity for cleanly intersecting Lagrangians} \label{app:remsing}

In this appendix, we sketch a proof of removal of singularity for a holomorphic curve satisfying a generalized Lagrangian boundary condition in an immersed Lagrangian with locally-clean self-intersection.
We emphasize that this is not a new result, see e.g.\ \cite{ca:sympsI, cel:switch, uf:thesis, is:reflection, fs:thesis}.
We have included the following proposition in this paper because our methods allow us to give a short proof.

This removal of singularity will be stated for maps \(u\) with Lagrangian boundary conditions lifting to paths \(\gamma,\gamma'\):
\begin{gather} \label{eq:diskremsingsetup}
u\colon (B(0, 1) \cap \H) \less \{0\} \to M, \qquad \gamma'\colon (-1, 0) \to L', \qquad \gamma\colon (0, 1) \to L, \\
\varphi'(\gamma'(s')) = u(s', 0), \qquad \varphi(\gamma(s)) = u(s, 0) \qquad \forall\: s' \in (-1,0), \: s \in (0,1), \nonumber \\
\partial_s u + J(s,t, u)\partial_t u = 0, \qquad E(u) := \tint u^*\om < \infty, \nonumber
\end{gather}
where \((M,\omega)\) is a closed symplectic manifold, \(\varphi\colon L \to M\) and \(\varphi'\colon L' \to M'\) are Lagrangian immersions with \(L,L'\) closed, and \(J\) is an almost complex structure \(J\colon B(0,1)\cap \H \to \J(M,\om)\).
We will assume that \(\varphi(L), \varphi'(L')\) intersect locally cleanly, which means that there are finite covers \(L = \bigcup_{i=1}^k U_i\), \(L' = \bigcup_{i = 1}^l U_i'\) such that \(\varphi\) resp.\ \(\varphi'\) restrict to an embedding on each \(U_i\) resp.\ \(U_i'\), and \(\varphi(U_i), \varphi'(U_j')\) intersect cleanly for all \(i, j\).

\begin{proposition} \label{prop remsing for L L'}
If \(u,\gamma,\gamma'\) satisfy \eqref{eq:diskremsingsetup}, then \(u\) extends continuously to \(0\).
\end{proposition}

\begin{proof}[Sketch proof of Proposition~\ref{prop remsing for L L'}]
The first part of the proof of \cite[Theorem~7.3.1]{ah} yields a uniform gradient bound on \(u\) in cylindrical coordinates near the puncture.
We must make a minor modification due to the fact that the Lagrangians defining our boundary conditions are immersed, not embedded: Recall that the uniform gradient bound in cylindrical coordinates is established in \cite{ah} by assuming that there is a sequence \(((s_k,t_k)) \subset (-\infty, 0] \times [0, \tfrac 1 2]\) so that \(\lim_{k \to \infty} | \d u(s_k, t_k) | = \infty\), which necessarily has \(s_k \to -\infty\).
Rescaling at the points \((s_k, t_k)\) yields a sequence of maps that converges in \(\cC^\infty_\loc\) to a nonconstant map on either \(\R^2\) or \(\pm\H\), which contradicts the finiteness of the energy.
To adapt this proof to our situation, let \(\delta\) be a Lebesgue number for \(L = \bigcup_{i = 1}^k U_i\) and \(L' = \bigcup_{i = 1}^l U_i'\).
That is, if \(A\) is a subset of \(L\) (resp.\ of \(L'\)) with \(\diam A \leq \delta\), then \(A \subset U_i\) (resp.\ \(A \subset U_i'\)) for some \(i\).
Now rescale at the points \((s_k, t_k)\) as in \cite{ah}, but restrict the resulting maps to the intersection of \(B(0, \tfrac 1 4\delta)\) with their domain.
The gradient bound on these rescaled maps and our choice of \(\delta\) allows us to pass to a subsequence so that for some \(i, j\), all the rescaled maps have boundary values in \(\pi(U_i)\) or \(\pi'(U_j')\).
A further subsequence converges in \(\cC^\infty_\loc\), so we get a contradiction and therefore a uniform bound on \(|\nabla u|\) in cylindrical coordinates.

The analogue of Lemma~\ref{lem:lengthsgotozero} holds in this setting; the proof is the same as for Lemma~\ref{lem:lengthsgotozero} but simpler.
As in the first paragraph, some care must be taken with the immersed Lagrangians.

The analogue of Lemma~\ref{lem:isoineq} holds in this setting, though the proof must be modified.
Specifically, the domains \(U_0, U_1, U_2, U_3\) used in the proof of that lemma must be replaced by the domain \(\ol B(0,1) \cap \H\).

A slight modification of the proof of Theorem~\ref{thm:remsing} establishes Proposition~\ref{prop remsing for L L'}.
\end{proof}

\end{document}